\documentclass[a4paper,11pt,reqno]{amsart}
\usepackage{amsmath,a4wide}
\usepackage{xfrac}
\usepackage{stmaryrd,mathrsfs,bm,amsthm,mathtools,yfonts,amssymb,color}
\usepackage{esint}
\usepackage{xcolor}
\usepackage[pagebackref,citecolor=black,linkcolor=blue,colorlinks=true]{hyperref}
\usepackage{tikz,tikz-3dplot}

\begingroup
\newtheorem{theorem}{Theorem}[section]
\newtheorem{lemma}[theorem]{Lemma}
\newtheorem{proposition}[theorem]{Proposition}
\newtheorem{corollary}[theorem]{Corollary}
\endgroup

\theoremstyle{definition}
\newtheorem{definition}[theorem]{Definition}
\newtheorem{remark}[theorem]{Remark}
\newtheorem{ipotesi}[theorem]{Assumption}

\numberwithin{equation}{section}

\newcommand{\N}{\mathbb{N}} %naturali
\newcommand{\R}{\mathbb{R}} %reali
\newcommand{\C}{\mathbb{C}} %complessi
\newcommand{\bC}{\mathbf{C}} %cilindri
\newcommand{\bB}{\mathbf{B}} %palle
\newcommand{\p}{\mathbf{p}} %orthogonal projections
\newcommand{\A}{\mathcal{A}} %Q-punti
\newcommand{\G}{\mathcal{G}} %distanza su A_Q
\newcommand{\bG}{\mathbf{G}} %grafico

\newcommand{\sep}{{\rm sep}}

\newcommand{\mass}{\mathbf{M}} %mass of a current
\newcommand\res{\mathop{\hbox{\vrule height 7pt width .3pt depth 0pt\vrule height .3pt width 5pt depth 0pt}}\nolimits}

\newcommand{\reg}{\mathrm{Reg}} %insieme regolare
\newcommand{\sing}{\mathrm{Sing}} %insieme singolare
\newcommand{\Sing}{\mathrm{Sing}} %insieme singolare

\newcommand{\bE}{\mathbf{E}} %eccesso riscalato
\newcommand{\bh}{\mathbf{h}} %height

\newcommand{\bI}{\mathbf{I}}

\newcommand{\Ha}{\mathcal{H}} %Hausdorff distance

\newcommand{\eps}{\varepsilon} %epsilon
\newcommand{\spt}{\mathrm{spt}} %support (usually of a measure)
\newcommand{\supp}{\mathrm{supp}} %support (usually of a function)
\newcommand{\dist}{\mathrm{dist}} %distance
\newcommand{\tr}{\mathrm{tr}} %trace
\newcommand{\diam}{\mathrm{diam}} %diameter
\newcommand{\Lip}{\mathrm{Lip}} %Lipschitz
\renewcommand{\epsilon}{\varepsilon}

\def\XXint#1#2#3{{\setbox0=\hbox{$#1{#2#3}{\int}$ }
		\vcenter{\hbox{$#2#3$ }}\kern-.6\wd0}}

\newcommand{\mint}{-\hskip-1.10em\int}

\DeclareMathOperator{\Div}{div} %divergence (conflict?)
\def\I#1{{\mathcal{A}}_{#1}}

\newcommand{\Iqs}{{\mathcal{A}}_Q(\R^{n})}
\newcommand{\Iq}{{\mathcal{A}}_Q}
\def\a#1{\left\llbracket{#1}\right\rrbracket}
\newcommand{\norm}[1]{\left\lVert#1\right\rVert} %norm
\newcommand{\etab}{\boldsymbol{\eta}}

\newcommand{\Iqn}{{\mathcal{A}_Q}}

\newcommand\cG{{\mathcal{G}}}

\title[Stationary $2$-valued graphs]{Dimension of the singular set for  $2$-valued stationary Lipschitz graphs}

\author[J. Hirsch]{Jonas Hirsch}
\address{Mathematisches Institut, Universit\"at Leipzig, Augustusplatz 10, D-04109 Leipzig, Germany}
\email{hirsch@math.uni-leipzig.de}
\author[L. Spolaor]{Luca Spolaor}
\address{Department of Mathematics, UC San Diego, AP\&M, La Jolla, California, 92093, USA}
\email{lspolaor@ucsd.edu}

\begin{document}

\maketitle

\begin{abstract}
We prove that the singular set of a $2$-valued Lipschitz graph that is stationary for the area is of codimension $1$.
\end{abstract}

\tableofcontents

\section{Introduction}
In his groundbreaking work \cite{All}, Allard proved that the singular set of stationary integral varifolds is meager. Since then little to no progress has been made on the question of the optimal dimension of the singular set for integral stationary varifolds. In this note we answer this question under two assumptions: multiplicity $2$ and Lipschitz graphicality. Moreover we do this by applying Almgren's strategy for the first time to the stationary setting, that is without any minimizing (nor stability) assumption.

Given a domain $\Omega \subset \R^m$, we will consider Lipschitz multiple valued functions $f\colon \Omega \subset \R^m \to \Iq(\R^n)$, and denote with $\Lip(f)$ their Lipschitz constant and with $\bG_f$ the integral current naturally associated to the graph of $f$ (see \cite{DS, DLS_Currents} for the relevant definitions).

\begin{definition}
    Given a function $f\colon \Omega \to \Iq(\R^n)$, we say that a point \emph{$x\in \Omega$ is regular} if there exists a neighborhood $B\subset \Omega$ of $x$ and $Q$ analytic functions $f_i \colon B \to \R^n$ such that 
\[
f(y) = \sum_{i=1}^Q \a{f_i(y)}\qquad \text{ for
almost every }y \in B\,,
\]
and either $f_i(x)\neq f_j(x)$ for every $x\in B$ or $f_i \equiv f_j$. The \emph{singular set $\Sing(f)$ of $f$} is the complement in $\Omega$ of the set of regular points. 
\end{definition}

Our main result is the following optimal dimensional bound on the singular set of stationary $2$-valued Lipschitz maps.

\begin{theorem}[Dimension of the singular set]\label{thm:main}
    Let $f\colon \Omega \to \I2(\R^n)$, with $\Omega\subset \R^m$ open, be a Lipschitz map such that $\bG_f$ is a stationary varifold. Then $\dim (\Sing(f)\cap \Omega)\leq m-1$ and all the points in $\reg(f)$ have either multiplicity $1$ or $2$. Furthermore, in the second case $\dim(\sing(f)\cap \Omega)\leq m-4$.
\end{theorem}

In codimension $1$ and under the additional assumption of stability of the regular part, works of Schoen-Simon, Wicramasekera, Minter and Minter-Wickramasekera, provide beautiful partial results \cite{ScSi,Wic0,Wic,MiWi,Min1,Min2}.

When the varifold is associated to an area minimizing current, then a celebrated result of Almgren \cite{Alm}, later revisited by De Lellis-Spadaro \cite{DS,DLS_Currents,DS1,DLS_Center,DLS_Blowup}, shows that the optimal dimension of the singular set is $(m-2)$. Recently De Lellis-Minter-Skorobogatova and Krummel-Wickramasekera \cite{DSk1,DSk2,DSk3,KWw1,KWw2}, proved that in fact such singular set is $(m-2)$ rectifiable. When the varifold is associated to an area minimizing current mod $p$, then work of De Lellis-Hirsch-Marchese-Stuvard \cite{DHMS} shows that the optimal dimension of the singular set is $(m-1)$, with a finer description achieved in codimension $1$ in work of De Lellis-Hirsch-Marchese-Spolaor-Stuvard \cite{DHMSS1,DHMSS2,DHMSS3}
combined with a result of Minter-Wickramasekera \cite{MiWi}. Our situation is somewhat more similar to this case, at least in the fact that for stationary varifolds the singular set can be of dimension $(m-1)$ and branch points can occur, however the minimizing assumption is used crucially in these works, while it's missing in the this note.   

For $C^{1,\alpha}$ multivalued maps, works of Simon and Wickramasekera \cite{SiWi} and Krummel and Wickramasekera \cite{KW1,KW2} investigate the size and the structure of the branching set. 

As a corollary to our paper and deep work of Minter \cite{Min1}, we can give a dimensional bound on the singular set of stationary stable varifolds  in codimension one, in the multiplicity $2$ class and with no triple junctions.

\begin{corollary}[Multiplicity $2$ branching set for codimension $1$ stationary stable varifolds with no triple junctions]
    Let $V\in \mathcal S_2$ as in \cite[Theorem B]{Min1}. Then $\dim(\mathcal B\cup \mathcal T\cup \mathcal C )\leq n-1$.
\end{corollary}

\subsection{Strategy and main new contributions}
 By standard arguments using monotonicity formula and dimension reduction, it is enough to understand the size of the branching set, that is the collection of points where at least one blow-up is a plane with multiplicity. To understand such set we use Almgren's approach \cite{Alm} in the revisited form of De Lellis and Spadaro \cite{DS1,DLS_Center,DLS_Currents}. In order to do that, except for minor technicalities, the main difficulties are: the construction of a small Lipschitz approximation to $\bG_f$ with errors that are superlinear in the excess, the development of a suitable linearization theory for stationary graphs and of  unique continuation and regularity theories for multivalued maps that arise through such linearization (in particular which are stationary, but not necessarily minimizing for the Dirichlet energy), and a suitable capacity argument to reach a contradiction at the linearized level.

To overcome these difficulties, the main new ingredients of our proof with respect to Almgren's approach are a
a higher integrability estimate for the Dirichlet energy of $f$, that allows to prove the existence of a superlinear small Lipschitz approximation (see Theorem \ref{thm:high}) and of a strong Dir-stationary approximation (see Corollary \ref{cor:harmapprox}) to $\bG_f$. This is where we use crucially the $Q=2$ assumption. Moreover, in order to have good compactness properties for stationary sequences, we introduce the notion of $\Iq$-generalized gradient Young measures and we study their regularity and unique continuation type properties under various assumptions of stationariety: this seems to be the correct linear problem in the stationary setting. Finally we revisit the capacity argument of \cite{DLS_Blowup}, replacing it with a weaker, but more general, argument that doesn't require any stronger regularity than Sobolev. This is needed, since we cannot guarantee that our final blow-up sequence converges to a strong solution, but only to a measure solution, as the higher integrability statement is not preserved when subtracting averages from multivalued functions.

\subsection{About our assumptions} We wish to make some remark on the three main assumptions of Theorem \ref{thm:main}: graphicality, multiplicity $2$ and stationariety of $\bG_f$.

\begin{remark}[Graphicality and multiplicity $2$] The multiplicity two  and graphicality are used only in the proof of the higher integrability Theorem \ref{thm:high}. In particular, we use crucially the fact that in the regions where there are no singular points of maximal densities, the multivalued graph splits into two single valued graphs.    
    
\end{remark}

\begin{remark}[On the stationariety of multivalued graphs]
Given a Lipschitz function $f\colon \Omega \to \Iqn$, we  consider the following quantities:
\[
g_{ij}(f):=\delta_{ij}+\partial_if\cdot \partial_jf\,,
\qquad g^{ij}(f):=(g_{ij}(f))^{-1}
\qquad \text{and}\qquad 
|g(f)|=\det(g_{ij}(f))\,,
\] 
and we define the
\emph{outer variation formula for area}
\begin{align}\label{eq:outer}
    \mathcal O^A(f, \psi):=&\int \sum_{l=1}^Q \sqrt{|g(f_l)|}\,g^{ij}(f_l)\, \partial_if_l^\alpha(x)\, \partial_j( \psi^\alpha(x,f_l(x))) \,dx
\end{align}
where $\psi\in C^\infty(\Omega_x\times \R^n_u;\R^n)$ with $\Omega_x$ compact and 
\[
|D_u\psi|\leq C<\infty\qquad \text{and}\qquad |\psi|+|D_x\psi|\leq C\, (1+|u|)\,,
\]
and the \emph{inner variation formula for area}
\begin{equation}\label{eq:inner}
    \mathcal I^A(f, \phi):= \int \left(\sum_{l=1}^Q \sqrt{|g(f_l)|} g^{ij}(f_l)\right)\, \partial_i\phi^j \,dx\,,\qquad \forall \phi \in C^\infty_c(\Omega, \R^m)\,,
\end{equation}
and we will say that $f$  is \emph{stationary with respect to the area functional} if 
    \[
\mathcal O^A(f,\psi)=0 \qquad \text{and}\qquad \mathcal I^A(f,\varphi)=0\,,
\]
for all admissible $\psi\in C^\infty(\Omega_x\times \R^n_u;\R^n)$, with $\Omega_x$ compact, and all $\varphi \in C^\infty_c(\Omega, \R^m)$.
Here and in the sequel we use Einstein's convention, meaning repeating indexes are implicitly summed. 
Clearly if $\bG_f$ is stationary, then $f$ is stationary with respect to area. If $Q=1$, then the converse is also true. If $Q=2$ and the codimension $n=1$ or the dimension $m=2$, then the two notions are in fact equivalent, as we show in Appendix \ref{app:statio}. If a weaker version of the Lawson and Osserman's conjecture was known to be true in every dimension, then for $Q=2$ the two notions would be always equivalent.

We remark that stationariety of the graph is needed to prove the monotonicity formula of the mass ratio which we use both for the usual stratification and in the proof of the higher integrability Theorem \ref{thm:high}. We do not know whether a function which is stationary for the area satisfies such monotonicity.
\end{remark}

\subsection{Organization of the paper} In Section \ref{ss:meas} we introduce the notion of $\Iq$-generalized gradient Young measures and in Section \ref{ss:meassol} we study the regularity and compactness properties of such measures under various stationariety assumptions. In Section \ref{ss:almapp} we prove the key higher integrability estimate and use it to obtain the results of \cite{DS1}. In Section \ref{ss:CmandNa} we point out the modification needed for the results of \cite{DLS_Center} to hold. Finally in Section \ref{ss:BU}, we modify the blow-up argument and the capacity argument of \cite{DLS_Blowup} to conclude the proof.

\subsection{Some notations} We will follow the notations of \cite{DS1,DLS_Currents,DLS_Center,DLS_Blowup,DS}, and we will recall some of it when needed. In particular, given $f\colon \Omega\to \Iq(\R^n)$ and $g\colon \Omega \to \R^n$, we will use the notations
\[
\eta\circ f(x):=\frac1Q\sum_{l=1}^Qf_l(x)\,,\qquad f\ominus g(x) =\sum_{l=1}^Q\a{f_l(x)-g(x)}\qquad \text{and}\qquad \overset{\circ}{f}(x):=f(x)\ominus Q(\eta\circ f(x))\,. 
\]

\subsection{Acknowledgments}  The second author is grateful for the support of the NSF Career Grant DMS-2044954.

\section{$\Iq$ generalized gradient Young measures}\label{ss:meas}

We are going to denote with 
\[
\mathbb V:=\underbrace{\R^n}_{=:\mathbb V_1}\times \underbrace{\R^{n\times m}}_{=:\mathbb V_2}\times \underbrace{\R^{(n\times m)^2}}_{=:\mathbb V_3}\,,
\]
with $\pi_i\colon \mathbb V\to \mathbb V_i$, $i=1,2,3$, the corresponding orthogonal projections, and with $v=(y,p,M)\in \mathbb V_1\times\mathbb V_2\times\mathbb V_3$ the corresponding variables. Moreover on the space of zero dimensional normal currents $\mathcal N^0(\mathbb V)$, we define the mass
\[
\mass^w(T)=\sup\left\{ T(\phi)\,:\,\phi\in C^\infty_c(\mathbb V)\,,\,\sup_{(y,p,M)\in \mathbb V}\frac{|\phi(y,p,M)|}{1+|y|^{2^*}+|p|^2+|M|}\leq1 \right\}\,,
\]
where $p^*$ denotes the Sobolev conjugate of $p$.

\begin{definition}[$\Iq$ generalized Young measures]\label{def:young} A couple $\mathcal F:=(F(x)\otimes dx,F^{\infty})$ is called \emph{$\Iq$ generalized Young measure} if it satisfies the following conditions:
\begin{enumerate}
    \item the map $F\colon \Omega \to \mathcal N^0(\mathbb V)$ is Lebesgue measurable;
    \item $F^\infty\in \mathcal N^0(\Omega \times \mathbb{S}_{\mathbb V_3})$, where $\mathbb{S}_{\mathbb V_3}:=\partial B_1^{\mathbb V_3}$ is the unit sphere in $\mathbb V_3$;
    \item $F(x)(\chi_{\mathbb V})=\mass(F(x))=Q$ for a.e. $x\in \Omega$, where $\chi_A$ denotes the characteristic function of the set $A$;
    \item $\int_\Omega\mass^w(F(x))\,dx+\mass(F^\infty)<\infty$.
\end{enumerate}
We will denote the space of such measures by $\mathcal Y^Q(\Omega, \mathbb V)$, or simply $\mathcal Y^Q$ when clear from the context. Moreover we shall denote by 
\[ 
\|\mathcal F\|_{\mathcal Y^Q}:=\int_\Omega\mass^w(F(x))\,dx+\mass(F^\infty)\,.
\]
Next we define the \emph{space of test functions $\mathcal{C}(\Omega\times \mathbb{V})$} as the set of functions $\phi \in C^0(\Omega\times \mathbb{V})$ such that $\spt(\phi) \in K \times \mathbb{V}$ for some $K \Subset \Omega$, and such that there is a function $\phi^\infty\in C^0(\Omega\times \mathbb {V}_3)$ which is $1$-homogeneous in $M$ and satisfies
\[
\lim_{|(y,p,M)|\to \infty}\frac{|\phi(x,y,p,M)-\phi^\infty(x,M)|}{1+|y|^{2^*}+|p|^2+|M|}=0 \qquad \forall x\in \Omega \,.
\]
We will call the function $\phi^\infty$ the \emph{recession function} of $\phi$.
An element $\mathcal F\in \mathcal Y^Q$ acts on $\mathcal{C}(\Omega\times \mathbb{V})$ by 
\[
\mathcal F(\phi):=( F\otimes dx)(\phi)+ F^\infty(\phi^\infty)=\int_\Omega F(x)(\phi)\,dx+ F^\infty(\phi^\infty)\,.
\]
Finally we will write $\mathcal F_k \rightharpoonup \mathcal F$ if $\lim_{k\to \infty} \mathcal F(\phi)=\mathcal F(\phi)$ for every $\phi \in \mathcal C(\Omega\times \mathbb V)$.
% for any $\phi$ such that there is a $1$-homogeneous function $\phi^\infty(x,M)$    is a  satisfying
% \[
% \lim_{|(y,p,M)|\to \infty}\frac{|\phi(x,y,p,M)-\phi^\infty(x,M)|}{1+|y|^{2*}+|p|^2+|M|}=0 \quad \forall x 
% \]
\end{definition}

In the definition above we can interpret $F^\infty$ as the concentration measure, while $F(x)$ is the oscillation measure. Moreover, analogously to usual Young measures, we have the following closure theorem.

\begin{proposition}[$\mathcal Y^Q$ is weakly closed]\label{le:weakclosed}
    Let $\mathcal F_k:=(F_k\otimes dx, F_k^\infty)\in \mathcal Y^Q$ be a sequence of $\Iq$ generalized Young measures such that
    \begin{gather}
        \|\mathcal F_k\|_{\mathcal Y^Q}\leq C<\infty\qquad \text{for all $k\in \N$}\,,
    \end{gather}
    then there exists $\mathcal F:=(F\otimes dx, F^\infty)\in \mathcal Y^Q(\Omega)$ such that the followings hold
    \begin{gather}
        \|\mathcal F\|_{\mathcal Y^Q} \leq  \liminf_{k\to \infty} \|\mathcal F_k\|_{\mathcal Y^Q}\,, \label{eq:lscYoung}\\
        \lim_{k\to \infty} \mathcal F_k(\phi)=\mathcal F(\phi)\qquad \forall \phi\in \mathcal{C}(\Omega\times \mathbb{V})\,.\label{eq:cYoung}
    \end{gather}
\end{proposition}

In the sequel we will be interested in generalized Young measures which arise as limits of sequences of $W^{1,2}$ functions, so we give the following:

\begin{definition}[Elementary measures] Let $f\in W^{1,2}(\Omega, \Iq(\R^n))$, then we define the \emph{elementary Young measure $\mathcal E_f$} associated to $f$ as the $\Iq$ generalized Young measure $\mathcal E_f=(E_f\otimes dx,0)$ where
\[
E_f(x)(\phi):=\sum_{l=1}^Q \phi(x,f_l(x),D f_l(x),Df_l(x)\otimes Df_l(x))\qquad \forall \phi \in \mathcal{C}(\Omega\times \mathbb{V})\,.
\]
\end{definition}

It is easy to check that
\[
\mass(E_f(x))=Q \qquad \text{and}\qquad \mass^w(E_f(x))\leq \sum_{l=1}^Q\left(|f_l(x)|^2+2|Df_l(x)|^2\right)\,,
\]
so that in particular
\begin{equation}\label{eq:elYoungbd}
\mathcal E_f\in \mathcal Y^Q
    \qquad \text{and}\qquad 
    \|f\|_{W^{1,2}(\Omega, \Iq(\R^n))}\leq \|\mathcal E_f\|_{\mathcal Y^Q}\leq 2 \, \|f\|_{W^{1,2}(\Omega, \Iq(\R^n))}\,.
\end{equation}

Finally we come to the definition of the class of objects we will be mostly interested in, that is the limits of elementary Young measures. In analogy with the classical theory we define them as follows:

\begin{definition}[$\Iq$ generalized gradient Young measures]\label{de:Aqgradmeasures}
    We say that $\mathcal F\in \mathcal Y^Q$ is a \emph{$\Iq$ generalized gradient Young measure}, and write $\mathcal F\in \operatorname{grad}\mathcal Y^Q$, if there exists a sequence of elementary Young measures $\mathcal E_{f_k}$ such that
    \[
    \lim_{k\to \infty}\mathcal E_{f_k}(\phi)=\mathcal F(\phi)\qquad \forall \phi \in \mathcal{C}(\Omega\times \mathbb{V})\,.
    \]
\end{definition}

The following are the main properties of generalized gradient Young measures.

\begin{proposition}[Elementary properties of gradient measures]\label{p:elementarylimit}
    Let $\mathcal F=(F(x)\otimes dx, F^\infty)\in \operatorname{grad}\mathcal Y^Q$, then there exists $f\in W^{1,2}(\Omega, \Iq(\R^n))$ and a family of probability measures 
    \[
    \{\nu_{x,y} \}_{(x,y)\in \Omega \times \mathbb V^1}\in \mathcal P(\mathbb V^2\times\mathbb V^3)
    \]
    such that
    \begin{gather}\label{eq.gradientmeasure01}
        F(x)(\phi)=\int \sum_{l=1}^Q\phi(f_l(x), p, M) \, d\nu_{x, f_l(x)}\,,\\
        \label{eq.gradientmeasure02}
        \int_{\mathbb{V}_2\times \mathbb{V}_3} p_i^\alpha \,d\nu_{x,f_l(x)}=\partial_if^\alpha_l(x)\qquad \forall l=1,\dots,Q\,,\text{ a.e. }x\in \Omega\,,\\\label{eq.gradientmeasure03}
        \int_{\mathbb{V}_2\times \mathbb{V}_3} M^{\alpha\beta}_{ij}\,d\nu_{x,f_l(x)} = \int_{\mathbb{V}_2\times \mathbb{V}_3} p^\alpha_ip_j^\beta\,d\nu_{x,f_l(x)} \geq \partial_if_l^\alpha(x) \,\partial_j f_l^\beta(x)\qquad\forall l=1,\dots,Q\,,\text{ a.e. }x\in \Omega\,,\\\label{eq.gradientmeasure04}
        \spt(F^\infty) \subset \Omega \times \left(\mathbb{S}_{\mathbb{V}^3} \cap \{ M \ge 0\}\right)\,,
        \end{gather}
        where the inequalities above should be understood in the sense of matrices. 
\end{proposition}

\subsection{Proof of Proposition \ref{le:weakclosed}}  We start by using convergence as currents to construct subsequential limits. Notice that $F_k\otimes dx \in \mathcal N^0(\Omega \times \mathbb V)$, for every $k$, since 
\[
 \mass(F_k\otimes dx)\leq \int_\Omega \mass(F_k(x))\,dx=Q\,|\Omega|\,.
\]
In particular we can apply the weak compactness for $\mathcal N^0(\Omega\times \mathbb V)$, that is there exists $\hat{F}\in \mathcal N^0(\Omega\times \mathbb V) $ such that
\begin{equation}\label{eq:currconv}
    \hat{F}(\varphi) =\lim_{k\to \infty}(F_k\otimes dx)(\varphi) \qquad \forall \varphi\in C^{\infty}_c(\Omega\times \mathbb V)\,.
\end{equation}
 Analogously we have that $\mass(F^\infty_k)\leq \|\mathcal F\|_{\mathcal Y^Q}\leq C<\infty$, and so again by compactness for normal currents there exists $\hat{F}^\infty\in \mathcal N^0(\Omega\times \mathbb S_{\mathbb{V}_3})$ such that
\[
\hat{F}^\infty(\varphi) =\lim_{k\to \infty}F_k^\infty(\varphi)\,, \qquad \forall \varphi\in C^{\infty}_c(\Omega\times \mathbb S_{\mathbb{V}_3})\,.
\]
Finally we consider the sequence
\[
\tilde{F}_k:=\eta_{\sharp} \left(\sqrt{1+|M|^2}\, F_k \otimes dx\right)\,,\qquad \text{where }\eta(M):=\frac{M}{\sqrt{1+|M|^2}}\,,
\]
that is
\begin{align*}
    \tilde{F}_k(\psi) 
        &=\int  F_k(x)\left(\psi\left(x, \frac{M}{\sqrt{1+|M|^2}}\right)\, \sqrt{1+|M|^2}\right) \,dx\\
        &\leq \|\psi\|_{\infty} \,\|\mathcal F_k\|_{\mathcal Y^Q} \leq C\,.
\end{align*}
This implies that $(\tilde{F}_k)_k$ is a uniformly bounded sequence in $\mathcal N^0(\Omega\times B_1^{\mathbb V_3})$ and so there exists a normal current $\tilde F\in (\Omega\times \overline{B_1}^{\mathbb V_3})$ such that up to subsequence
\[
\tilde{F}(\varphi) =\lim_{k\to \infty}\tilde{F}_k(\varphi) \qquad \forall \varphi\in C^{\infty}_c(\Omega\times \overline{B_1}^{\mathbb V_3})\,.
\]

Now we claim that there is a Lebesgue measurable map $F\colon \Omega \to \mathcal N^0(\mathbb V)$ such that
\begin{equation}\label{properties12}
    \hat{F}=F\otimes dx \qquad \text{and}\qquad Q=F(x)(\chi_{\mathbb V})\leq \mass(F(x))\leq Q\,,
\end{equation}
for almost every $x\in \Omega$, and moreover that \eqref{eq:cYoung} holds if we set $F^\infty=\tilde{F}^\infty+\hat{F}^\infty$, with $\tilde{F}^\infty:=\tilde{F}\res (\mathbb S_{\mathbb V_3})$. These two claims together will conclude the proof.

\medskip

\noindent \emph{Step 1:} We claim that for every $\phi \in C^\infty(\Omega\times \mathbb V)$ such that there exists $\psi \in C^\infty_c(\Omega)$ such that for every $\delta>0$ there exists $R_{\delta}>0$ with 
\[
|\phi(x,y,p,M)|\leq \delta \,\psi(x)\,\left( 1+|y|^{2^*}+|p|^2+|M| \right)\qquad \forall v=(y,p,M) \text{ such that }|v|>R_\delta\,,
\]
we have
\[
\hat{F}(\phi) =\lim_{k\to \infty} (F_k\otimes dx)(\phi)\,.
\]
Indeed, let $\delta>0$ and $\chi$ be a smooth function supported in $B_1^{\mathbb V}$ which is identically $1$ on $B_{1/2}^{\mathbb V}$. Let $\chi_\delta(x):=\chi(x/R_\delta)$, with $R_\delta$ as above, the we have
\begin{align*}
    \limsup_{k\to \infty} &\left|(F_k\otimes dx)(\phi)-\hat{F}(\phi)  \right|\\
        &\leq \limsup_{k\to \infty} \left|(F_k\otimes dx)(\chi_\delta\,\phi)- \hat{F}(\chi_\delta\,\phi)  \right|+2 \,\delta \, \sup_k\|\mathcal F_k\|_{\mathcal Y^Q}\\
        &\leq 0+ C\,\delta
\end{align*}
where in the last inequality we used \eqref{eq:currconv} and the lower semicontinuity of $\|\cdot\|_{\mathcal Y^Q}$ with respect to the convergence in \eqref{eq:currconv}.

Notice that, choosing $\phi(x,v)=\psi(x)$, which satisfies the above condition with $R_\delta=1/\delta$ for every $0<\delta<1$, implies  \eqref{properties12} by disintegration (see for instance \cite[Theorem 2.28]{AFP})

\medskip

\noindent \emph{Step 2:}
Finally we prove that \eqref{eq:cYoung} holds if we set $F_\infty=\tilde{F}^\infty+\hat{F}^\infty$, with $\tilde{F}^\infty:=\tilde{F}\res (\mathbb S_{\mathbb V_3})$.

To see this let $\phi(x,M)\in \mathcal C(\Omega\times \mathbb V)$ be as in Step 1, then the function 
$$\tilde{\phi}(x,M):=\sqrt{1+|M|^2}\,\,\phi\left(x,\frac{M}{\sqrt{1+|M|^2}}\right)$$ 
extends to an element of $C^0_C(\Omega\times \overline{B_1}^{\mathbb V_3})$ and so we have
\begin{equation}\label{eq:tilde1}
\tilde{F}(\tilde{\phi})=\lim_{k\to \infty} \tilde{F}_k(\tilde{\phi})=\lim_{k\to \infty}( F_k\otimes dx)(\phi)=(F\otimes dx)( \phi)\,.
\end{equation}
Now suppose $\phi\in \mathcal C(\Omega\times\mathbb V)$ is such that $\phi(x,v)=\phi(x,M)$ and $\phi(x, \lambda M)=\lambda \phi(x, M)$, then we have
\begin{equation}\label{eq:tilde3}
\lim_{k\to \infty}(F_k\otimes dx)(\phi) =\lim_{k\to \infty}\tilde{F}_k(\phi)= \tilde{F}(\phi)\,.
\end{equation}
Let $\chi_\eps(t)$ be a smooth function that is identically $1$ on $[0,1-\eps]$ and $0$ at $1$, then 
\[
 \tilde F(\phi) =  \tilde F\left( \chi_\eps(|M|)\phi\right)+ \tilde F\left( (1-\chi_\eps(|M|))\phi\right)\,.
\]
Notice that by construction
\[
\lim_{\eps \to 0}\tilde F\left((1-\chi_\eps(|M|))\phi\right)= \tilde{F}^\infty (\phi) \,,
\]
while by \eqref{eq:tilde1} we have as $\eps \to 0$ that
\[
 \tilde F( \chi_\eps(|M|)\phi)= (F\otimes dx)\left( \chi_\eps\left(\frac{|M|}{{\sqrt{1+|M|^2}}}\right)\phi\right)
\to   (F\otimes dx)(\phi)\,,
\]
so that
 \begin{equation}\label{eq:tilde2}
     \tilde{F}(\phi)=\tilde{F}^\infty(\phi)+(F\otimes dx)(\phi)\,.
 \end{equation}
 
Finally let $\phi\in \mathcal C(\Omega\times \mathbb V)$ and let $\phi^\infty(x, M)$ be the corresponding $1$-homogeneous function in $M$, and set $\psi(x,v)=\phi(x,v)-\phi^\infty(x, M)$. Then $\psi$ satisfy the assumption of Step 1 and so
\[
\lim_{k\to \infty}  (F_k\otimes dx) (\psi)=(F\otimes dx)(\psi)\,.
\]
On the other hand $\phi^\infty$ satisfies \eqref{eq:tilde3} and \eqref{eq:tilde2}, that is 
\[
\lim_{k\to \infty} (F_k\otimes dx)(\phi^\infty)= \tilde{F}^\infty(\phi^\infty)+(F\otimes dx)(\phi^\infty)\,.
\]
Combining the last equalities proves the claim.
\qed

\subsection{Proof of Proposition \ref{p:elementarylimit}} We divide the proof in several parts.

\medskip 
\noindent \emph{Proof of \eqref{eq.gradientmeasure01}: } 
Let us fix a generating sequence $f_k \in W^{1,2}(\Omega, \Iqs)$, i.e. $\mathcal{E}_{f_k} \rightharpoonup \mathcal{F}$. In particular the uniform boundedness principle implies that 
\[ \limsup_{k} \norm{f_k}_{W^{1,2}(\Omega,\Iqs)} \le C\,.\]
Hence Sobolev embedding, \cite[Proposition 2.11]{DS} provides a function $f \in W^{1,2}(\Omega, \Iqs)$ such that, up to a subsequence, $\lim_{k} \norm{\G(f_k,f)}_{L^p(\Omega)}=0$ for all $p<2^*$ and $\norm{Df}_{L^2(\Omega)} \le \liminf_k \norm{Df_k}_{L^2(\Omega)}$.

%To deduce \eqref{eq.gradientmeasure01} we argue similar as above. Let $\phi \in C_c(\Omega \times \R^n)$ and $\chi \in C^\infty_c(\mathbb{V}_3)$ such that $\chi=1$ on $B_1^{\mathbb{V}_3}$. 
Let $\phi(x,y) \in C_c^0(\Omega \times \R^n)$, one has $|E_{f_k}(x)(\phi)-E_{f}(x)(\phi)|\le \norm{D_y\phi}_\infty \G(f_k(x),f(x)) $, hence \eqref{eq:cYoung} implies that 
\[\mathcal{F}(\phi) = (F\otimes dx)(\phi) = \lim_{k} \mathcal{E}_{f_k}(\phi) = \mathcal{E}_f(\phi)\,,\]
or in other words $(\pi_0\otimes \pi_1)_{\sharp} \mathcal{F} = \sum_{l=1}^Q \a{f_l(x)} \otimes dx$. 
Now we can apply the classical disintegration theorem to deduce \eqref{eq.gradientmeasure01}.

\medskip

\noindent\emph{Proof of \eqref{eq.gradientmeasure02}: } We divide the proof in steps.

\noindent \emph{Step 1:} We claim the following. Given 
\begin{enumerate} 
    \item a sequence of set $E_k \subset E$ with $|E\setminus E_k| \to 0$;
    \item two uniformly bounded sequences $f_k, g_k \in W^{1,2}(\Omega, \Iqs)$, with $\mathcal{E}_{f_k} \rightharpoonup F\otimes dx + F^\infty $ and $\mathcal{E}_{g_k} \rightharpoonup G\otimes dx + G^\infty $ and such that $f_k=g_k$ on $E_k$,
\end{enumerate} 
then we have $F(x)=G(x)$ a.e. $x\in E$. 

Firstly observe that as a result of the approximate differentiability of $Q$-valued functions, \cite[Corollary 2.7]{DS} we have that 
\[ 
T_{x_0}f_k(x)= T_{x_0}g_k(x) \qquad \text{for all $x$ and a.e. $x_0 \in E_k$}\,,
\]
where we have used the notation of \cite[Definiton 1.9]{DS} for the first-order approximations of $f,g$ at $x_0$. 
In particular this implies that for a.e. $x \in E_k$ we have $E_{f_k}(x)(\phi)=E_{g_k}(x)(\phi)$. Now given any $\phi$ such that for some $\epsilon >0$ 
\[A(x) = \sup_{y,p,M} \frac{|\phi(x,y,p,M)|}{(1+|y|^{2^*} + |p|^2 +|M|)^{1-\epsilon}}\] 
is uniformly bounded, we deduce that 
\[ \mathbf{1}_{E\setminus E_k} \left(|E_{f_k}(x)(\phi)|+|E_{g_k}(x)(\phi)| \right) \le C \mathbf{1}_{E\setminus E_k} A(x) (1+ |f_k(x)|^{2^*} +|Df_k|^2 + |g_k(x)|^{2^*} +|Dg_k|^2)^{1-\epsilon}\,. \]
This implies that for $k \to \infty$ 
\[ \int_{E\setminus E_k} \left(|E_{f_k}(x)(\phi)|+|E_{g_k}(x)(\phi)| \right) \le \norm{A}_{L^\infty} \left(\norm{f_k}_{W^{1,2}(\Omega)}+\norm{g_k}_{W^{1,2}(\Omega)}\right)^{1-\epsilon}|E\setminus E_k|^\epsilon \to 0\,.\]
In particular, we conclude that for every compactly supported $\phi$ we have 
\[ \int_E F(x)(\phi) \, dx = \int_E G(x)(\phi) \, dx\,,\]
which given the arbitrariness of $\phi$ gives the claim .

\noindent \emph{Step 2:}
We claim that for a.e. $x \in E=\{ f = Q \a{\eta \circ f}\}$ we have 
\begin{equation}\label{eq.gradientoncollapsed}
    \int_{\mathbb{V}} p_i^\alpha \,d\nu_{x,\eta\circ f} = \partial_i(\eta\circ f)^\alpha.
\end{equation}
Firstly we note that $\eta \circ f_k \rightharpoonup \eta \circ f$ in $W^{1,2}(\Omega, \R^n)$, hence we deduce that for any $\varphi \in C_c^0(\Omega)$, the function $\phi(x,p)=\varphi(x)p^\alpha_i$ has $\phi^\infty=0$, so that 
\[ \mathcal{F}(\phi)=(F\otimes dx) (\phi)=\lim_{k\to \infty}\mathcal{E}_{f_k} (\phi) = \lim_{k\to \infty} Q \int_{\Omega} \varphi(x) \partial_i (\eta \circ f_k)^\alpha  = Q \int_{\Omega} \varphi(x) \partial_i (\eta \circ f)^\alpha\,.\]
By the arbitrariness of $\varphi$, it follows that for a.e. $x \in \Omega$ 
\begin{equation}\label{eq:average} Q \,\partial_i(\eta \circ f)^\alpha= \sum_{l=1}^Q \int_{\mathbb{V}} p^\alpha_i \,d\nu_{x,f_l(x)} \end{equation}
Now we consider any $\psi(x,y) \in C^0_c(\Omega \times \R^n)$ and the associated $\phi(x,y,p)=\psi(x,y)p_i^\alpha$. We define the auxiliary function $\tilde{\phi}(x,y,p)=\psi(x,\eta \circ f(x)) p_i^\alpha$. Furthermore given any $E\subset U$ open and $\chi\in C^\infty_c(U)$ such that $\mathbf{1}_E \le \chi \le \mathbf{1}_{U}$, we note that for $k \to \infty$, since $f=Q\a{\eta \circ f}$ on $E$, 
\begin{align*}
    &\left|\mathcal{E}_{f_k}(\chi \phi) - \mathcal{E}_{f_k}(\chi \tilde{\phi})\right|
    \le \int_{E} |D_y\psi(x,y)| \, \G(f_k, f) \, |Df_k| + \int_{U\setminus E} \norm{\psi}_\infty |Df_k|\\
   & \le \norm{\psi}_{C^1}\,\norm{Df_k}_{L^2(\Omega)}(\norm{\G(f_k,f)}_{L^2} + |U\setminus E|^\frac12) \to C \norm{\psi}_{C^1}|U\setminus E|^\frac12\,.
\end{align*}
Passing to the limit, and using \eqref{eq:average} we obtain
\[
\int \chi \psi(x,y) \int_{\mathcal V}p^\alpha_i \,d\nu_{x,f_l(x)} dx =F\otimes dx(\chi\phi)=F\otimes dx(\chi\tilde \phi)=Q \int \chi \psi(x,\eta\circ f) \partial_i(\eta\circ f)^\alpha\,.
\]
Approximating $E$ by open sets $U$, we conclude the claim.

\noindent \emph{Step 3:} Now we can conclude \eqref{eq.gradientmeasure02} by induction on $Q$. 
The case $Q=1$ is just a statement on the weak convergence of Sobolev functions. 
Hence we may assume it holds for all $Q'<Q$. Due to Step 2, \eqref{eq.gradientmeasure02} holds on $E_0=\{ f = Q \a{\eta\circ f}\}$. Fix any $S \in \Iqs$ with $\sep(S)>0$. Hence there is $\epsilon_S>0$, $Q=Q_1+Q_2$ with $Q_i \ge 1$ and Lipschitz retractions $\chi_i\colon \Iqs \to \A_{Q_i}(\R^n)$ such that 
\[ 
T= \chi_1(T) + \chi_2(T) \qquad \forall \G(T,S) \le 2 \epsilon_S\,.
\] 
We consider $E=\{ x \colon \G(f(x),S)<\epsilon_S \}$ and we define the sequences 
\[ f^i_k= \chi_i \circ f_k \in W^{1,2}(\Omega, \A_{Q_i}(\R^n) )\]
and $E_k= \{ x \in E \colon \G(f_k(x), S) < \epsilon_S \}$. We clearly have $|E\setminus E_k| \to 0$. Furthermore we have that for a.e. $x \in E_k$ 
\[ E_{f_k^1}(x) + E_{f_k^2}(x) = E_{f_k}(x)\,,\]
by the almost everywhere differentiability of $W^{1,2}$ multivalued functions. We may assume that for $i =1,2$ we have $\mathcal{E}_{f_k^i} \rightharpoonup F^i\otimes dx +(F^i)^\infty$. Now due to Step 1 we have for a.e. $x \in E$
\[ F(x)\otimes dx = F^1(x)\otimes dx + F^2(x) \otimes dx\,. \]
Hence we can apply the inductive hypothesis to $f_k^i$ for $i=1,2$ separately and deduce that \eqref{eq.gradientmeasure02} holds a.e. on $E$. Choosing a dense family in $\Iqs \setminus \{ Q \a{p} \colon p \in \R^n \}$ we cover a.e. all of $\Omega\setminus E_0$ by sets of the above type, which concludes the proof.

\medskip

\noindent\emph{Proof \eqref{eq.gradientmeasure03}: } 
The equality part can be seen as follows. Fix $\chi \in C^\infty_c(\mathbb{V}_3)$ with $\chi = 1$ on $B_1^{\mathbb{V}_3}$ and let $\psi \in C^\infty_c(\Omega \times \R^n)$ be arbitrary. Now consider $\phi(x,y,M) = \psi(x,y) \chi(\frac{M}{R}) M^{\alpha \beta}_{ij}$ and $\tilde{\phi}(x,y,p) = \psi(x,y) \chi(\frac{p^tp}{R}) p^\alpha_i p^\beta_j$ then 
\begin{align*}
    &\int \sum_{l=1}^Q \psi(x,f_l(x)) \int \chi( \frac{M}{R}) M^{\alpha\beta}_{ij} \, d\nu_{x,f_l(x)} \; dx = (F\otimes dx)(\phi)=\lim_{k \to \infty} \mathcal{E}_{f_k}(\phi)\\
    &=\lim_{k \to \infty} \mathcal{E}_{f_k}(\tilde{\phi}) = (F\otimes dx)(\tilde{\phi})= \int \sum_{l=1}^Q \psi(x,f_l(x)) \int \chi\left( \frac{p^tp}{R}\right) p^\alpha_i p^\beta_j \, d\nu_{x,f_l(x)}\;dx\,.
\end{align*}
Taking the limit $R\to \infty$ provides the equality part. 

The inequality now follows from Jensen's inequality applied to the family of convex function $G(p)=p^\alpha_ip^\beta_j \xi^i\xi^j \eta_\alpha \eta_\beta$ for $\xi \in \R^m, \eta \in \R^n$ i.e. 
\begin{align*}
    \xi^i\xi^j \eta_\alpha \eta_\beta \int M^{\alpha\beta}_{ij} \, d\nu_{x,f_l(x)} &= \xi^i\xi^j \eta_\alpha \eta_\beta \int p_i^\alpha p_j^\beta \, d\nu_{x,f_l(x)} = \int G(p) \, d\nu_{x,f_l(x)}\\
    &\ge G\left(\int p \, d\nu_{x,f_l(x)}\right) = (\eta\cdot \partial_\xi f_l(x))^2\,.
\end{align*}

\medskip

\noindent\emph{Proof \eqref{eq.gradientmeasure04}: } Let $M_0 \in \mathbb{V}^3$ with $\dist(M_0,\{M\ge 0\})=2\delta>0$ be given. Consider a smooth function $\varphi$ compactly supported in $B_{\delta}(M_0)$ and its $1$-homogeneous modification $\varphi^\infty(M)=|M|\varphi\left(\frac{M}{|M|}\right)$. Furthermore we consider for any $\eta \in C^\infty_c(\Omega)$ and any $R>0$ the test function $\phi_R(x,M) = \eta(x)\rho(\frac{|M|}{R}) \varphi^\infty(M)$ where $\rho$ is non-decreasing function with $\rho(t)=0$ for $t<\frac12$ and $\rho(t)=1$ for $t\ge 1$. We note that $\phi^\infty=\eta(x) \varphi^\infty$ is the associated recession function to $\phi_R$ for all $R$. Hence we can deduce that 
\begin{align*}
    F^\infty(\phi^\infty)= \lim_{R\to \infty} \mathcal{F}(\phi_R) = \lim_{R\to \infty} \lim_{k \to \infty} \mathcal{E}_{f_k}(\phi_R) =0\,.
\end{align*}
This proves the claim.

\qed

% Finally we will be mostly interested in the following types of Young measures.

\section{Linear theory for measure solutions}\label{ss:meassol}

We introduce the following notions of stationary measures for the Dirichlet energy.

\begin{definition}[Stationary measures for energy]
    An $\Iq$-generalized gradient Young measure $\mathcal F:=(F\otimes dx,F^\infty)\in \operatorname{grad}\mathcal Y^Q$ is a \emph{inner and outer measure solution in $\Omega$} if 
    \begin{gather}
        \mathcal O(\mathcal F, \varphi):=( F\otimes dx)(p_i^\alpha y^\alpha \partial_i\varphi)+ \mathcal F(M_{ii}^{\alpha\alpha}\varphi)=0 \quad \forall \varphi \in C_c^\infty(\Omega) \label{eq:outervar}\\
        \mathcal I(\mathcal F, \phi):=\mathcal F\left((2 M_{ij}^{\alpha\alpha}-M_{kk}^{\alpha\alpha} \delta_{ij}) \partial_i\phi^j\right) =0 \quad \forall \phi\in C^{\infty}_c(\Omega,\R^m)\label{eq:innervar}\\
        \mathcal F (p_i^\alpha\,\partial_i\psi^\alpha)=0 \quad \forall \psi\in C^\infty_c(\Omega,\R^n)\label{eq:averageharmonic}\,.
    \end{gather}
    In the sequel we will refer to $\mathcal O$ as outer variation and $\mathcal I$ as inner variation. Moreover we will say that $\mathcal{F}$  satisfies \emph{the strong outer variations in $\Omega$} if $F^\infty = 0$\footnote{Let us shortly elaborate on the condition $F^\infty=0$. The main issue is that otherwise the needed test function $\phi(x,y,p,M)=M_{ii}^{\alpha \beta} \partial_{y^\beta}\varphi^\alpha(x,y)$ does not admit a $1$-homogeneous recession function $\phi^\infty(x,M)$. Nonetheless we want to emphasis that if $F^\infty=0$ then the function $\phi$ can be approximate by $\phi_R(x,y,p,M) = \chi(\frac{M}{R}) \phi(x,y,p,M)$ where $\chi \in C^\infty_c(\mathbb{V}^3)$ and $1$ on $B_1^{\mathbb{V}_3}$.} and 
\begin{equation}\label{eq:strongouter}
        \mathcal{S}(\mathcal{F}, \varphi) = (F\otimes dx)\left( p_i^\alpha \partial_i\varphi^\alpha(x,y) + M_{ii}^{\alpha \beta} \partial_{y^\beta}\varphi^\alpha(x,y) \right) \quad \forall \varphi \in C^1_c(\Omega \times \R^n, \R^n)\footnote{by approximation with $\chi(\frac{y}{R}) \varphi(x,y)$ we can consider all the test functions $\varphi(x,y)$ that had been considered in \cite{DS}.}\,.
    \end{equation}
    % for all $\varphi \in C^1(\Omega \times \R^n, \R^n)$ such that $\spt\varphi \subset K\times \R^n$ with $K \Subset \Omega$.
    We will call \emph{stationary measures} the measure solutions  that satisfy the strong outer variations, and we will say that $f\in W^{1,2}(\Omega, \Iq(\R^n))$ is a \emph{classical solution} if $\mathcal E_f$ is a stationary measure, in particular \eqref{eq.gradientmeasure03} holds with equality.
\end{definition}
 
In this section we will prove two main results: a unique continuation/regularity type result for inner and outer measure solutions and a compactness theorem for uniformly higher integrable almost classical solutions. While the first result follows from the monotonicity of a suitably defined frequency function, which requires only inner and outer variations to be zero, the second result requires a stronger assumption since we will need more general test functions (i.e., projections) in the outer variation.

\begin{theorem}[Unique continuation of inner and outer measure solutions]\label{thm:UCformeas}
 Let $\mathcal F\in \operatorname{grad}\mathcal Y^Q$ be a inner and outer measure solution, then either
 \[
 \dim(\{x\in \Omega\,:\,f(x)=Q\a{\etab \circ f(x)}\})\leq m-1 
 \]
 or $f\equiv Q\a{\etab \circ f}$ in $\Omega$.
\end{theorem}

To state the compactness result in its most general form, which will be needed later, we first fix some notation. The argument is based on a concentration compactness  argument applied to a sequence of $Q$-points $T_k \in \Iq(\R^n)$. Associated to such a sequence we can find a sequence of points $p^k_1, \dots, p^k_N \in \R^n$ and a sequence $s_k \to \infty$ such that 
\begin{enumerate}
    \item $|p_k^i-p_k^j|>4s_k$ for all $i\neq j$;
    \item\label{finitediam2} $T_k=\sum_j^N T_k^j$ with $T_k^j \in \A_{Q_j}(\R^n)$, $\limsup_k|T_k^j\ominus p_k^j|<\infty$.
\end{enumerate}
Note that (2) implies that $\G(T_k, P_k)< s_k $, $P_k = \sum_{j=1}^N Q_j \a{p_k^j}$ for $k$ sufficient large, which we will assume from now on. In case the diameter of the $T_k$ is bounded we may choose any sequence $p_k \in B_{2 \diam(T_k)}(\eta \circ T_k)\subset \R^n$ and set $P_k=Q\a{p_k}$.

Furthermore the splitting of the $T_k$ introduces Lipschitz retractions, compare \cite[Lemma 3.7]{DS} \[\chi_k^j \colon \Iqs \to B_{2s_j}\left(Q_j\a{p_k^j}\right)\subset \A_{Q_j}(\R^n)\,.
\]
In the bounded case note that the retractions reduce to the identity map and we can drop the extension by the artificial dimension as in the sequel.

With their help we define the following sequence of ``projection map'' (the inner variation prevents us from considering each projection on its own): 
\begin{equation}\label{eq.projectionmap}
    \chi_k(T)=\sum_{j=1}^N \chi_k^j(T)\ominus(p_k^j - j e_{0})\,,
\end{equation}
where we have implicitly used the identification of $\R^n$ with $\{ x_{0} =0 \} \subset \R^{n+1}$. 
%On $\A_{Q}(\R^n)$ we may similarly define projections $\hat{\chi}_j \colon \A_{Q}(\R^{n+1}) \to \A_{Q_j}(\R^{n+1})$ associated to the retraction onto $\sum_j A_{Q_j}( B_{\frac14}(j) \times \R^n)$ given by $\hat{\chi}=\sum_j \hat{\chi}_j$

Now we define an ``almost inverse''to $\chi_k$ by
\begin{equation*}\label{eq.recoverymap}
    \chi^{+}_k(y,y_{n+1})= \sum_{j=1}^N y+ \sigma_j(y_{n+1}) p_k^j\,,
\end{equation*}
where $(\sigma_j)_j$ is a partition of unity subordinate to the intervals $(j-2/3,j+2/3)$.
It behaves like an almost left inverse since 
\[ \chi^+_k(\chi_k(T))=\sum_{j=1}^N \chi_k^j(T) \]
and the right hand side agrees with $T$ if $\G(T,P_k)< 2s_k$.
Concerning the other direction we have for $S\in \Iq(\R^{n+1})$ that
\[ \chi_k(\chi^+_k(S)) = S \]
if $\pi_0(S) = \sum_{j=1}^N Q_j \a{j}$ and $|S|<s_k$, where $\pi_0$ is the orthogonal projection onto $\R \times \{0\}$\,.

\begin{theorem}[Compactness for higher integrable classical solutions]\label{thm:compactness} Let $(f_k)_k$ be a sequence satisfying the following assumptions
\begin{itemize}
    \item[c1)]\label{ass.compactness1} uniform energy bounds, i.e. 
    \[
\limsup_{k\to\infty}\int_{B_4}|Df_k|^2\leq C<\infty\,,
\]
    \item[c2)]\label{ass.compactness2} uniform higher integrability, that is there exists $p>1$ such that 
    \[
\left(\mint_{B_s(x)}|Df_k|^{2p}\right)^{\frac1p}\leq C\, \mint_{B_{2s}(x)} |Df_k|^2\,,\qquad\text{ for every $B_{2s}(x)\in B_4$}\,,
\]
    \item[c3)]\label{ass.compactness3} almost stationariety, that is 
    \[
     \lim_{k\to \infty}\mathcal S(\mathcal E_{f_k},\varphi)=\lim_{k\to \infty}\mathcal I(\mathcal E_{f_k},\phi)=\lim_{k\to \infty} \mathcal E_{f_k}(p_i^\alpha\,\partial_i\psi^\alpha)=0\,,
    \]
    for every admissible $\varphi,\phi, \psi$ as in \eqref{eq:outervar}, \eqref{eq:innervar} and \eqref{eq:averageharmonic} respectively,
\end{itemize}
then there exists a classical solution $f\in W^{1,2}(B_3, \Iqn)$, satisfying c2) in $B_2$ and such that for each $r<3$ one has
\begin{equation}\label{eq.strongconvergence}
   \lim_k \norm{\G(f_k, \chi_k^+\circ f)}_{L^2(B_r)}+\norm{|Df_k|-|Df|}_{L^2(B_r)}=0\,. 
\end{equation}

If in addition to c1) one has a uniform $L^2$-bound, i.e. $\limsup_{k \to \infty} \int_{B_2} |f_k|^2 < \infty$, the map $\chi_k= \operatorname{Id}$ for all $k$. In particular the sequence is compact. 
\end{theorem}

In the course of proving Theorem \ref{thm:compactness} we will also establish the following two regularity results for measure solutions that satisfy the stronger outer variation and are either higher integrable or continuous. We state them separately as we believe that they could be of independent interests.

\begin{proposition}[Continuity of higher integrable stationary measures]\label{prop:higofmeas}
    Let $\mathcal F\in \operatorname{grad}\mathcal Y^Q$ be a stationary measure in $B_2$, and suppose that there exists $p>1$ such that for all $x_0\in B_1, r<1$
    \begin{equation}
        \mathcal F\left(r^{-m} \varphi\left(\frac{|x-x_0|}{r} \right) |M|^{p}\right)^{\frac1{p}}\lesssim \mathcal F\left(r^{-m} \varphi\left(\frac{|x-x_0|}r \right) |M|\right)\,,
    \end{equation}
    where $\varphi(t)$ is identically $1$ in $[0,1]$ and $0$ for $t>2$.
    Then the associated map $f\in W^{1,2}(B_2, \Iq(\R^n))$ as in Proposition \ref{p:elementarylimit} is H\" older continuous in $B_1$.
\end{proposition}

While to prove continuity in the previous proposition we require both inner variation and strong outer variations to be zero, in the following proposition we do not require anything on the inner variation. This is consistent with the results in \cite{HS}.

\begin{proposition}[Regularity of continuous strongly outer stationary measures]\label{prop:regformeas}
    Let $\mathcal F\in \operatorname{grad}\mathcal Y^Q$ satisfy the strong outer variation and let $f$ be the associated $W^{1,2}(\Omega, \Iq(\R^n))$ map as in Proposition \ref{p:elementarylimit}. If $f$ is continuous then $f$ is a classical solution and  $\mathcal F=\mathcal E_f$, in particular \eqref{eq.gradientmeasure03} holds with equality.
\end{proposition}

\subsection{Basic properties of inner and outer measure solutions}

In this subsection we collect some basic properties of inner and outer measure solutions: compactness, harmonicity of the average, classification of $1$-dimensional solutions and monotonicity of an appropriate frequency function, defined as follows:
\begin{align}\label{eq.definitionD}
    D_{\mathcal F}(x_0,r)&=r^{2-m} \mathcal{F}\left(\varphi\left(\frac{|x-x_0|}{r}\right) M^{\alpha\alpha}_{ii}\right)\,,\\\label{eq.definitionH}
    H_{\mathcal F}(x_0,r)&= r^{-m} (F\otimes dx)\left( \psi\left(\frac{|x-x_0|}{r}\right) |y|^2\right) = \frac1{r^m}\int \psi\left(\frac{|x-x_0|}{r}\right) |f|^2 \,dx \,,\\\label{eq.definitionI}
    I_{\mathcal F}(x_0,r)&=\frac{D_{\mathcal F}(x_0,r)}{H_{\mathcal F}(x_0,r)}\,,
\end{align}
where $\varphi(t) \in C^\infty_c((-1,1))$ is a fixed non-increasing function satisfying $\varphi(t)=1$ for $t<\frac12$ and $\psi(t)=\frac{-\varphi'(t)}{t} \ge 0$. We will drop the subscript $\mathcal F$ when the measure is clear from the context.

\begin{remark}
As in the classical case one can consider the frequency as a scale independent object. To do so we introduce the scaling map 
\begin{equation}\label{eq.scalingmap}
\hat{\eta}^\lambda_{x_0,r}(x,y,p,M)=\left(\frac{x-x_0}{r},\lambda y, \lambda rp, \lambda^2 r^2M\right).     
\end{equation}
corresponding to the scaling $\eta_{x_0,r}(x)=\frac{x-x_0}{r}$ and the multiplication by the scalar $\lambda$. We will drop $\lambda$ for $\lambda =1$. Moreover, given $\mathcal F\in \operatorname{grad} \mathcal Y^Q$, we define the rescaled measure
\[
(\hat{\eta}^\lambda_{x_0,r})_\sharp\mathcal{F}=\mathcal{F}^\lambda_{x_0,r}\,,
\]
and we remark that a direct computation shows that if $\mathcal{E}_{f_k}$ generates $\mathcal{F}$ then $\mathcal{E}_{\lambda f_k\circ\eta^{-1}_{x_0,r}}$ generates $\mathcal{F}^\lambda_{x_0,r}$.
Moreover we have 
\[ \lambda^2 D_{\mathcal{F}}(x_0,r)= D_{\mathcal{F}^\lambda_{x_0,r}}(0,1) \qquad \lambda^2 H_{\mathcal{F}}(x_0,r)= H_{\mathcal{F}^\lambda_{x_0,r}}(0,1) \qquad \text{and}\qquad I_{\mathcal{F}}(x_0,r)= I_{\mathcal{F}^\lambda_{x_0,r}}(0,1)\,.\]

\end{remark}

\begin{proposition}[Properties of measure solutions]\label{prop.propertiesMeasureSolutions}
The following properties hold. 
\begin{enumerate}
    \item\label{weaklyClosed} The class of inner and outer measure solutions is weakly closed, that is if $(\mathcal F_k)_k$ is a sequence of measure solutions such that $\mathcal F_k(\phi)\to \mathcal F(\phi)$ for all $\phi \in \mathcal C(\Omega\times \mathbb V) $, then $\mathcal F$ is a measure solution. 
    \item\label{hHrmonic} Let $f\in W^{1,2}(\Omega, \Iq(\R^n))$ be a function associated to a measure solution $\mathcal F$, then $\eta\circ f$ is a classical harmonic function.
    \item\label{subtrackingHarmonic} If $\mathcal F=(F\otimes dx, F^\infty )$ is a inner and outer measure solution and $h\colon \Omega \to \R^n$ is a classical harmonic function, then $\tilde{ \mathcal F}=(H_\sharp F\otimes dx, F^\infty)$ is a inner and outer measure solution, where $H\colon \Omega\times \mathbb V \to \Omega\times \mathbb V  $ is defined by
    \[
    H(x, y, p, M):= (x, y-h(x), p_i^\alpha-\partial_ih^\alpha(x),M^{\alpha\beta}_{ij}-p_i^\alpha\, \partial_jh^\beta(x)-\partial_ih^\alpha(x)\, p_j^\beta+\partial_ih^\alpha(x)\, \partial_jh^\beta(x)) \,.
    \]
    Notice in particular that if $f$ is associated to $\mathcal F$, then $f\ominus h$ is associated to $\tilde{\mathcal F}$.
    \item\label{frequencyMonoton}  Let $r_0\in ]0,\dist(0,\partial \Omega)[$ be such that  $H(x_0,r_0)>0$, then the function $(r_0, \dist(x_0, \partial \Omega)) \ni r\mapsto I(x_0,r)$ is monotone non decreasing.
%, and moreover
%\begin{equation}\label{eq:freqremainder}
%    I(x_0,t)-I(x_0,s)\geq 2\int_{s}^tr^{1-n} \left( \int \frac{\varphi'(\frac{|x|}{r})}{|x|r} \sum_{l=1}^Q\left| Df_l(x)x - I(r) f_l\right|^2\right)\,dr
%\end{equation}
    \item\label{integrationH}
     Either $H(x_0, \dist(0,\partial \Omega))=0$ or $H(r)>0$ for any $0<r<R<\dist(0,\partial \Omega)$ and the following holds 
    \begin{equation}\label{eq.integrationH}
        \left(\frac{R}{r}\right)^{I(r)}\le \frac{H(R)}{H(r)} \le \left( \frac{R}{r}\right)^{I(R)}\,.
    \end{equation}
    \item\label{1dimensionalSolutions} If $\mathcal F\in \operatorname{grad} \mathcal Y^{Q}((-R,R)\times \mathbb V)$ is a inner and outer measure solution satisfying $\lim_{r\to0} H(0,r) = 0$, then $\mathcal F=\mathcal E_f$, for the associated function $f\in W^{1,2}(\R, \Iq(\R^n))$, and moreover there are $T_+,T_- \in \Iqs$ with $|T_+|=|T_-|$ such that
    \[ f(t)= \begin{cases} t \,T_+ &\text{ for } t>0\\ t\,T_- &\text{ for } t<0
    \end{cases}\,.\]
\end{enumerate}
\end{proposition}

\begin{proof}
\noindent \emph{Proof of \ref{weaklyClosed}: } We observe that for any $\varphi \in C^\infty_c(\Omega)$ the test functions appearing in the inner and outer variations, i.e.  
\begin{align*}
    &\phi_1(x,y,p,M)= p_i^\alpha y^\alpha \partial_i \varphi+ M^{\alpha\alpha}_{ii} \varphi \quad\text{ for } \varphi \in C^\infty_c(\Omega)\\
    &\phi_2(x,y,p,M) = p_i^\alpha\partial_i\psi^\alpha \quad\text{ for } \psi \in \C^\infty_c(\Omega, \R^n)\\
    &\phi_3(x,y,p,M)= \left(2M_{ij}^{\alpha\alpha} - \delta_{ij} M^{\alpha\alpha}_{kk}\right)\partial_i\phi^j \quad\text{ for } \phi \in C^\infty_c(\Omega, \R^m),
\end{align*}
are all in the space of test functions $\mathcal{C}(\Omega\times \mathbb{V})$. Hence Proposition \ref{le:weakclosed} applies and the claim follows.
\medskip

\noindent \emph{Proof of \ref{hHrmonic}:}
We notice that for the test function $\phi_2(x,y,p,M)= p_i^\alpha\partial_i\psi^\alpha, \psi \in C^\infty_c(\Omega,\R^n)$ one has that the recession function $\phi_2^\infty =0$, hence we can appeal to \eqref{eq.gradientmeasure02} and deduce that 
\begin{equation}\label{eq:ave}
    0=\mathcal{F}(\phi_2) = Q \int \partial_i (\eta\circ f) \partial_i \psi \, dx\,,
\end{equation}
and the claim follows. 

\medskip

\noindent \emph{Proof of \ref{subtrackingHarmonic}:} Let $(\mathcal E_{f_k})_k$ be the generating sequence of $\mathcal F$, then $(\mathcal E_{f_k\ominus h})_k$ is the generating sequence of $\tilde{\mathcal F}$. Indeed let $\phi \in (\Omega\times \mathbb V)$. It is easy to check that
\[
(E_{f_k\ominus h}\otimes dx)(\phi)=(E_{f_k}\otimes dx)(\phi\circ H)\,.
\]
and moreover that
\[
\lim_{|v|\to\infty} \frac{|\phi(H(x,v))-\phi^\infty(x,v)|}{1+|y|^{2*}+|p|^2+|M|}
\le \lim_{|v|\to\infty} \frac{|\phi(H(x,v))-\phi^\infty (H(x,v))|}{1+|y|^{2*}+|p|^2+|M|}+\lim_{|v|\to\infty} \frac{|\phi^\infty(H(x,v)-\phi^\infty(x,v)|}{1+|y|^{2*}+|p|^2+|M|}=0\,,
\]
which implies that the recession function of $\phi$ is the same as that of $\phi\circ H$. This implies that
\[
\tilde{\mathcal F}(\phi)=\mathcal F(\phi \circ H)=\lim_k \mathcal E_{f_k}(\phi\circ H)=\lim_k \mathcal E_{f_k\ominus h}(\phi)\,,
\]
and so $\tilde{\mathcal F}$ is a $\Iq$ generalized gradient Young measure. In particular this also shows that if $f$ is associated to $\mathcal F$, then $f\ominus h$ is associated to $\tilde{\mathcal F}$, and so $\tilde {\mathcal F}$ satisfies \eqref{eq:averageharmonic}, by \eqref{eq:ave}, harmonicity of $h$ and the fact that 
\[
\eta\circ (f\ominus h)=(\eta\circ f)-Q\,h\,.
\]

Next we check \eqref{eq:outervar}. We compute
\begin{align*}
    \mathcal O(\tilde{\mathcal F}, \varphi)
        &=(H_\sharp F\otimes dx)(p_i^\alpha y^\alpha \partial_i\varphi+M_{ii}^{\alpha\alpha} \, \varphi)+ (F^\infty, M_{ii}^{\alpha\alpha} \, \varphi)\\
        &= \mathcal O(\mathcal F, \varphi)+(F\otimes dx)(\partial_ih^{\alpha}h^\alpha \partial_i\varphi+|Dh|^2 \varphi-\partial_ih^\alpha y^\alpha \partial_i\varphi-h^\alpha p_i^\alpha \partial_i\varphi-2 \partial_ih^\alpha p_i^\alpha \varphi )\\
        &=0+\int_\Omega \Div(\varphi\,h^\alpha\,\nabla h^\alpha)\,dx-Q\,\int_\Omega \left(\Div((\eta\circ f)^\alpha\,\varphi\,\nabla h^\alpha)+\Div(\varphi\, h^\alpha\,\nabla( \eta\circ f)^\alpha)\right)\,dx\\
        &=0\,,
\end{align*}
where in the second inequality we used that both $h$ and $\eta\circ f$ are harmonic, and in the last equality that $\varphi$ is compactly supported.

Finally we check \eqref{eq:innervar}. We have
\begin{align*}
    \mathcal I(\tilde{\mathcal F}, \phi)
    &=(H_\sharp F\otimes dx+F^\infty)\left((2 M_{ij}^{\alpha\alpha}-M_{kk}^{\alpha\alpha} \delta_{ij}) \partial_i\phi^j\right)\\
    &=\mathcal I(\mathcal F, \phi)+\underbrace{(F\otimes dx)\left((2 \partial_ih^\alpha\partial_jh^\alpha -\delta_{ij}\partial_kh^\alpha\partial_kh^\alpha)\,\partial_i\phi^j\right)}_{=0\,\text{ since }\Delta h=0}\\
     &\quad-2(F\otimes dx)\left(  (p_i^\alpha \partial_jh^\alpha+ p_j^\alpha \partial_ih^\alpha)\partial_i\phi^j+\Div(\phi) p_k^\alpha \partial_kh^\alpha\right)\\
     &=0-2Q\,\int_\Omega \left(  (\partial_i(\eta\circ f)^\alpha \partial_jh^\alpha+ \partial_j(\eta\circ f)^\alpha \partial_ih^\alpha)\partial_i\phi^j+\Div(\phi) \partial_k(\eta\circ f)^\alpha \partial_kh^\alpha\right)\,dx\\
     &=2Q\,\int_\Omega \left(  \partial_i(\eta\circ f)^\alpha( \partial_{ji}h^\alpha \,\phi^j+\partial_ih^\alpha\Div(\phi))+ \partial_j(\eta\circ f)^\alpha \partial_ih^\alpha\,\partial_i\phi^j\right)\,dx\\
     &=2Q\,\int_\Omega \left(  \partial_i(\eta\circ f)^\alpha\,\partial_j( \partial_{i}h^\alpha \,\phi^j)+ \partial_j(\eta\circ f)^\alpha \,\partial_i(\partial_ih^\alpha\,\phi^j)\right)\,dx=0\,,
\end{align*}
 where we used that $\eta\circ f$ and $h$ are harmonic respectively in the second to last and last equalities above.

\medskip\noindent \emph{Proof of \ref{frequencyMonoton}: }
The computations are essentially the same as in the classical case, but we present them here for completeness. We can assume without loss of generality that $x_0=0$ and write $D(x_0,r)=D(r)$ and analogously for $H$ and $I$.

Firstly we use the classical vector fields $\phi(x)= \varphi(\frac{|x|}{r})\,x$ in the the inner and $\varphi(\frac{|x|}{r})$ in the outer variations to obtain 
\begin{align*}
    0 &= \mathcal{F}\left((2-m)\varphi\left(\frac{|x|}{r}\right)M^{\alpha\alpha}_{kk}-\varphi'\left(\frac{|x|}{r}\right)\frac{|x|}{r} M^{\alpha\alpha}_{kk}\right) + 2 \mathcal{F}\left(\varphi'\left(\frac{|x|}{r}\right)\,\frac{|x|}{r} \,\frac{x^t}{|x|} M \frac{x}{|x|}\right)\,,\\
    0 &= \mathcal{F}\left(\varphi\left(\frac{|x|}{r}\right)M^{\alpha\alpha}_{kk}\right) + \int \frac{\varphi'\left(\frac{|x|}{r}\right)}{|x|r} \sum_{l=1}^Q f_l(x) \cdot Df_l(x)x\,dx\,,
\end{align*}
where we have used in the last equality \eqref{eq.gradientmeasure02}. 
These identities enable us to calculate the derivatives of $D$ and $H$:
\begin{align}\label{eq.D'}
    D'(r)&= r^{1-m} \mathcal{F}\left((2-m) \varphi\left(\frac{|x|}{r}\right) M_{kk}^{\alpha\alpha}-\varphi'\left(\frac{|x|}{r}\right)\frac{|x|}{r}M_{kk}^{\alpha\alpha}\right)\\
    &=-2r^{1-m} \mathcal{F}\left(\varphi'\left(\frac{|x|}{r}\right)\,\frac{|x|}{r} \,\frac{x^t}{|x|} M \frac{x}{|x|}\right)\,,\\\nonumber
    \frac12 r H'(r)&= r^{-m} \int \psi\left(\frac{|x|}{r}\right) \sum_l f_l(x) Df_l(x)\,x \, dx = D(r)\,,
\end{align}
where in the last equality we have used the $\psi(t)=-\frac{\varphi'(t)}{t}$. In particular these equalities show that $D(r)$ and $H(r)$ are monotone non decreasing, so that $H(r_0)>0$ implies $H(r)>0$ for every $r\in ]0, \dist(0, \partial \Omega)[$. This implies that $I(r)$ is well defined in this range. 
Now we can combine these last identities to calculate the derivative of the frequency. We have 
\begin{align}\nonumber
    \frac12 r H(r) I'(r) =& \frac12 r \left(D'(r)- I(r)H'(r)  \right)\\ =& r^{2-m} \mathcal{F}\left(-\frac{\varphi'(\frac{|x|}{r})}{|x|r} x^t M x\right) - r I(r)\,H'(r) + I^2(r)\,H(r) \\\label{eq.remainder1}
    =&r^{2-m} \left(\mathcal{F}\left(-\frac{\varphi'(\frac{|x|}{r})}{|x|r} x^t M x\right)- \int -\frac{\varphi'(\frac{|x|}{r})}{|x|r} \sum_{l=1}^Q |Df_l(x)x|^2\,d x \right) \\\label{eq.remainder2}
    &+r^{2-m} \left( \int -\frac{\varphi'(\frac{|x|}{r})}{|x|r} \sum_{l=1}^Q\left| Df_l(x)x - I(r) f_l\right|^2\right)\,.
\end{align}
Since by \eqref{eq.gradientmeasure03} the first term in the remainder above is nonnegative,  the proof is complete.

\medskip
\noindent \emph{Proof of \ref{integrationH}:}
The claim follows as for classical solutions by the identities established in the previous step and the monotonicity of $I$. Indeed we have  
\[
\frac{d}{dr} \ln(H(r)) = \frac{2}{r} \, \frac{\frac12rH'(r)}{H(r)} = \frac{2}{r} \frac{D(r)}{H(r)} = \frac{2}{r} I(r)\,.
\]
Integrating the above and appealing to the monotonicity of $I$ leads to \eqref{eq.integrationH}. This in particular implies that if $H(r_0)>0$ then $H(r)>0$ for all $0<r<r_0$. 

\medskip
\noindent \emph{Proof of \ref{1dimensionalSolutions}:}
Let $\mathcal{F}$ be a one-dimensional measure solution on an interval $I\subset \R$. The inner variation reads
\[ 0= \mathcal{F}( M^{\alpha\alpha}_{11} \phi' ) \qquad \text{ for all } \phi \in C^\infty_c(I, \R)\,.\]
This implies that there is a constant $m_0\geq0$ such that
\[ (F\otimes dx + F^\infty)(\varphi(x)M^{\alpha\alpha}_{11}) = \int \varphi(x) m_0 \, dx \qquad \forall \varphi \in C^\infty_c(I)\,.\]
In particular we deduce that due to \eqref{eq.gradientmeasure03} the map $f$ is Lipschitz continuous. Furthermore it implies that $D(r)= r^2 D_0$ for some $D_0\ge 0$. Using the above established identity $\frac12 r H'(r)=D(r)$ and the assumption $H(0)=0$ provides $H(r)=D(r)$ and therefore $I(r)=1$ for all $r<R$. Finally we can use the remainders in $I'(r)=0$: \eqref{eq.remainder1} tells us that equality holds a.e. in \eqref{eq.gradientmeasure03} and \eqref{eq.remainder2} that $f$ is $1$-homogeneous. Hence the claim follows. 
\end{proof}

Due to the fact that we have to keep track of the concentration part in the norm of the gradient $M$, we will provide all necessary adaptations related to the monotonicity of the frequency and its consequences. 

\begin{corollary}[{Constant frequency and homogeneity \cite[Corollary 3.16]{DS}}]\label{cor:homogeneousmeassol}
Let $\mathcal{F}$  be a measure solution then $I(0,r)\equiv \alpha$ if and only if $\mathcal F$ is $\alpha$-homogeneous, that is: 
\begin{align}\label{eq.alphaHomogeneous}
    f(\lambda x ) = \lambda^\alpha f(x) \text{ for } \lambda >0\quad \text{ and }\quad\int x^t M x \, d\nu_{x,f_l(x)} &= \alpha^2 |f_l(x)|^2 \text{ for a.e. } x\\\nonumber x^t M^{\alpha \alpha}x \res F^\infty&=0 \;.
\end{align}
\end{corollary}
\begin{proof}The proof is similar to the one given in \cite[corollary 3.16]{DS}.

If $f$ is $\alpha$-homogeneous then one clearly has $H(r)= r^{2\alpha} H_0$, but then $D(r)=\frac12 r H'(r) = \alpha r^{2\alpha} H_0$. Hence $I(r) \equiv \alpha$. This implies $I'(r)\equiv 0$ and therefore the second identity follows from the first part of the remainder in the derivative of $I$, that is \eqref{eq.remainder1}, and \eqref{eq.gradientmeasure03} and \eqref{eq.gradientmeasure04}.

If $I'(r)\equiv 0$ for all $r<r_0$, then the expression in the remainder \eqref{eq.remainder1} must be zero, which reasoning as above gives the second part of \eqref{eq.alphaHomogeneous}. It remains to show that $f$ is $\alpha$-homogeneous. The second remainder \eqref{eq.remainder2} being zero implies that 
    \[\int_{B_{r_0}} \sum_{l=1}^Q |Df_l(x)x-\alpha f_l(x)|^2 \,dx =0\,,\qquad \text{with }\alpha=I(0)\,. \]
Fubini's theorem implies that for a.e. $y \in \partial B_1$ we have $t \mapsto f(ty) \in W^{1,2}((0,r_0), \Iqs))$ and $\int_0^{r_0} \sum_{l=1}^Q |Df_l(ty)ty-\alpha f_l(ty)|^2\, dt=0$. Using the selection principle for $W^{1,2}$-functions \cite[Proposition 1.2]{DS} we can find $Q$ maps $\tilde{f}_l(t) \in W^{1,2}((0,r_0), \R^n)$ such that $f(ty) = \sum_{l=1}^Q \tilde{f}_l(t)$. Combining these two identities we deduce that $\tilde{f}'_l(t)t= \alpha \tilde{f}_l(t)$ for a.e. $t$. Integration provides the claim. 
\end{proof}

Next we introduce the usual blow-up sequence via frequency function and prove that its subsequential limits are homogeneous.

\begin{definition}[Frequency blow-up]
    Let $x_0\in \Omega$ and $r\in ]0,\dist(x_0,\partial \Omega[$ and assume that $D(x_0,r)>0$. We define the \emph{frequency blow up sequence} by
    \[\mathcal{F}_{x_0,r}= (\hat{\eta}_{x_0,r}^{\lambda_r})_\sharp \mathcal{F}\,,\qquad \text{with }\lambda_r:=\frac{1}{\sqrt{H_{\mathcal{F}}(x_0,r)}}\,.
    \]
    We will call the subsequential limits of the sequence $(\mathcal F_{x_0,r})_r$ \emph{tangent measures to $\mathcal F$ at $x_0$}.
\end{definition}

\begin{corollary}[Blow-ups]\label{cor.frequencyblowup}
    Let $\mathcal{F}$ be a inner and outer measure solution on $B_1$. Assume $\lim_{r\downarrow 0} H(r) = 0$ and $D_{\mathcal F}(\rho)>0$ for every $\rho\leq 1$. Then, for any sequence $\{\mathcal F_{\rho_k}\}$ with $\rho_k \downarrow 0$, a subsequence, not relabeled, converges weakly to a $\Iq$-generalized gradient Young measure $\hat{\mathcal F}$ with the following properties:
\begin{enumerate}
    \item $H_{\hat{\mathcal F}}(1)=1$ and $\hat{\mathcal F}$ is a inner and outer measure solution;
    \item $\hat{\mathcal F}$ is $\alpha$ homogeneous, that is \eqref{eq.alphaHomogeneous} holds with $\alpha=I_{\mathcal F}(0,0)>0$. 
\end{enumerate}
\end{corollary}
\begin{proof}
    Let $R>1$ be arbitrarily fixed, then for any $r>0$ such that $rR<r_0\le 1$ we have
    \[\norm{\mathcal{F}_{r}}_{\mathcal{Y}^Q} \le \frac{1}{H_{\mathcal F}(r)} \left(H_{\mathcal F}(rR) + 2D_{\mathcal F}(rR) \right)\le R^{I(r_0)}\left( 1 + 2 I_{\mathcal F}(r_0)\right)\,,\] 
    where we have used the monotonicity of $I$ and \eqref{eq.integrationH}. Hence for any $\rho_k \downarrow 0$ the sequence $\mathcal{F}_k = \mathcal{F}_{x_0,\rho_k}$ is uniformly bounded on $B_R$ for any $R>1$. Appealing to Propositions \ref{le:weakclosed} and \ref{prop.propertiesMeasureSolutions} we can pass to a subsequence converging weakly to a inner and outer measure solution $\hat{\mathcal{F}}\in \operatorname{grad}\mathcal Y^Q(\R^m\times\mathbb V)$. As observed above, the test functions in the definitions of $D_\mathcal{F}$ and $H_\mathcal{F}$ are in the class of test functions $\mathcal{C}(B_1 \times \mathbb{V})$, so that for any $R>1$
    % \begin{align*}
    %     &D_{\hat{\mathcal{F}}}(x_0,R) = \lim_{k\to \infty} D_{\mathcal{F}_{\rho_k}}(x_0,R) = H(\rho_k)^{-2} D_{\mathcal{F}}(\rho_k x_0, \rho_k R)\\
    %     & H_{\hat{\mathcal{F}}}(x_0,R) = H_{\mathcal{F}_{\rho_k}}(x_0,R) = H(\rho_k)^{-2} H_{\mathcal{F}}(\rho_k x_0, \rho_k R)\;.
    % \end{align*}
     \begin{align*}
        %& H_{\hat{\mathcal{F}}}(R) = \lim_{k\to \infty} H_{\mathcal{F}_{\rho_k}}(R) =\lim_{k\to \infty} H(\rho_k)^{-1} H_{\mathcal{F}}(\rho_k R)= \lim_{k\to \infty} \frac{H(R\rho_k)}{H(\rho_k)} \le R^{I(r_0)}\\
        %&D_{\hat{\mathcal{F}}}(R) = \lim_{k\to \infty} D_{\mathcal{F}_{\rho_k}}(R) =\lim_{k\to \infty} H(\rho_k)^{-1} D_{\mathcal{F}}(\rho_k R)= \lim_{k\to \infty} \frac{H(R\rho_k)}{H(\rho_k)} I(\rho_k R) \le R^{I(r_0)} I(0)\\
        &I_{\hat{\mathcal{F}}}(R)= \lim_{k\to \infty} I_\mathcal{F}(\rho_k R) = I(0)\,.
    \end{align*}
    Therefore  Corollary \ref{cor:homogeneousmeassol} applies, and the conclusion follows.
\end{proof}

Next we consider blow-ups of homogeneous functions.

\begin{lemma}[Cylindrical blow-ups {\cite[Lemma 3.24]{DS}}]\label{lem:cylbu}
    Let $\mathcal F\in \operatorname{grad}\mathcal Y^Q(\R^m\times\mathbb V)$ be a inner and outer measure solution which is $\alpha=I_{\mathcal F}(0)$-homogeneous, that is it satisfies \eqref{eq.alphaHomogeneous}, and such that $H_{\mathcal F}(0,1)>0$. Suppose moreover that $\lim_{r\to 0}H_{\mathcal F}(e_1,r)=0$. Then any tangent measure $\hat{\mathcal F}=(F_{x_1} \otimes dx', \hat{F}^\infty_{x_1})\otimes dx_1 $ to $\mathcal F$ at $e_1$ satisfies:
    \begin{enumerate}
        \item $\hat{\mathcal F} $ is $\alpha$ homogeneous and $H_{\hat{\mathcal F}}(0,1)=1$;
        \item if $\hat{f}$ is the map associated to $\hat{\mathcal F}$, then $\hat{f}(x_1,x_2,\dots,x_m)=\hat{f}(x_2,\dots,x_m)$ and $\hat{f}(s e_1)=Q\a{0}$. Moreover we have for a.e. $x$ and any $l=1,\dots,Q$
        \[ \int M e_1 \, d\nu_{x, \hat{f}_l(x)} = 0\,. \]
        \item Inner and outer variations for $\hat{\mathcal F}$ do not depend on $x_1$, that is denoting with $x'=(x_2,\dots,x_n)$ we have for a.e. $x_1\in \R$ that $\mathcal{F}_{x_1}=(F_{x_1} \otimes dx', \hat{F}^\infty_{x_1})$ is a $\A_Q$ stationary gradient Young measure i.e.  
        \begin{align*}
        	\mathcal{F}_{x_1}\left(p_j^\alpha y^\alpha \partial_j\varphi +\varphi M_{jj}^{\alpha\alpha} \right) &=0 \quad  \forall \varphi \in C^\infty_c(\R^{m-1})\\
        	\mathcal{F}_{x_1}\left((2M^{\alpha\alpha}_{ij}-M^{\alpha\alpha}_{kk}\delta_{ij})\partial_i\phi^j\right)&=0 \quad \forall \phi \in C^\infty_c(\R^{m-1},\R^{m-1}) 
        \end{align*}

        %{\color{red}
        %\begin{gather*}
         %   \int\sum_{l=1}^Q \partial_j\hat{f}_l(x')\hat{f}_l(x')\partial_j\varphi(x') \,dx'+ \int \varphi(x')M^{\alpha\alpha}_{jj}\, d\nu_{(x_1,x'),\hat{f}_l(x')}\,dx'=0\quad \forall \varphi \in C^\infty_c(\R^{m-1})\,,\\
         %   \int (2M^{\alpha\alpha}_{ij}-M^{\alpha\alpha}_{kk}\delta_{ij}) \, d\nu_{(x_1,x'),\hat{f}_l(x')}\, \partial_i\phi^j(x')=0 \quad \forall \phi \in C^\infty_c(\R^{m-1},\R^{m-1})\,,
        %\end{gather*}}
        where $i,j=2,\dots,n$. Moreover $\eta \circ \hat{f}(x')$ is harmonic.
    \end{enumerate}

\end{lemma}

\begin{proof}
Given any sequence $r_k\downarrow 0$, we consider a related blow-up sequence $\mathcal{F}_k=\mathcal{F}_{e_1, r_k}$ as in Corollary \ref{cor.frequencyblowup}. We may assume that we fixed a converging subsequence with $\mathcal{F}_k \to \hat{\mathcal{F}}$. Then (1) follows immediately from Corollary \ref{cor.frequencyblowup}.

Observe that since $\mathcal{F}$ is homogeneous we have for any $x_0, r,\lambda>0$ that $I_{\mathcal{F}}(x_0, r)= I_{\mathcal{F}}(\lambda x_0, \lambda r)$. Using this in the second equality below, we obtain for any $t\in \R$ and $r>0$
\[I_{\mathcal{F}_k}(t e_1, r)= I_{\mathcal{F}}(e_1 + tr_k e_1, r_k r) = I_{\mathcal{F}}\left(e_1, \frac{r_k}{1+t r_k} r\right)=I_{\mathcal{F}_k}\left(0, \frac{r}{1+ tr_k}\right)\,.\]
Taking the limit $k\to \infty$ we obtain 
\begin{equation}\label{eq.splittingFrequency}
    I_{\hat{\mathcal F}}(te_1, r) = I_{\hat{\mathcal F}}(0,r)=\alpha\qquad \forall t\in \R\,, r>0.
\end{equation}
Thus we can appeal to Corollary \ref{cor:homogeneousmeassol} and deduce that $(\hat{\eta}_{te_1, 1})_\sharp \hat{\mathcal{F}}$ is $\alpha$ homogeneous for every $t\in \R$. In particular $\hat{f}(e_1+\lambda y)=\lambda^\alpha \hat{f}(e_1+y)$, so that  
\[ 
\hat{f}(x) = \lambda^{-\alpha} \hat{f}(\lambda x) = \lambda^{\alpha} \hat{f}(e_1 + \lambda^{-1}\left(\lambda x -e_1\right)) = \hat{f}(x + (\lambda - \lambda^{-1})e_1)\,.  
\]
Hence $\hat{f}$ is invariant along the line $e_1$, which proves (2).
Finally we want to show that for a.e. $x$ and any $l$
\[ \int_{\mathbb V} M^{\alpha \alpha}_{i1} \, d\nu_{x, \hat{f}_l(x)} = 0 \quad \text{ for all } i\qquad \text{and}\qquad M^{\alpha\alpha }_{11}\res F^{\infty}\,. \]
Consider for a fixed non-negative function $\phi(x)$
\[ g(t)= \hat{\mathcal{F}}(\phi(x)(x-te_1)^t M^{\alpha\alpha} (x- te_1) ) - \int \phi(x) \sum_{l=1}^Q |D\hat{f}_l(x)(x-te_1)|^2 \, dx\,. \]
Due to \eqref{eq.splittingFrequency}, the remainder \eqref{eq.remainder2} at $x_0=te_1$ must be zero, so that $g(t) \equiv 0$. We may differentiate twice in $t$ and conclude that  
\[0 = g''(t) = \hat{\mathcal{F}}(\phi(x) M^{\alpha \alpha}_{11} ) - \int \phi(x) \sum_{l=1}^Q |D\hat{f}_l(x)e_1|^2= \hat{\mathcal{F}}(\phi(x) M^{\alpha\alpha}_{11} )\,,\]
where in the last equality we used that $\hat{f}$ is independent of $e_1$. This concludes since $\int_{\mathbb{V}_2\times \mathbb{V}_3} M^{\alpha\alpha}_{ij}\nu_{x, f_l(x)}$ and $F^\infty$ are supported by \eqref{eq.gradientmeasure03} and \eqref{eq.gradientmeasure04} on positive definite matrices.

It remains to show that every slice $x_1=t$ is itself a measure solution. We will use $x' =(x_2, \dotsc, x_n) \in \R^{n-1}$. Since $\partial_1 (\eta \circ \hat{f})=0$ and $\Delta (\eta \circ \hat{f})=0$ we clearly have \eqref{eq:averageharmonic} for the map $\hat{f}(x')$. 

Applying the disintegration theorem we have $\hat{\mathcal{F}}=(\hat{F}\otimes dx, \nu^\infty_{x_1} \otimes d\lambda(x_1))$, with $d\lambda(x_1)=\rho(x_1)\, dx_1 + d\lambda^{\sing}$. Since $\hat{F}$ itself is an $\A_Q$-gradient Young measure we deduce as a consequence of Fubini's theorem that $\hat{\mathcal{F}}_{x_1}=(\hat{F}_{x_1}\otimes dx', \rho(x_1) \nu^\infty_{x_1})$ is a $\A_Q$-gradient Young measure.\footnote{This can be seen as follows: Let $\mathcal{E}_{f_k} \rightharpoonup \mathcal{F}$ and $(1+|f_k|^{2p^*}+|Df_k|^2) \,dx \to \mu=\nu_{x_1}\otimes d\sigma(x_1)$. Since for every $x_1\notin \spt(d\sigma^\sing)$ we have $\lim_k \int (1+|f_k(x_1,x')|^{2p^*} + |Df_k(x_1, x')|^2) \, dx' < \infty$, we deduce that $\mathcal{E}_{f_k(x_1, \cdot)} \rightharpoonup \mathcal{F}_{x_1}$ in the sense of $\A_Q$-gradient young measures. On the other hand applying the disintegration theorem to $\hat{F}^\infty=\nu^\infty \otimes d\lambda(x)1)$ i.e. $\mathcal{F}=(F_{x_1}\otimes dx')\otimes dx_1 + \nu^\infty \otimes d\lambda(x_1)$ we deduce that for any admissible test function $\psi=\psi(x',y,p,M)$ independent of $p_1, M_{1\cdot}$ and $\eta\in C_c^\infty(\R)$ 
\begin{align*}
    &\int \eta(x_1)\, \left( (F_{x_1}\otimes dx')(\psi) \, dx_1 + \nu^\infty(\psi)\, d\lambda(x_1)\right)= \mathcal{F}(\eta(x_1)\psi) = \lim_k \mathcal{E}_{f_k}(\eta(x_1)\psi)\\
    &=\lim_k \int \eta(x_1) \mathcal{E}_{f_k(x_1, \cdot)}(\psi)\,.
\end{align*}
Hence we deduce that for a.e. $x_1$ we have $\mathcal{F}_{x_1}= (F_{x_1}\otimes dx') + \rho(x_1)\nu^\infty$\,.} 

\smallskip

\noindent\emph{Outer variations:} For any $\eta(x_1)\in C^1_c(\R), \varphi(x')\in C^1_c(\R^{n-1})$, implicitly summing in $j=2, \dots, n$, we have 
\begin{align*}
    0 &= \hat{\mathcal{F}}(p_1^\alpha y^\alpha \eta'(x_1)\varphi(x') + \eta(x_1) p_j^\alpha y^\alpha \partial_j\varphi(x') + \eta(x_1)\varphi(x') (M_{11}^{\alpha\alpha}+ M_{jj}^{\alpha\alpha}))\\& = \hat{\mathcal{F}}(\eta(x_1) \left(p_j^\alpha y^\alpha \partial_j\varphi(x') +\varphi(x') M_{jj}^{\alpha\alpha} \right) )\\&
    %= \int \eta(x_1) \left( \sum_{l=1}^Q \partial_j\hat{f}_l(x')\hat{f}_l(x')\partial_j\varphi(x') + \int \varphi(x')M^{\alpha\alpha}_{jj}\, d\nu_{x,\hat{f}_l(x')}\right) \, dx,
\end{align*}
where in the second equality we used (2) in the statement.
Note that the above implies that for $\Psi_O= \left(p_j^\alpha y^\alpha \partial_j\varphi(x') +\varphi(x') M_{jj}^{\alpha\alpha} \right)$ we have \[(\hat{F}_{x_1}\otimes dx')(\Psi_O)\, dx_1 + \nu^\infty_{x_1}(\Psi_O)\, d\lambda \equiv 0 \,. \]
This in particular implies that $\nu^\infty_{x_1}(\Psi_O)\, d\lambda= \nu^\infty_{x_1}(\varphi M^{\alpha \alpha}_{jj})\, d\lambda$ does not have a singular part for any of these $\Psi_O$'s. Since $\nu^\infty_{x_1}$ is supported on the non-negative matrices and $\varphi$ is arbitrary, this implies that $d\lambda$ does not have a singular part, i.e. $d\lambda=\rho(x_1) \, dx_1$ \footnote{One could this information to show that every $\mathcal{F}_{x_1}$ is a $\A_Q$-gradient young measure that is stationary for every $x_1$.}.  Setting $\hat{F}^\infty_{x_1}:=\rho(x_1)\nu^\infty_{x_1}$ we conclude that $\mathcal{F}_{x_1}$ is stationary for the outer variations for almost every $x_1$.  

\smallskip

\noindent\emph{Inner variations:} For any $\eta(x_1)\in C^1_c(\R), \phi(x')\in C^1_c(\R^{n-1},\R^{n-1})$ implicitly summing in $i,j,k=2, \dots, n$ we have 

\begin{align*}
    0&=\hat{\mathcal{F}}( 2M^{\alpha\alpha}_{1j}\eta'(x_1)\phi^j(x')-M^{\alpha\alpha}_{11}\eta(x_1)\operatorname{div}\phi(x') + \eta(x_1)\left((2M^{\alpha\alpha}_{ij}-M^{\alpha\alpha}_{kk}\delta_{ij})\partial_i\phi^j(x') \right))\\
    &=\hat{\mathcal{F}}(\eta(x_1)\left((2M^{\alpha\alpha}_{ij}-M^{\alpha\alpha}_{kk}\delta_{ij})\partial_i\phi^j(x') \right))\\
    %&= \int \eta(x_1) \left(\int (2M^{\alpha\alpha}_{ij}-M^{\alpha\alpha}_{kk}\delta_{ij}) \, d\nu_{x,\hat{f}_l(x')}\, \partial_i\phi^j(x') \right) \, dx\,,
\end{align*}
where we used (2) in the second equality once again. 

% Note that $I_{\mathcal{F}}(e_1 + r_k s e_1, r_k r) = I_{\mathcal{F}_k}(e_1, r)$ for all $r_k$. 
% Now we want to show that $I_{\hat{\mathcal{F}}}(t e_1, r) = I_{\hat{\mathcal{F}}}(0, R)$ for all $t\in \R$ and $0<r,R$. 

 Arguing as for the outer variation, we deduce that 

\[\left((\hat{F}_{x_1}\otimes dx')(\Psi_I) + \hat{F}^\infty_{x_1}(\Psi_I)\right) \, dx_1 \equiv 0 \,, \]
where $\Psi_I=(2M^{\alpha\alpha}_{ij}-M^{\alpha\alpha}_{kk}\delta_{ij})\partial_i\phi^j(x')$. 
\end{proof}

\subsection{Regularity: proof of Theorem \ref{thm:UCformeas}}
    Due to proposition \ref{prop.propertiesMeasureSolutions} part \ref{subtrackingHarmonic} we may assume that $\eta\circ f \not\equiv 0$. 
    We prove the theorem as in the case of classical solution by induction on $m$. 
    
    The base case $m=1$ coincides with Proposition \ref{prop.propertiesMeasureSolutions} part \ref{1dimensionalSolutions}, i.e. the set $\{f = Q\a{0}\}$ is at most one point. 

    Suppose the theorem holds for $m'<m$ and assume by contradiction that $\mathcal{H}^{t}(E)>0$ for some $t>m-1$ and $E=\{f=Q\a{0}\}$. If $\Omega \neq E$, there is a point $x_0 \in \Omega\cap E$ with positive $\mathcal H^t\res E$ density that is not in the measure theoretic interior of $E$, i.e. \begin{equation}\label{eq:positivedensity}
        \limsup_{r \to \infty} \frac{ \mathcal{H}^{t}(E\cap B_r(x_0))}{r^{t}} >0 \quad \text{ and } \quad \liminf_{r\to \infty} \frac{ \mathcal{H}^{n}(B_r(x_0)\setminus E)}{r^{n}} >0\,.
    \end{equation}
    After translation we may assume $x_0 =0$. Now the conclusion follows almost by the same arguments presented in \cite[Subsection 3.6.2]{DS}. Indeed
    let $r_k \downarrow 0$ be a subsequence realising the $\limsup$ in \eqref{eq:positivedensity} and consider the corresponding blow-up sequence $\mathcal{F}_{r_k}$. By Corollary \ref{cor.frequencyblowup}, we find a nontrivial $\alpha$-homogeneous measure solution $\hat{\mathcal{F}}$.  Moreover  by \eqref{eq:positivedensity} and the fact that $f_{k} \to \hat{f}$ converges uniformly up to a set of arbitrary small $q$-capacity for every $q<2$ we conclude that $\hat{f}\not\equiv Q\a{0}$ and satisfies 
$$\mathcal{H}^t_\infty( B_1 \cap \{ \hat{f}=Q\a{0}\})>0.$$ 
Therefore there exists $y \in \partial B_1 \cap \{\hat{f}=Q\a{0}\}$ once again with positive $\mathcal{H}^t_\infty$-density and not in the measure theoretic interior. After rotation we may assume that $y=e_1$. Now we can argue as above and perform a second blow-up at $e_1$. However the corresponding tangent function $\hat{\hat{\mathcal{F}}}$ is nontrivial, homogeneous and inner and outer measure solution independent of $e_1$, but still satisfies $\mathcal{H}^{t-1}_\infty( B_1 \cap \{ \hat{\hat{f}}=Q\a{0}\})>0$. This contradicts the inductive assumption.
\qed

\subsection{Proof of Propositions \ref{prop:higofmeas} and \ref{prop:regformeas}} In this subsection we will work with stationary measures $\mathcal F$, in particular the concentration part $F^\infty\equiv 0$. 

\begin{proof}[Proof of \ref{prop:higofmeas}]
The proof is based on a simple contradiction argument which combines inner variation and strong outer variation.

\smallskip

\noindent\emph{Claim: }
Let $\mathcal{F}=(F\otimes dx,0)$ be a measure solution that satisfies the strong form of the outer variations \eqref{eq:strongouter} and 
\begin{equation}\label{eq.hiInt1scale}
    \left((F\otimes dx)\left(2^{-m}\varphi(2|x|)|M|^p\right)\right)^{\frac1p} \le C (F\otimes dx)\left(\varphi(|x|)|M|\right)
\end{equation}
Then for any $r<\frac13$ there is $\theta=\theta(m,n,Q,C,p,r)<1$ such that 
\begin{equation}\label{eq.campanoatodecay} D(r)\le \theta D(1)\,.\end{equation}

Suppose the claim does not hold, hence there is a sequence $(F_k\otimes dx)$ contradicting \eqref{eq.campanoatodecay} for a sequence $\theta_k \uparrow 1$. Passing to the normalised sequence $(\hat{\eta}^{\lambda_k})_\sharp (F_k\otimes dx)$, with $\lambda_k^{-2}= D_{\mathcal{F}_k}(1)$, we may assume that $D_{\mathcal{F}_k}(1)=1$ for all $k$. Integrating the expression for the differential of $D$ in \eqref{eq.D'}, we have for $\tilde{\varphi}(t)-\tilde{\varphi}(s)= \int_s^t \varphi'(\tau)\tau^{m-2}\,d\tau$, 
\begin{align} c \int_{B_{\frac12}\setminus B_{\frac13}} \sum_{l} \left|Df_l(x)\frac{x}{|x|}\right|^2 \, dx 
    &\le (F_k\otimes dx)\left(\left(\tilde{\varphi}(|x|)-\tilde{\varphi}\left(\frac{|x|}{r}\right)\right) \frac{x^tMx}{|x|^m}\right)\nonumber\\
    &\le D_{\mathcal{F}_k}(1)-D_{\mathcal{F}_k}(r) \le (1-\theta_k)\,,\label{eq:radtozero}
\end{align}
where we have used \eqref{eq.gradientmeasure02} and that $\tilde{\varphi}(|x|)-\tilde{\varphi}(\frac{|x|}{r})>c$.% Note that $\tilde{\varphi}(|x|)-\tilde{\varphi}(\frac{|x|}{r})>0$ on $\frac13<|x|<\frac12$. 

We claim that this contradicts the strong outer variation. The argument is close to the arguments used in the concentration compactness. Since we need to preserve the inner variation, an argument based on induction on $Q$ is cumbersome. Hence we argue directly. The same argument would apply for the concentration compactness as well. 
We select a sequence of averages $T_k$ for the associated maps $f_k$, i.e. $\norm{\G(f_k, T_k)}_{L^2(B_\frac12)} \lesssim \norm{Df_k}_{L^2(B_\frac12)}\le D(1)$. Associated to $T_k$ we can find a sequence of points $p^k_1, \dots, p^k_N \in \R^n$ and a sequence $s_k \to \infty$ such that 
\begin{enumerate}
    \item $|p^k_i-p^k_j|>4s_k$ forall $i\neq j$;
    \item\label{finitediam} $T_k=\sum_j^N T^k_j$ with $T_j^k \in \A_{Q_j}(\R^n)$, $\limsup_k|T^k_j\ominus p^k_j|<\infty$.
    % \item $\spt(T_k) \subset \bigcup_{j} B_{s_k}(p_j^k)$ i.e. $T=\sum_{j}^N T_j$ wi
    % \item $\operatorname{card}(T_k \cap B_{s_k}(p_j^k))=Q_j$ independent of $k$.
\end{enumerate}
% We may fix Lipschitz retractions $\chi_i^k\colon \Iqs \to B_{2s_k}(Q_j \a{p_j^k})$. By construction we have 
% \[
% \limsup_{k\to \infty} \norm{\chi_i^k(f_k)\ominus p_j^k}_{L^2(B_\frac12)} \lesssim 1\,. 
% \]
Let us define $S_k= \sum_j Q_j \a{p_j^k}$ and observe that \eqref{finitediam} implies that $\norm{\G(f_k, S_k)}_{L^2(B_\frac12)} \lesssim 1$.
Now we use the strong form of the outer variation. Fix a smooth non-decreasing function $\theta$ vanishing on $[2,\infty)$ and equal to $1$ on $[0,1]$. Next, define the function 
\[\tilde{\theta}_k(y)=\sum_{j} \theta\left(\frac{|y-p_j^k|}{ s_k}\right) (y-p_j^k)\] 
and observe that  
\[\mathbf{1}-D_y\tilde{\theta} = \underbrace{\sum_{j} \left(1-\theta\left(\frac{|y-p_j^k|}{s_k}\right)\right) \mathbf{1}}_{=:\Theta_k^1(y)} +  \underbrace{ \frac{|y-p_k^j|}{s_k} \theta'\left(\frac{|y-p_j^k|}{s_k}\right) \frac{y-p_j^k}{|y-p_j^k|}\otimes \frac{y-p_j^k}{|y-p_j^k|}}_{=:\Theta_k^2(y)}\,.\]
The matrix valued functions $\Theta_k^i(y)$ are bounded and supported in $\cup_j B_{s_k}(p_j^k)$. 
Hence we can test the strong outer variation with $\eta(x)\,\tilde{\theta}(y)$, where $\eta$ is a smooth function which is $1$ in the ball of radius $r$, less than or equal to $1$, and supported in $B_1$, to obtain the  estimate 
\begin{align*}
    \mathcal{F}_k(\eta\, |M|) &= \mathcal{F}_k(\eta\,\mathbf{1}\colon  M)=\mathcal{F}_k(\eta \,\Theta_k^1(y): M) - \mathcal{F}_k(\eta\, \Theta_k^2(y):M) - \mathcal{F}_k\left(p^\alpha_i\frac{x_i}{|x|} \theta^\alpha(y) \eta'\right)\\
   &\lesssim |\{\G(f_k, S_k)>s_k\}|^{1-\frac1p} (\mathcal{F}_k(\varphi(2|x|))|M|^p)^\frac{1}{p} + \norm{\G(f_k,S_k)}_{L^2(B_\frac12)} \norm{Df_k\frac{x}{|x|}}_{L^2(B_\frac12\setminus B_{\frac13})}\\
   &\to 0,
\end{align*}
where in the last line we used \eqref{eq:radtozero} and the fact that $|\{\G(f_k,S_k)>s_k\}|\to 0$ as $k\to \infty$, as a consequence of Chebyshev and the fact that $\norm{\G(f_k, S_k)}_{L^2(B_{1/2})}\lesssim 1$ and $s_k \to \infty$.

This contradicts $D_{\mathcal{F}_k}(r) \ge \theta_k$ for sufficiently large $k$ and proves the claim.

\smallskip
Now we may fix $\rho<\frac13$ and iterating the above by considering the re-scaled solutions $(\hat{\eta}^1_{x_0, \rho^k})_\sharp \mathcal{F}$ leads to
\[ D(x_0,\rho^k) \le \theta^k D(x_0,\frac12) \quad \forall x_0 \in B_1, k \in \N\,.\]
Appealing once more to \eqref{eq.gradientmeasure02} this implies to a Morrey decay: there is a $\alpha >0$, $C>0$ such that for all $x_0 \in B_1$ and $r<1$ 
\begin{equation}\label{eq:Morrey} r^{2-m-2\alpha} \int_{B_r(x_0)} |Df_k|^2 \le C D(0,2)\,. \end{equation}
This concludes the proof.

% \[\varphi(x,y)=\eta(|x|) \left(y-P^k(y)\right)\]
% where $P^k\in C^1(\R^n,\R^n)$ with $P^k(y)=p_j^k$ if $y \in B_{s_k}(p_j^k)$ and $\eta(t) \in C^1$ non-increasing, $\eta(t)=1$ for $t<\frac13$ and $\eta=0$ for $t\ge \frac12$.
\end{proof}

\begin{proof}[Proof of \ref{prop:regformeas}] Due to proposition \ref{prop.propertiesMeasureSolutions} part \ref{subtrackingHarmonic}, we can assume without loss of generality that the map $f$ associated to $\mathcal F$ is average free, that is $\eta \circ f\equiv 0$.
We prove the statement by induction on $Q$.

\emph{$Q=1$ :} We need to show that $\mathcal{F}( \varphi(x) |M|)=0 $ for all $ \varphi \in C^\infty_c(\R^m)$. This follows immediately from \eqref{eq:outervar} and the fact that $(F\otimes dx)(p_i^\alpha y^\alpha \partial_i\varphi) = \int \partial_i f(x) f(x) \partial_i \varphi(x) \, dx =0$, since $f=\eta\circ f\equiv 0$.

\emph{$Q'\to Q$ :} Let $E:=\{x\in \Omega\,:\, |f(x)|>0\}$. 
Since $f$ is continuous, for every $x_0 \in E$ we can find non-empty disjoint open sets $O_1, O_2$ and a radius $r=r(x_0)>0$ such that $\spt f(x) \in O_1 \cup O_2 \quad \forall x \in B_r(x_0)$. Furthermore we may assume that $\spt f(x_0) \cap O_i \neq \emptyset$. Hence there are related Lipschitz retractions $\chi_i\colon \Iq \to \A_{Q_i}$ for $i=1,2$ such that  
\[ f(x)=\chi_1\circ f(x) + \chi_2\circ f(x) \quad \forall x \in B_r(x_0)\,.\]

Now we can argue analogously as in the case of Proposition \ref{p:elementarylimit} part \eqref{eq.gradientmeasure02}.
Let $\mathcal{E}_{f_k}$ be a generating sequence for $\mathcal{F}$. By the strong outer variation assumption, it does not generate a concentration part, i.e $F^\infty\equiv 0$. We deduce that $F\otimes dx = F_1\otimes dx + F_2\otimes dx$, where $F_i\otimes dx$ is generated by $\mathcal{E}_{\chi_i\circ f_k}$.
Finally using a test function $\varphi(x,y)$ vanishing on $\overline{O}_2$ in \eqref{eq:strongouter} we deduce that 
\[0 = (F_1\otimes dx)(p_i^\alpha \partial_i\varphi^\alpha(x,y) + M_{ii}^{\alpha \beta} \partial_{y^\beta}\varphi^\alpha(x,y))\,.\]
This implies that $F_1\otimes dx$ satisfies itself the strong outer variation. By induction we deduce that $F_1\otimes dx = \mathcal{E}_{f_1}$ on $B_r(x_0)$ for some Sobolev function $f_1 \in W^{1,2}(B_r(x_0),\A_{Q_1})$. The same argument for $F_2\otimes dx$ shows that the claim holds on $B_r(x_0)$. 

It remains to show that the set $\{|f|=0\}$ does not contribute in the outer variation. Here we can argue as in \cite[theorem 1.6]{HS}. Fix a smooth non-decreasing function $\theta$ vanishing on $(-\infty, 1]$ and equal to $1$ on $[2,\infty)$. Moreover let $\eta$ be a smooth function which is $1$ in the ball of radius $r$, less than or equal to $1$, and supported in $B_1$. Then we test \eqref{eq:strongouter} with the test function
\[
\varphi(x,y):=\eta(x)\,\theta\left(\frac{\ln|y|}{\ln\delta}\right)\, y\,,
\]
to obtain
\begin{align*}
\mathcal F\left(\eta\,\theta\left(\frac{\ln |y|}{\ln\delta}\right) M_{ii}^{\alpha\alpha}  \right)
    &=-\mathcal F\left(y^\alpha\,p_i^\alpha \partial_i\eta\, \theta\left(\frac{\ln|y|}{\ln \delta} \right)\right)-\frac{1}{\ln\delta}\mathcal F\left(\eta\, \theta'\left(\frac{\ln|y|}{\ln\delta}\right)M_{ii}^{\alpha\beta}\,\frac{y^\alpha}{|y|}\, \frac{y^\beta}{|y|} \right)\\
    &\leq \frac{C}{|\ln \delta|}\,\|\mathcal F\|_{\mathcal Y^Q}\to 0\,, \qquad \text{as }\delta\to 0 \,.
\end{align*}
This concludes the proof. 

% By continuity of $f$, for every $x_0\in E$ we can argue similarly as in the case of Proposition \ref{p:elementarylimit} part \eqref{eq.gradientmeasure02}. There is a point $S\in \Iq$\footnote{For instance one may apply \cite[Lemma 3.8]{DS} to $f(x_0)$.} and related Lipschitz retractions $\chi_i\colon \Iq \to \A_{Q_i}$ for $i=1,2$ and  $Q_1+Q_2=Q$ and  a radius $r=r(x_0)>0$ such that 
% \[ f(x)=\chi_1\circ f(x) + \chi_2\circ f(x) \quad \forall x \in B_r(x_0)\,.\]
% Let $\mathcal{E}_{f_k}$ be a generating sequence for $\mathcal{F}$. In particular it does not generate a concentration part. As argued in the proof of \eqref{eq.gradientmeasure02}, we deduce that $F\otimes dx = F_1\otimes dx + F_2\otimes dx$, where $F_i\otimes dx$ is generated by $\mathcal{E}_{\chi_i\circ f_k}$.  

% and two disjoint open sets $O_1, O_2 \subset \R^n$ such that 
% \[ \spt (f(x)) \in O_1 \cup O_2 \quad \forall x \in B_r(x_0)\,.\]
% We may fix $O_1 \Subset O \subset \overline{O_2}^c$ and fix a smooth partition of unity $\chi_1,\chi_2$ subordinate to the open sets $O$ and $\overline{O_1}^c$. Now consider the maps
% \[\hat{\chi}_i(x,y,p,M) = (x, \chi_i(y), D\chi_i(y)p, D\chi_i^t(y)MD\chi_i(y)\,.\]
% One readily checks that if $\mathcal{E}_{f_k}$ generates $\mathcal{F}$ then
% We fix a smooth partition of unity subordinate to 

% and two functions $f_i\colon B_r(x_0)\to \I{Q_i}(\R^n)$, $Q_i\geq 1$ with $Q_1+Q_2=Q$, such that
% \[
% f(x)=\a{f_1(x)}+\a{f_2(x)}\qquad \forall x\in B_r(x)\,.
% \]
% Notice that we may use %proposition \ref{p:elementarylimit} \eqref{de:Aqgradmeasures}
% to define 
\end{proof}

\subsection{Compactness: proof of Theorem \ref{thm:compactness}}
% {\color{blue}Since it is now for the fourth time that a similar concentration compactness argument appears - shall we move it to the appendix?}
% The proof is a modification of the concentration compactness argument presented in \cite[Theorem 1.6]{HS}. The main difference is that we want to be able to reverse the concentration along the sequence. Hence we need to keep track where the concentration happens. An additional problem arises since the inner variation does not localise, hence we are not allowed to consider the different pieces on their own.
After rescaling we may assume without loss of generality that $C=1$ in c1) i.e. $\limsup_k \int_{B_4} |Df_k|^2 \le 1$ for all $k$. 

Associated to a sequence of averages $T_k$ for $f_k$, i.e. $\norm{\G(f_k, T_k)}_{L^2(B_1)} \lesssim \norm{Df_k}_{L^2(B_1)}\lesssim 1$, we may fix the sequence of projections $\chi_k$, as in  \eqref{eq.projectionmap}, and their almost inverse $\chi_k^+$. Note that under the additional assumption of a $L^2$ bound, the diameter of the $T_k$ is bounded since 
\[ |T_k| |B_1|^\frac12 \le \norm{\G(f_k,T_k)}_{L^2(B_1)} + \norm{f_k}_{L^2(B_1)} \le C \qquad \forall k\,,\]
hence we may choose $\chi_k \equiv \operatorname{Id}$ for all $k$ in such situation.

\noindent\emph{Claim 1: \footnote{This part is only needed if there is no additional $L^2$ bound. In the case of an additional $L^2$ bound we have $\chi_k=\operatorname{Id}$ and therefore $\hat{f}_k=f_k$.} } $\hat{f}_k=\chi_k \circ f_k$ is uniformly bounded in $W^{1,2}(B_2)$ and moreover 
\begin{align}\label{eq.smallenergyerror}
    &\lim_k\norm{\G(\chi_k^+\circ\hat{f}_k, {f}_k)}_{L^2(B_{3})}+ \norm{|D\chi_k^+\circ \hat{f}_k|-|D{f}_k|}_{L^2(B_{3})}=0\\
    &\left(\fint_{B_r(x)} |D\hat{f}_k|^{2p}\,dx\right)^{1/p} \lesssim \fint_{B_{2r}(x)} |D\hat{f}_k|^2 + o(1) \quad \forall x \in B_3, r<\frac12\,.
\end{align}

\smallskip

\noindent\emph{Proof of Claim 1: }
To prove the claim we define the open sets 
\begin{equation}\label{eq.Uk-faraway}
    U_k=\{ x \colon \G(f_k(x), T_k)> s_k \} \cap B_4 \,.
\end{equation}
Note that if $x \notin U_k $ then 
\[ \chi_k^+ \circ \hat{f}_k(x) = f_k(x)\,.\]
And we can estimate 
\begin{equation}\label{eq.estimateU_k}
    |U_k|^{1/{2^*}} \lesssim s_k^{-1} \norm{Df_k}_{L^2(B_4)}\lesssim s_k^{-1} \lesssim \frac12|B_1|\,.
\end{equation}
Thus $T_k$ works as well as an average for $\chi_k^+\circ \hat{f}_k$. 
Hence we deduce 
\begin{align*}
    &\lim_{k} \norm{\G(\chi_k^+\circ\hat{f}_k, f_k)}_{L^2(B_3)}+ \norm{|D\chi_k^+\circ \hat{f}_k|-|Df_k|}_{L^2(B_3)}\\
    &\le \lim_{k}\norm{\G(\chi_k^+\circ\hat{f}_k, T_k)}_{L^2(U_k)}+ \norm{\G( f_k,T_k)}_{L^2(U_k)}+\norm{|D\chi_k^+\circ \hat{f}_k|-|Df_k|}_{L^2(U_k)}\\
    %&\lesssim |U_k|^{1/2-1/p}\norm{Df_k}_{L^p(B_3)} \lesssim |U_k|^{1/2-1/p}\norm{Df_k}_{L^2(B_4)}{\color{red}\lesssim s_k^{1/p-1/2}}\,.\\
    &\lesssim |U_k|^{\frac12-\frac1{2p}}\norm{Df_k}_{L^{2p}(B_3)} \lesssim |U_k|^{\frac12-\frac1{2p}}\norm{Df_k}_{L^2(B_4)}\lesssim s_k^{\frac1{2p}-\frac12}\to 0\,.
\end{align*} 
where in the second to last inequality we used Poincare and Holder inequality and in the last inequality we used Chebyshev inequality.
It remains to show the uniform boundedness and  almost higher integrability.
We have
\begin{align*}
    \norm{D\hat{f}_k}_{L^2(B_3)}&\lesssim \norm{Df_k}_{L^2(B_3)}\\
    \norm{\hat{f}_k}_{L^2(B_3)} &\le \norm{\G(\hat{f}_k, \chi_k(T_k))}_{L^2(B_3)} + |\chi_k(T_k)|\,|B_3|^{1/2}\\
    &\lesssim \norm{Df_k}_{L^2(B_3)} + |\chi_k(T_k)|\,.
\end{align*}
Note that \eqref{finitediam2} ensures that $\limsup_k |\chi_k(T_k)|<\infty$\,.
Concerning the higher integrability we have for any $x\in B_2, r<1$
\begin{align}\label{eq.almosthigherintegrability1}r^{-n/{2p}}\norm{D\hat{f}_k}_{L^{2p}(B_r(x))} &\lesssim r^{-n/2} \norm{Df_k}_{L^2(B_{2r}(x))} \\\label{eq.almosthigherintegrability2} &\lesssim r^{-n/2} \norm{D\hat{f}_k}_{L^2(B_{2r}(x)\setminus U_k)} + r^{-n/2}|U_k|^{1/2-1/2p} \norm{Df_k}_{L^2(B_3)}\end{align}

\medskip

\noindent\emph{Claim 2:} The limit $\mathcal{E}_{\hat{f}_k} \rightharpoonup \mathcal{F}$ is a continuous classical measure solution, i.e. $\mathcal{F}=\mathcal{E}_{f}$ for some $f\in W^{1,2}(B_3, \I{Q}(\R^{n+1}))\cap C^{0,\alpha}(B_3,\Iqn)$, satisfying c2) and such that $\pi_0\circ f \equiv \sum_{j} Q_j\a{j}$

\smallskip

\noindent\emph{Proof of Claim 2.} To prove the claim, let $\mathcal{F}$ be the generalised gradient young measure generated by $\mathcal{E}_{\hat{f}_k}$, which exists due to the uniform $W^{1,2}$-bounds for $\hat{f}_k$ established in Claim 1. Furthermore we denote with $f$ the $W^{1,2}$-map associated to $\mathcal{F}$.

We start by using the higher integrability to show that $F^\infty\equiv 0$, so that in particular later we will need to test only with compactly supported functions. Although the argument is classical in the context of measurable functions and Young measures, we present it since we deal with currents.

We may choose $\varphi$ as in the definition of \eqref{eq.definitionD} and let $\sigma \in C^1$ non-decreasing such that $\sigma(t)=0$ if $t< 1$ and $\sigma =1 $ for $t>2$.
For each $x_0 \in B_2$ and any real number $m_0>0$ the function $\phi(x,M)=\varphi(x) \sigma(\frac{|M|}{m^2_0})M^{\alpha\alpha}_{jj}$ is admissible, that is it is an element of $\mathcal{C}(\Omega \times \mathbb{V})$. This implies 
\begin{align*}
    \mathcal{F}(\phi) &= \lim_k \mathcal{E}_{\hat{f}_k}(\phi) \le \lim_k |\{x\in B_3 \colon |D\hat{f}_k|>m_0\}|^{1-1/p} \norm{ D\hat{f}_k}_{L^{2p}(B_1(x_0))}^2 \\
    &\lesssim \lim_k |\{x\in B_3 \colon |D\hat{f}_k|>m_0\}|^{1-1/p}\lesssim m_0^{2/p-2}\,,
\end{align*}
where we have used \eqref{eq.almosthigherintegrability1} in the last inequality. Hence sending $m_0\to \infty$ we conclude that $F^\infty=0$.
Now, by approximation with compactly supported functions, we are allowed to use as well function that grow faster than linear in $M$. In particular, we set $\phi_r(x-x_0,M)=\varphi(\frac{|x-x_0|}{r})|M|$, and we observe that \eqref{eq.almosthigherintegrability1} reads for any $x_0 \in B_2, r<1/2$ 
\[ r^{-n/p} \left(\mathcal{E}_{\hat{f}_k}(\phi^p_{r}(x-x_0,M))\right)^{1/p} \lesssim r^{-n} \mathcal{E}_{\hat{f}_k}(\phi_{5r/2}(x-x_0,M))+ r^{-n} |U_k|^{2-2/p}\,.\]
Now we can take $k\to \infty$ and obtain the desired higher integrability inequality
\begin{equation}\label{eq.higherintegrability}
   r^{-n/p} \left(\mathcal{F}(\phi^p_{r}(x-x_0,M))\right)^{1/p} \lesssim r^{-n} \mathcal{F}(\phi_{5r/2}(x-x_0,M)) \,,
\end{equation}
that is property c2).

Next we check that inner and strong outer variations of $f$ are zero. We already know that $F^\infty\equiv 0$.

\noindent\emph{Outer-variations: $e_0$-direction.}
Since we have by definition that $\pi_0(\hat{f}_k)= \sum_j Q_j \a{j}$ we deduce that $\pi_0(f)=\sum_j Q_j \a{j}$ and moreover
$$\sum_l e_0 D(f_l)_k^t D(f_l)_k e_0 =0\qquad \text{a.e}\,.
$$
 Using the test-function $\phi(x,M)=\varphi(x) M^{00}_{ii}$, we have 
$\mathcal{F}(\phi)=\lim_{k} \mathcal{E}_{\hat{f}_k}(\phi) =0$. Since $\mathcal{F}$ is supported on the non-negative matrices we deduce that \[\mathcal{F}( p_i^0\partial_i\varphi(x,y) + M^{0\alpha}_{ii} \partial_{y^\alpha}\varphi(x,y)) =0 \quad \forall \varphi(x,y) \in C^1_c(B_2\times \R^n)\,.\]

\noindent\emph{Outer-variations: the other directions.}
To a given a vector field $\psi \in C^1_c(B_2\times \R^{n+1}, \{0\}\times\R^n)$, we associate $\tilde{\psi}_k \in C^1_c(B_2\times \R^n, \R^n)$ by 
\[ \tilde{\psi}_k(x,y) =  \sum_{j=1}^N \psi(x,(j, y-p_k^j))\,.\]
Firstly for a.e. $x \in B_2 \setminus  U_k$, with $U_k$  defined in \eqref{eq.Uk-faraway} assuming that $s_k$ sufficient large, we have 
\begin{align*}
 &\sum_{l=1}^Q \partial_i (f_k)^\alpha_l \partial_i\tilde{\psi}^\alpha_k(x,(f_k)_l) +  \partial_i (f_k)^\alpha_l \partial_k (f_k)_l^\beta \partial_{y^\beta}\tilde{\psi}_k^\alpha(x, (f_k)_l) \\
 =&\sum_{l=1}^Q \partial_i (\hat{f}_k)^\alpha_l \partial_i\psi^\alpha(x,(\hat{f}_k)_l) +  \partial_i (\hat{f}_k)^\alpha_l \partial_k (\hat{f}_k)_l^\beta \partial_{y^\beta}\psi^\alpha(x, (\hat{f}_k)_l)  
\end{align*}
Thus we can estimate 
\begin{align*}
    &\left|\int \sum_{l=1}^Q \partial_i(\hat{f}_k)_l^\alpha(x)\, \partial_i( \psi^\alpha(x,(\hat{f}_k)_l(x))) \,dx\right| \le \underbrace{\left|\int_{B_2\cap U_k} \sum_{l=1}^Q \partial_i(\hat{f}_k)_l^\alpha(x)\, \partial_i( \psi^\alpha(x,(\hat{f}_k)_l(x))) \,dx\right|}_{a_k}\\& + \underbrace{\left|\int_{B_2\cap U_k} \sum_{l=1}^Q \partial_i(f_k)_l^\alpha(x)\, \partial_i( \tilde\psi_k^\alpha(x,(f_k)_l(x))) \,dx\right|}_{b_k} + \underbrace{\left|\int \sum_{l=1}^Q \partial_i(f_k)_l^\alpha(x)\, \partial_i( \tilde\psi_k^\alpha(x,(f_k)_l(x))) \,dx\right|}_{c_k}
\end{align*}
Since $c_k\to 0$ by assumption c3), it remains to estimate $a_k$ and $b_k$:
\begin{align*}
    a_k \le\norm{D\psi}_\infty \int_{B_2\cap U_k} (|D\hat{f}_k|+|D\hat{f}_k|^2) \le \norm{D\psi}_\infty \left(|U_k|^{1/2} + |U_k|^{1-1/p}\right)\int_{B_3} |D{f}_k|^2\, dx \to 0
\end{align*}
where we have used the higher integrability \eqref{eq.almosthigherintegrability1}.
Very similarly we can estimate $b_k$ using that $\norm{D_x\tilde{\psi}_k}_\infty \le \norm{D_x\psi}_\infty $ and $\norm{D_y\tilde{\psi}_k}_\infty \le \norm{D_y\psi}_\infty $ so that
\begin{align*}
    b_k \le \norm{D\psi}_\infty \int_{B_2\cap U_k} |Df_k| + |Df_k|^2 \,dx \le  \norm{D\psi}_\infty \left(|U_k|^{1/2} + |U_k|^{1-1/p}\right)\int_{B_3} |D{f}_k|^2\, dx \to 0\,.
\end{align*}
Now using that $\mathcal{E}_{\hat{f}_k}$ generates $\mathcal{F}$ we conclude in combination with the outer variation in direction $e_0$.
\[ \mathcal{F}( p_i^\alpha \partial_i\psi^\alpha(x,y) + M^{\alpha\beta}_{ii} \partial_{y^\beta}\psi^\alpha(x,y)) =0 \quad \forall \psi \in C^1_c(B_2\times \R^{n+1}, \R^{n+1})\,.\]

\medskip 
\noindent\emph{Inner variations:}
As observed above that for a.e. $x \in B_3\setminus U_k$
\[ \sum_{l=1}^Q D(\hat{f}_k)_l(x)^t D(\hat{f}_k)_l(x) = \sum_{l=1}^Q D({f}_k)_l(x)^t D({f}_k)_l(x)\,.\]
Given any $\phi \in C_c^1(B_2,\R^m)$ we can estimate
\begin{align*}
    &\left|\int \sum_{l=1}^Q \left(\frac12 |D(\hat{f}_k)_l|^2 \delta^i_j - \partial_i (\hat{f}_k)_l \partial_j (\hat{f}_k)_l \right) \partial_i\phi^j\, dx\right| \le \underbrace{ \int_{B_2\cap U_k} |D\hat{f}_k|^2 |D\phi| \, dx}_{a_k}\\&+\underbrace{\int_{B_2\cap U_k} |D{f}_k|^2 \, dx}_{b_k} + \underbrace{\left| \int \sum_{l=1}^Q \left(\frac12 |D({f}_k)_l|^2 \delta^i_j - \partial_i ({f}_k)_l \partial_j ({f}_k)_l \right) \partial_i\phi^j\, dx \right|}_{c_k} \,.
\end{align*}
The first two terms are estimated as in the outer variation:
\[a_k+b_k \lesssim \norm{D\phi}_\infty |U_k|^{1-1/p} \int_{B_3} |D\tilde{f}_k|^2\, dx \to 0\,,\]
while $c_k\to 0$ by assumption c3), so that $\mathcal{E}_{\hat{f}_k}$ generates $\mathcal{F}$ so that 
\[ \mathcal{F}\left( \left(\frac12 |M|\delta_i^j - M^{\alpha \alpha}_{ij}\right)\,\partial_i\phi^j(x) \right)=0 \quad \forall \phi \in C^1_c(B_2,\R^m)\,. \]
Hence we can appeal to Proposition \ref{prop:higofmeas} and thereafter to Proposition \ref{prop:regformeas} to deduce the claim. 

\medskip
\noindent\emph{Claim 3:} We have 
\[\lim_k \norm{\G({f}_k, \chi^+_k\circ f)}_{L^2(B_2)}+\norm{|D\tilde{f}_k|-|D f|}_{L^2(B_2)}=0\,.\]

\smallskip

%\noindent\emph{Proof of Claim 3:}
%{\color{red} There is a problem - perhaps we should go back to your proposal to keep the different $\chi_k^j\circ \tilde{f}_k \ominus p_k^j$ separate.}
 
\noindent \emph{Proof of Claim 3.} This claim is a direct consequence of the the two previous claims. Firstly we observe that if $\hat{T}, \hat{S} \in \pi_0^{-1}(\sum_j Q_j\a{j})$ then we have 
\begin{align*}
   \G(\chi_k^+(T), \chi_k^+(S)) &= \G(T,S) \quad \text{ if } \G(T,S) \le \frac12\,,\\
   \G(\chi_k^+(T), \chi_k^+(S)) &\le |T| + |S|,
\end{align*}
where the second inequality follows from the triangle inequality with middle point $\chi_k(P_k)= \sum_{j} Q_j \a{je_0}$. 

We clearly have 
\[\norm{\G({f}_k, \chi_k^+\circ {f})}_{L^2(B_3)}\le \underbrace{\norm{\G({f}_k,\chi_k^+\circ \hat{f}_k)}_{L^2(B_3)}}_{a_k}+\underbrace{\norm{\G(\chi_k^+\circ\hat{f}_k, \chi_k^+\circ {f})}_{L^2(B_3)}}_{b_k}\,.\]
The first claim implies that $a_k\downarrow 0$. We let $V_k= \{ x \colon \G(\hat{f}_k, f) > 1/2 \}$ and observe that $|V_k|^{1/2^*}\le 2 \norm{\G(\hat{f}_k, f)}_{L^2(B_3)} \to 0$. Hence we have the $L^2$-convergence 
\[
    b_k \le \norm{\G(\hat{f}_k, f)}_{L^2(V_k^c)} + |V_k|^{1/n}\left(\norm{\hat{f}_k}_{L^{2^*}(B_3)}+\norm{f}_{L^{2^*}(B_3)}\right) \to 0\]
Now we come to the $L^2$-convergence of the gradient. The first part is the same 
\[\norm{|D{f}_k|-|D f|}_{L^2(B_2)}\le \underbrace{\norm{|D{f}_k|-|D\chi^+_k\circ \hat{f}_k|}_{L^2(B_2)}}_{\tilde{a}_k}+\underbrace{\norm{|D\chi^+_k\circ\hat{f}_k|-|D f|}_{L^2(B_2)}}_{\tilde{b}_k}\,.\]
As before the $\tilde{a}_k\downarrow 0$ due to claim 1.
Estimating $\tilde{b}_k$ is even simpler since $|D\chi_k^+\circ \hat{f}_k|=|D\hat{f}_k|$ because $\pi_0\circ \hat{f}_k\equiv \sum_j Q_j\a{j}$. The same argument applies to $f$. Thus, since by Claim 2 $\mathcal E_{\hat{f_k}}\to \mathcal E_{f}$, we conclude
\[\tilde{b}_k = \norm{|D\hat{f}_k|-|D f|}_{L^2(B_2)} \to 0\,.
\] 

\qed

\section{Almgren's Strong approximation}\label{ss:almapp}

In this section we prove a graphical approximation result with superlinear error in the excess and  a small Lipschitz constant, as first proven by Almgren for area minimizing currents in \cite{Alm} and later revisited by De Lellis and Spadaro in \cite{DS1}. Like in \cite{Alm,DS1}, the key step is to prove a higher integrability estimate for the excess. However our proof is fundamentally different from all the previous ones in that we cannot rely on area minimality, that is on the construction of suitable competitors. Our approach is based on a variant of Gehring's lemma, and one of the key ingredient is Poincar\'e inequality at collapsed points.  

Following \cite{DS1, DLS_Center, DLS_Blowup}, we will denote with $\pi_0:=\R^m\times \{0\}$. Open balls in $\R^{m+n}$ will be denoted by $\bB_r(p)$. For any linear subspace $\pi\subset \R^{n+m}$, $\pi^\perp$ is its orthogonal complement, $p_\pi$ the orthogonal projection onto $\pi$, $B_r(q,\pi)$ the disk $\bB_r(q)\cap(q +\pi)$ and $\bC_r(p, \pi)$ the cylinder $\{(x+y) \,:\, x \in \bB_r(p)\,,\,\, y\in \pi^\perp\}$ (in both cases $q$ is omitted if it is the origin and $\pi$ is omitted if it is clear from the context or if $\pi=\pi_0$), so
\[
\bC_{r}(x):=\bC_r(x,\pi_0)=B_{r}(x)\times \R^n\,.
\]
We also assume that each $\pi$ is oriented by a $k$-vector $\vec{\pi}:= v_1\wedge\dots\wedge v_k$ (thereby making a distinction when the same plane is given opposite orientations).

We need some notations for integral currents $T\in \bI_{m}(\R^{m+n})$:
\begin{itemize}
    \item $\Theta(T, x)$ will denote the density of the current $T$ at the point $x$;
    \item $\bE(T, A, \pi):=(\omega_m r^m)^{-1} \int_A\left|\vec{T}-\vec{\pi} \right|\,d\|T\|$, where $A=\bB_r(x)$ or $A=\bC_r(x,\pi')$, will denote the excess of the spherical and cylindrical excess of current respectively;
    \item $\bh(T, A, \pi):=\sup_{x,y\in \spt(T)\cap A}\left|\p_{\pi^\perp}(x)-\p_{\pi^\perp}(y) \right|$ will denote the height of the current in the set $A\subset \R^{n+m}$.
\end{itemize}

\begin{theorem}[Almgren's strong approximation]\label{thm:almstr}
    There exist constants $C, \gamma, \eps>0$, depending on $m,n$ with the following property. Assume that $f\colon \Omega=B_{4r}(x) \to \Iq(\R^n)$ is a Lipschitz function with stationary graph, with $Q=2$, and suppose that $E=\bE(\bG_f, \bC_{4r}(x))<\eps$. Then there is a map $\hat{f}\colon B_r(x)\to \I2(\R^n)$ and a closed set $K\subset B_r(x)$ such that
    \begin{gather}
        \Lip(\hat{f})\leq C\, E^\gamma \,,\label{e:lipbd}\\
        \bG_{\hat{f}}\res (K\times \R^n)=\bG_{f}\res (K\times \R^n)\quad\text{and}\quad |B_r(x)\setminus K|\leq C\, E^{1+\gamma}\,r^m\,, \label{e:graphcoincide}\\
        \left|\|\bG_f\|(\bC_{\sigma r}(x))-2\,\omega_m\,(\sigma r)^m-\frac12\int_{B_{\sigma r}(x)}|D\hat{f}|^2 \right|\leq C\, E^{1+\gamma}\,r^m\quad\forall 0<\sigma\leq 1\,,\label{e:areaenergy}\\
        {\rm osc}_{B_r(x)}(\hat{f}):=\inf_p \sup_{y\in B_r(x)} \G(\hat{f}(y),2\a{p})\leq C \,\bh(\bG_f,\bC_{4r}(x), \pi_0)+C\, E^{\frac12}\,,\label{e:oscillation}\\
        \frac1{r^2}\int_{B_r(x)}\G(f,\hat{f})^2+\int_{B_r(x)}\left(|D f|-|D\hat{f}|\right)^2+\int_{B_r(x)}\left|D(\etab\circ f)-D(\etab\circ \hat{f})\right|^2\leq C\,E^{1+\gamma}\,r^m \,.\label{eq:approxerror}
    \end{gather}
\end{theorem}

As a nontrivial consequence of this theorem and the regularity result in the linear theory section, we will also deduce a strong harmonic approximation result. 

\begin{theorem}[Harmonic approximation]\label{cor:harmapprox}
Let $\gamma$ be the constant of Theorem \ref{thm:almstr}. Then, for every $\eta>0$, there is a positive constant $\eps>0$ with the following property. Assume that $f$ is as in Theorem \ref{thm:almstr}, $E := \bE(\bG_f, \bC_{4r}(x))<\eps$, then there exists a continuous classical solution $u\in W^{1,2}(B_r(x), \I2(\R^n))\cap C^{0,\alpha}(B_r(x), \I2(\R^n))$ such that
\begin{equation}\label{eq:harmonicapp}
    \frac1{r^2}\int_{B_r(x)}\G(f,u)^2+\int_{B_r(x)}\left(|D f|-|Du|\right)^2+\int_{B_r(x)}\left|D(\etab\circ f)-D(\etab\circ u)\right|^2\leq \eta\,E\,r^m\,.
\end{equation}
\end{theorem}

\begin{remark}
    Notice that, thanks to \eqref{eq:approxerror}, we can replace $f$ with $\hat{f}$ in \eqref{eq:harmonicapp} thus obtaining the same statement as in \cite[Theorem 2.6]{DS1}
\end{remark}

We will also need the following persistency of the $Q$-point result:

\begin{proposition}[Persistency of $2$-points]\label{prop:basicpersistency} Let $f\colon \Omega=B_{4r}(x) \to \I2(\R^n)$ be a Lipschitz function with stationary graph, and suppose that there exists $y_0\in \Omega$ such that $f(y_0)=2 \a{t}$ for some $t\in \R^n$.  Then there exists a constant $C$, depending on the Lipschitz constant of $f$, such that 
\begin{equation}\label{eq:basicper}
\sup_{x\in B_r(y_0)}\G(f(x), 2\a{t}) \leq C\, r^{2+m} \int_{B_{4r(x_0)}}|Df|^2\,,\qquad \forall B_{4r}(x_0)\subset \Omega\,. 
\end{equation}
Moreover, under the assumptions of Theorem \ref{thm:almstr}, we can replace $f$ with $\hat f$ in \eqref{eq:basicper}.
\end{proposition}

The key estimate to derive both of the above results is the following higher integrability result for the gradient of $f$.  We remark that it is in proving this result that we crucially use the assumption $Q=2$.  

\begin{theorem}[Higher integrability]\label{thm:high}
There exist $p>1$ and a constant $C=C(\Lip(f))> 0$ such that if  $f\colon B_1\to \I2(\R^n)$ is a Lipschitz function with stationary graph, then 
\begin{equation}
    \left( \int_{B_{\frac12}} |Df|^{2p} \right)^\frac1p \le C \int_{B_1} |Df|^2\,. 
\end{equation}
\end{theorem}

\subsection{Preliminaries on multivalued Lipschitz functions} Given a Lipschitz function $f\colon \Omega \to \R^n$, we consider the following quantities:
\[
g_{ij}(f):=\delta_{ij}+\partial_if\cdot \partial_jf\,,
\qquad g^{ij}(f):=(g_{ij}(f))^{-1}
\qquad \text{and}\qquad 
|g(f)|=\det(g_{ij}(f))\,.
\]
We notice that if $f=\sum_{l=1}^Qf_l$ is a Lipschitz function, then by \cite{DS} the quantities $g_{ij}(f_l)$ and $|g(f_l)|$ are well defined almost everywhere, and moreover the following estimates hold for every $l=1,\dots Q$ and almost every $x\in \Omega$:
\begin{equation}\label{eq:lip1}
    \delta_{ij}\le g_{ij}(f_l) \le \delta_{ij}(1+\Lip(f)^2)
\end{equation}
\begin{equation}\label{eq:lip2}
    \frac{1}{(1+\Lip(f)^2)^{\frac 12}} \le 
\sqrt{|g(f_l)|}g^{ij}(f_l) \le (1+\Lip(f)^2)^{\frac{m-1}{2}}. 
\end{equation}

\subsection{Higher integrability: proof of Theorem \ref{thm:high}}
The goal of this section is to prove the following higher integrability result. Before coming to the proof of Theorem \ref{thm:high}, we need some preliminary results. We start with a lower density estimate at collapsed points on the graph.

\begin{lemma}[Lower density estimate]
    Let $f$ be as Theorem \ref{thm:high} and $f(0)=2\a{t}$, then 
    \[\Theta(\bG_f,(0,t)) \ge 2\,.\]
\end{lemma}
\begin{proof}
Passing to $f\ominus t$ we can assume without loss of generality that $t=0$. Furthermore, by the monotonicity formula of area for the graph of $f$, we know that that the limit of the sequence
\[ 
(\eta_{r})_\sharp \bG_f = \bG_{f_r}\,,
\]
with $f_r(x) = \frac{1}{r}f(rx)$ and $\eta_r(x)=x/r$, is a $1$-homogeneous Lipschitz graph $\hat{f}$ that is stationary and satisfies for all $r\ge 0$
\[ 
\Theta(\bG_f, 0) = \Theta( \bG_{\hat{f}} , B_r)\,.
\]

Hence it is sufficient to prove the claim under the additional assumption that $f$ is 1-homogeneous. We now apply a dimension reduction argument. 
If $f$ is $1$-dimensional and $1$-homogeneous, we easily deduce that $f(t)= \a{A_1t} + \a{A_2t}$ for two vectors $A_i \in \R^n$, hence $\Theta(\bG_f, 0)=2$.

Now we come to the induction step. Assume we have proven the claim for $m$, and $f$ is a $1$-homogeneous stationary Lipschitz function in dimension $m+1$. 

Either there is $x_0 \in \partial B_1$ s.t. $f(x_0)= 2\a{t}$, in which case we deduce from the upper semi-continuity of the density that 
\[\Theta(\bG_f,0) \ge \Theta(\bG_f,(x_0,t)) =\Theta(\bG_{Tf_{x_0}},0)\ge 2\,.\]
where $Tf_{x_0}$ denotes the blow-up at the point $(x_0,t)$, whose graph is stationary and splits a line. Here we have used the following observation. Let $f_{x_0,r}(y)=\frac{f(x_0+ry)-f(x_0)}{r}$ be the blow-up sequence at the point $x_0$. Then a sub-sequence $f_k=f_{x_0,r_k}$ convergences as $r_k\to 0$ to a Lipschitz function $Tf_{x_0}$. On the other hand, the varifolds associated the graphs of $f_k$, $V_k={\eta_{x_0,r_k}}_\sharp V$, converge in the varifold sense to a stationary varifold $V_\infty$. This stationary varifold is a cone due to the monotonicity formula, and even more, it splits a line (see for instance \cite{Whstrat}). Since $\mathbf{G}_{f_k}$ convergence in Hausdorff $\spt(V_\infty) \subset \spt(\mathbf{G}_{Tf})$. By the Lipschitz graphicality, the density cannot drop, hence we deduce that $\mathbf{G}_{Tf}$ has $V_\infty$ as a stationary varifold, compare \cite[Proposition 11.53]{GiaquintaMartinazzi}. This implies that $Tf$ is stationary, splits a line and hence we can apply the induction hypothesis.

If there is no collapsed point on $\partial B_1$ we have to distinguish between dimensions $m=1$ and $m\ge 2$.

For $m\ge 2$ the $m$-sphere $\partial B_1$ is simply connected and hence $f(x)= \a{f_1(x)} + \a{f_2(x)}$ for two Lipschitz functions $f_i\colon \partial B_1 \to \R^n$ satisfying $f_1(x)\neq f_2(x)$ for all $x \in \partial B_1$. The $1$-homogeneity of $f$ implies that each $\bG_{f_i}$ is itself stationary, where we have identified $f_i$ with their $1$-homogeneous extension. (In particular, this implies that $f_i$ are $1$-homogenous stationary $1$-valued Lipschitz graphs and therefore linear functions.) Hence the result follows.

If $m=1$ there are two possibilities: either the link is disconnected, in which case the same argument as above applies, or it is connected. In the latter case we would have a $1$-dimensional connected minimal curve in $\partial B_1^{m+n}$, which therefore must be a single geodesic. This contradicts the fact that $f$ is a $2$-valued Lipschitz graph. 
\end{proof}

The key estimate to prove Theorem \ref{thm:high} is contained in the following lemma.

\begin{lemma}[Key estimate]\label{l:key}
    Let $L>0$. There exist dimensional constants $C,M> 0$, depending on $L$, such that if  $f$ is as in Theorem \ref{thm:high} and $\Lip(f)\leq L$, and $B_{Mr}(x_0)\subset B_1$, then 
\begin{equation}
\mint_{B_r(x_0)} |Df|^2\leq \left(\mint_{B_{Mr}(x_0)}|Df|^2\right)^{\frac12}\,\mint_{B_{Mr}(x_0)}|Df|\,.
\end{equation}
\end{lemma}

\begin{proof} After scaling and translating we can assume that $r=1$ and $x_0=0$. Next, we distinguish two cases.

\smallskip

\noindent\emph{Case 1:} If $\sep(f(x))>0$ for every $x\in B_2$, then, since $B_2$ is simply connected, there are two Lipschitz functions $f_l\colon B_2\to \R^n$ such that $f_1(x)\neq f_2(x)$ and $f(x)=\sum_{l=1}^2\a{f_l(x)}$ for all $x\in B_2$. By Corollary \ref{cor.LinftyL2} (or standard elliptic estimates), together with Poincare inequality, we have
\begin{equation}\label{eq:linfl2}
\|f_l-(f_l)_1\|^2_{L^{\infty}(B_{\frac32})}\leq C\, \int_{B_2}|f_l-(f_l)_1|^2\leq C\,\int_{B_2} |Df_l|^2\,.
\end{equation}
Since $\mathcal S(f_l,\cdot)\equiv0$ for both $l=1,2$, we can test it with a function $\psi(x,u)=\varphi(x) (u-(f_l)_r)$, with $\varphi\equiv 1$ on $B_1$ and $0$ outside $B_{\frac32}$,    to obtain
\[
\int \varphi \sqrt{|g(f_l)|} g^{ij}(f_l) \,\partial_if_l^\alpha\partial_j f_l^\alpha-\int \partial_j\varphi\sqrt{|g(f_l)|} g^{ij}(f_l) \,\partial_if_l^\alpha (f_l-(f_l)_1)^\alpha=0\,,
\]
which using the bounds \eqref{eq:lip1} and \eqref{eq:lip2}, the estimate \eqref{eq:linfl2}, and Cauchy-Schwarz inequality, implies that
\[
\frac{1}{(1+L^2)^{\frac12}}\int_{B_1}|Df_l|^2\leq C\,\|f_l-(f_l)_1\|_{L^\infty(B_{\frac32})}\,\left(\int_{B_2} |Df_l|\right)\leq C\,\left(\int_{B_2}|Df_l|^2\right)^{\frac12}\,\int_{B_2} |Df_l| \,.
\]
Summing over $l$ the result follows, with $M=2$.
\smallskip 

\noindent\emph{Case 2:} Suppose there exists a point $y_0\in B_2$ such that $f(y_0)=2\a{t}$, for some $t\in \R^n$. Notice that by the Lipschitz continuity of $f$ we have 
\[
\bG_f\res \bC_4(y_0)\subset \bB_{4(1+L)}(y_0)\,.
\]
By the monotonicity formula for $\bG_f$, denoting with $p_0:=(y_0,t)$, we have
\begin{align}\label{eq:mono}
\int_{\bB_{4(1+L)}(p_0)} \frac{|(p-p_0)^\perp|^2}{|p-p_0|^{m+2}}\,d\|\bG_f\|
    &\leq \frac{\|\bG_f\|(\bB_{4(1+L)}(p_0))}{(4(1+L))^m\,\omega_m}-\Theta(\bG_f,p_0)\nonumber\\
    &\leq \frac{1}{(4(1+L))^m\,\omega_m}\left(\|\bG_f\|(\bC_{4(1+L)}(x_0))-2  \right)
\end{align}
where in the second inequality we used Lemma \ref{l:key}.

We can bound the left hand side in \eqref{eq:mono}  from below by
\begin{align*}
    &\int_{\bB_{4(1+L)}(p_0)} \frac{|(p-p_0)^\perp|^2}{|p-p_0|^{m+2}}\,d\|\bG_f\| \\&\ge \frac{1}{((1+L)4)^{m+2}} \int_{\bB_{4(1+L)}(p_0)\cap \bC_4(y_0)} |(p-p_0)^{\perp_0}|^2\, d\|\bG_f\| - C \int_{\bB_{4(1+L)}(p_0)\cap \bC_4(y_0)} |\pi - \pi_0|^2\, d\|\bG_f\|
\end{align*}
Expressing the geometric inequality above in terms of the Lipschitz function using
\begin{align*}
    \frac12|\pi(f_l)-\pi_0|^2=\tr(\pi(f_l)\pi_0^\perp)=g(f_l)^{ij} \partial_if_l\cdot \partial_jf_l\\
    \sqrt{|g(f_l)|}-1 \le \frac12 \sqrt{|g(f_l)|}|Df_l|^2
\end{align*}
we found, using once again \eqref{eq:lip1} and \eqref{eq:lip2},
\begin{equation*}
    \int_{B_4(y_0)}|f\ominus t|^2 \le C \int_{B_{4(L+1)}(y_0)} |Df|^2 
\end{equation*}
Combining the inequalities above, with the estimate in Corollary \ref{cor.LinftyL2}, we obtain
\begin{equation}\label{eq:basi}
\sup_{B_{\frac72}(y_0)}|f\ominus t|^2\leq C\,\int_{B_{4(L+1)}(y_0)} |Df|^2\,. 
\end{equation}
Finally, testing the outer variation for $f$ with the vector field $\psi(x,u)=\varphi(x)\,(u-t)$, we obtain
\[
\int \varphi \sum_{l=1}^2\sqrt{|g(f_l)|} g^{ij}(f_l) \,\partial_if_l^\alpha\partial_j f_l^\alpha-\int \partial_j\varphi\sum_{l=1}^2\sqrt{|g(f_l)|} g^{ij}(f_l) \,\partial_if_l^\alpha (f_l-t)^\alpha=0\,,
\]
Letting $\varphi$ be as in the previous case, with the same reasoning we obtain
\[
\int_{B_1}|Df|^2\leq C \, \left( \int_{B_{5(L+1)}}|Df|^2\right)^{\frac12}\, \int_{B_{5(L+1)}} |Df|\,,
\]
which concludes the proof with $M=5(L+1)$.

\end{proof}

\begin{proof}[Proof of Theorem \ref{thm:high}]
The result follows immediately applying Gehring's lemma, see \cite[Section 4.2]{HS}. 
\end{proof}

\subsection{Almgren's approximation: proof of Theorem \ref{thm:almstr} and Proposition \ref{prop:basicpersistency}}

Given the higher integrability of the previous subsection, the proof follows by a standard modification of \cite[Theorem 6.14]{EG}: we provide some details for the reader's convenience. 

Following the argument of \cite[Theorem 6.14]{EG} with $g=|Df|^2$, we have that if we set
\[
R^\delta:=\left\{y\in B_2\,:\,\mint_{B_r(y)}|Df|^2\leq \delta \,,\forall r<1  \right\}\,,
\]
then
\[
|B_2\setminus R_2^\delta|\leq \frac{C}{\delta} \, \int_{B_3\cap \{|Df|^2>\delta/2\} }|Df|^2
\]
so that, for $\delta=E^{2\gamma}$ and using Young's inequality, we conclude
\begin{align*}
|B_1\setminus R^\delta|
   &\leq \frac{C}{E^{2\gamma}} |B_3\cap \{M(|Df|^2)>2^{-m}\,E^{2\gamma}\}|^{1-\frac1p} \,\left(\int_{B_3}|Df|^{2p}\right)^{\frac1p}\\
    &
    \leq C\,  E^{(1-2\gamma)(1-\frac1p)-2\gamma} \, E\leq C\, E^{1+\gamma}\,,
\end{align*}
up to choosing $\gamma>0$ sufficiently small. Finally setting $K:=R^{E^{2\gamma}}$ concludes the proof of \eqref{e:lipbd}, \eqref{e:graphcoincide} and \eqref{e:oscillation}, since $f|_{K}$ is Lipschitz, with $\Lip(f|_K)\leq C\, E^\gamma$.

To prove \eqref{e:areaenergy} we notice that
\begin{align*}
&\left|\|\bG_f\|(\bC_{\sigma}(x))-\,\omega_m\,\sigma^m-\frac12\int_{B_{\sigma }(x)}|D\hat{f}|^2 \right|\\
&\leq
\|\bG_f+\bG_{\hat f}\|((B_{\sigma}(x)\setminus K)\times \R^n)+\left|\|\bG_{\hat {f}}\|(\bC_\sigma(x))-\,\omega_m\,\sigma^m-\frac12\int_{B_{\sigma }(x)}|D\hat{f}|^2\right|\\
&\leq C(\Lip(f),\Lip(\hat f))\,|B_\sigma(x)\setminus K|+C\,\Lip(\hat {f})\, \int_{B_\sigma}|D\hat{f}|^2\\
&\leq C\, E^{1+\gamma}\,,
\end{align*}
for a constant $C$ depending on the Lipschitz constant of $f$, where in the second inequality we used the Taylor expansion for the area of a multivalued graph in \cite{DLS_Currents}.

To prove Proposition \ref{prop:basicpersistency} notice that \eqref{eq:basicper} follows from \eqref{eq:basi} in Case 2 of the proof of Theorem \ref{thm:high}, while the last sentence follows from \eqref{e:oscillation}.
\qed

\subsection{Harmonic approximation: proof of Theorem \ref{cor:harmapprox}} We reason by contradiction, i.e. there is a $\eta>0$ such that the claim fails, in particular there is a sequence of maps $(f_k)_k$ as in Theorem \ref{thm:almstr} failing \eqref{eq:harmonicapp} but satisfying 
\[E_k= \bE(\bG_{f_k}, C_4) \approx \int_{B_4} |Df_k|^2\, dx\to 0\,. \]
Consider the renormalised sequence $\tilde{f}_k=f_k/E_k^\frac12$. Notice that, by Theorem \ref{thm:high}, assumptions c1) and c2) of Theorem \ref{thm:compactness} hold for the sequence $(\tilde{f_k})_k$. Next we prove that c3) hold.

\medskip

\noindent\emph{Outer-variations.}
We note that \eqref{eq:lip1}-\eqref{eq:outer} implies that for any $\psi \in C^1_c(B_2\times \R^n, \R^n)$ 
\begin{align}\label{eq.outervariationestimate1}
    &\left|\int \sum_{l=1}^Q \partial_i(f_k)_l^\alpha(x)\, \partial_i( \psi^\alpha(x,(f_k)_l(x))) \,dx \right|\nonumber\\&\le \int \sum_{l=1}^Q |\sqrt{|g((f_k)_l)|}\,g^{ij}((f_k)_l)-\delta^{ij}| \left(|D(f_k)_l(x)||D_x\psi|(x,(f_k)_l(x)) + |D(f_k)_l(x)|^2|D_y\psi(x,(f_k)_l(x))|\right)\,dx\nonumber\\
    &\lesssim \int \sum_{l=1}^Q |D(f_k)_l(x)|^2 \left(|D(f_k)_l(x)||D_x\psi|(x,(f_k)_l(x)) + |D(f_k)_l(x)|^2|D_y\psi(x,(f_k)_l(x))|\right)\, dx
\end{align}
To a given a vector field $\psi \in C^1_c(B_2\times \R^{n+1}, \{0\}\times\R^n)$ we associate $\tilde{\psi}_k \in C^1_c(B_2\times \R^n, \R^n)$ by 
\[ \tilde{\psi}_k(x,y) = E_k^{-1/2} \left(\sum_{j=1}^N \psi(x,(j, \frac{y}{E_k^{1/2}})\right)\,,
\]
so that
\begin{align*}
&|\mathcal S(\tilde{f}_k, \psi)|
    =|\mathcal S(f_k,\tilde{\psi}_k)|\\
    &\quad\leq \int \sum_{l=1}^Q |D(f_k)_l(x)|^2 \left(|D(\tilde{f}_k)_l(x)||D_x\psi|(x,(\tilde{f}_k)_l(x)) + |D(\tilde{f}_k)_l(x)|^2|D_y\psi(x,(\tilde{f}_k)_l(x))|\right)\, dx\\
    &\quad\leq \norm{D\psi}_\infty \left(\Lip(f_k)^{3-2p}\,E_k^{p-\frac12}+\Lip(f_k)^{4-2p} E_k^{p-1}\right) \int_{B_2} |D\tilde{f}_k|^{2p}  \\
    &\quad \leq \|D\psi\|_{\infty}\,  \left(\Lip(f_k)^{3-2p}\,E_k^{p-\frac12}+\Lip(f_k)^{4-2p} E_k^{p-1}\right)\, \|D\tilde{f}_k\|_{L^2(B_3)}^2\to 0\,,
\end{align*}
where in the last inequality we used the higher integrability Theorem \ref{thm:high} with $p<3/2$.

\medskip
\noindent\emph{Inner variations.}
We observe that for any $L>0$ there is a constant $C=C(L)$ such that if $\Lip(f)\leq L$ then
\begin{equation}\label{eq.innervarationexpansion}
    \left|\left(\sqrt{|g(p)|}g^{ij}(p) - \delta^{ij}\right)-\left(\frac12 |p|^2 \delta_i^j - p_i^\alpha p_j^\alpha \right)\right|\le C |p|^4 \quad \forall p\in \R^{n\times m}.
\end{equation}
Let $\phi$ be an admissible inner variation with $\int \Div {\phi}\,, dx=0$, and define  $\tilde{\phi}_k= E_k^{-1} \phi$. Then using  \eqref{eq:inner} we deduce that
\begin{align*}
    |\mathcal I(\tilde f_k, \phi)|&=|\mathcal I(f_k,\tilde \phi)|\le C \norm{D\phi}_\infty \int_{B_2} E_k^{-1} |Df_k|^4 \, dx \le C \norm{D\phi}_\infty \Lip(f_k)^{4-2p} E_k^{p-1} \int_{B_2} |D\tilde{f}_k|^{2p}\,dx\\& \le C \norm{D\phi}_\infty E_k^{p-1} \norm{D\tilde{f}_k}_{L^2(B_3)}^{2p} \to 0. 
\end{align*}

\medskip

Applying Theorem \ref{thm:compactness} we reach the desired contradiction.
\qed

\section{Center Manifold and Normal approximation}\label{ss:CmandNa}

The constructions of the Center Manifold and of the Normal approximation follow in the same way as in \cite{DLS_Center}, by using the results of the previous two sections. However, some of the proofs in \cite{DLS_Center} depend on the theory of Dir-minimizing functions, which we need to replace with measure solutions. In the next subsections we will present the main changes needed for all the results in \cite{DLS_Center} to hold.

\subsection{Existence of Center Manifold and Normal approximations} All the results of \cite[Sections 1, 2, 4, 5]{DLS_Center} hold for the integral current $\bG_f$ by Theorem \ref{thm:almstr} (instead of \cite[Theorem 2.4]{DS1} therein) and the stationarity of the current. The same is true for \cite[Proposition 3.1 and Corollary 3.2]{DLS_Center}.

\subsection{Quantitative unique continuation for strong measure solutions} Since the higher integrability property for the gradient of $f$ doesn't imply that of its average free counterpart, we need to modify some additional results to prove that the remaining results in \cite[Section 3]{DLS_Center} hold.

\begin{lemma}[See {\cite[Lemma 7.1, Proposition 7.2]{DLS_Center}}]\label{lem.Lemma7.1.replacement}
    For every $\eta \in (0,1)$ and $c>0$, there exists $\gamma>0$ with the following property. If $w\colon B_{2r} \mapsto \Iq(\R^n)$ is a classical solution satisfying 
    \begin{enumerate}
        \item[a1)] $c\le \int_{B_{r}} |Dw|^2 \le \int_{B_{2r}} |Dw|^2 \le 1$,
        \item[a2)]\label{ass:uniquecontinuation-a2} $\int_{B_{(1+\lambda)r}} |D\overset{\circ}{w}|^2 \ge c \int_{B_{2r}} |Dw|^2 $,
        \item[a3)]$\left(\fint_{B_s(x)}|Dw|^{2p}\right)^{\frac1p}\leq C \,\fint_{B_{2s}(x)}|Dw|^2$, for every $B_{2s}(x)\subset B_{2r}$ and $p>1$,
    \end{enumerate}
    then
    \begin{equation}\label{eq.non-degeneracy}
       \min\left\{\frac{1}{r^2}\int_{B_s(x)}|\overset{\circ}{w}|^2, \int_{B_s(x)} |D\overset{\circ}{w}|^2 \right\}\ge \gamma \qquad\text{ for every } B_s(x)\subset B_{2r} \text{ with } s\ge \eta r\,.
    \end{equation}
\end{lemma}

\begin{proof}
The proof is by contradiction and, similarly to the one of\cite[Lemma 7.1]{DLS_Center}, it is split into two parts: the ``unique continuation'' statement and the reduction of the Lemma to it. 

\medskip

We start by proving the following unique continuation statement:

\begin{itemize}
    \item[(UC)] If $\Omega$ is a connected open set and $w \in W^{1,2}(\Omega, \Iq(\R^n))$ is a continuous classical solution, then either $w$ is constant or $\int_{J} |Dw|^2 >0$ for any open set $J \subset \Omega$.
\end{itemize} 

\emph{Proof of (UC): } Assume $J \subset \{ |Dw|=0\}$ is a non-empty, connected open set. Hence due to continuity of $w$ we must have $w=T$ on $J$ for some $T \in \Iq(\R^n)$. Let $J'$ be the interior of $\{w=T\}$ and $J\subset K=\overline{J'}\cap \Omega$. Since $K$ is closed and nonempty and $\Omega$ is connected, we can conclude by showing that $K$ is open. We will first show this for $T=Q\a{0}$, and then we will show that locally the claim can always be reduced to this case. 

\smallskip

\emph{Case $T=Q\a{0}$ :} Clearly $\mathcal{E}_w$ is a stationary measure and therefore Theorem \ref{thm:UCformeas} holds. Since $w=Q\a{0}= Q \a{\eta \circ w }$ on an open set we must have $w\equiv Q\a{\eta\circ w}$ in $\Omega$. Now we can appeal to the classical unique continuation for harmonic functions to deduce that $w\equiv Q\a{0}$ in $\Omega$.

\smallskip

\emph{Case $T=\sum_{k=1}^K Q_k \a{p_k}$: } Fix any $x_0 \in K$ and consider $f(x)=w(x_0+x)$, i.e. the interior of $B_s(0)\cap \{ f = T\}$ is non-empty for every $s>0$. We can split $f$ as in \cite[Proposition 6.1]{HS}\footnote{Notice that as written in \cite{HS} Proposition 6.1 holds only in dimension $2$, however here we are assuming $f$ to be continuous which allows us to apply it in every dimension.}: there is $r>0$ such that 
\begin{enumerate}
    \item[i)] $f=\sum_k h_k \oplus g_k$ in $B_r$,
    \item[ii)] $\tilde{f} = \sum_k g_k$ is stationary in $B_r$ with the property that for all $x \in B_r$ with $\operatorname{card}(\supp f(x))=k$ we have $\tilde{f}(x)=Q\a{0}$,
    \item[iii)] $h_k$ are a single valued harmonic function.
\end{enumerate}
Note that we can apply the previous case to deduce that the new map $\tilde{f}\equiv Q\a{0}$, that is $g_k \equiv Q_k \a{0}$ in $B_r$. This implies that $f= \sum_k Q_k \a{h_k}$ within $B_r$ and $h_k= p_k$ on $B_r\cap (K-x_0)$. Hence $h_k \equiv p_k$ in $B_r$ by unique continuation for harmonic functions. Thus $B_r(x_0) \subset K$ and the claim follows. 

\medskip

\emph{Reduction to (UC):} Assume by contradiction that the statement does not hold, i.e. after rescaling there is a sequence $w_k \in W^{1,2}(B_2, \Iq(\R^n))$ satisfying the assumptions a1)-a2)-a3), but for some $B_{s_k}(q_k) \in B_2, s_k \ge \eta$ they satisfy
\[ \int_{B_{s_k}(q_k)} |D\overset{\circ}{w}_k|^2 \le \frac1k\, \qquad\text{ or }\qquad \int_{B_{s_k}(q_k)} |\overset{\circ}{w_k}|^2\leq \frac1k\,.\]
Passing if necessary to $w_k\ominus( \eta\circ w_k(0))$ we may assume that $\int_{B_2} \eta\circ w_k =0$\,.

By Theorem \ref{thm:compactness} we can find a convergent sub-sequence related to $(w_k)_k$ such that \[ \chi_k\circ w_k \to w \]
strongly in $W^{1,2}(B_r)$ for every $r<1$ to a continuous classical stationary solution $w$.
Note that in the second case by Poincar\'e inequality we have 
\[ \int_{B_2}|w_k|^2= \int_{B_2} |\overset{\circ}{w_k}|^2 + Q |\eta\circ w_k|^2 \le C(s) \int_{B_2} |D\overset{\circ}{w_k}|^2 + Q |D\eta\circ w_k|^2 + \int_{B_s} |\overset{\circ}{w_k}|^2 \,. \]
so that we can take $\chi_k = \operatorname{Id}$ in the second case. 

Up to passing to a further sub-sequence we can assume that $s_k \to s$ and $q_k \to q$. Hence in both cases we deduce
\[\int_{B_s(q)} |D\overset{\circ}w|^2 = 0\,. \]
Since $\overset{\circ}w$ is itself a continuous classical solution we can apply the just established unique continuation property to deduce that $\overset{\circ}w \equiv Q\a{0}$, but this contradicts the assumptions 
\begin{align*}
    0=\int_{B_{(1+\lambda)}} |D\overset{\circ}{w}|^2= \lim_{k} \int_{B_{(1+\lambda)}} |D\overset{\circ}{w}_k|^2 \overset{a2}{\ge} c\; \lim_{k} \int_{B_{2r}} |Dw_k|^2 \overset{a1}{\ge} c^2>0\,.
\end{align*}

\end{proof}

In particular, Lemma \ref{lem.Lemma7.1.replacement} is applied in \cite{DLS_Center} with assumption a2) replaced by \eqref{eq.no-decay} below.

\begin{lemma}\label{lem.no-decayconsequence}
Let $\delta_2>0$ be such that $2^{\delta_2}-(1+\lambda)^{m+2}>c_2>0$. Let $w \in W^{1,2}(B_{2r}, \Iqn)$ be a classical solution satisfying \begin{equation}\label{eq.no-decay}
    \int_{B_{(1+\lambda)r}} \G(Dw, Q\a{D(\eta\circ w)(0))})^2\geq 2^{\delta_2-m-2}\,\int_{B_{2r}}|Dw|^2\,. 
\end{equation}
then \ref{ass:uniquecontinuation-a2} holds.
\end{lemma}
\begin{proof}
 The claim follows from the decay of harmonic functions: 
Using scaling we may assume $r=1$. Furthermore we have writing $\bar{w}=\eta \circ w$
\begin{align*}
    \int_{B_2} |Dw|^2 &= \int_{B_2} \left(|D\overset{\circ}{w}|^2 + Q |D\bar{w}|^2\right)\\
    \int_{B_\theta} \G(Dw,Q\a{ D\bar{w}(0)})^2&= \int_{B_\theta} \left(|D\overset{\circ}{w}|^2 + Q |D\bar{w}- D\bar{w}(0)|^2\right)\,.
\end{align*}
The decay of harmonic functions implies 
\[ \int_{B_\theta}|D\bar{w}- D\bar{w}(0)|^2 \le \left(\frac{\theta}{2}\right)^{2+m}  \int_{B_1}|D\bar{w}- D\bar{w}(0)|^2 \le \theta^2 \int_{B_1}|D\bar{w}|^2\,.\]
Using the above expansion in \eqref{eq.no-decay} one has 
\begin{align*}
    \int_{B_{\theta}} |D\overset{\circ}{w}|^2 \ge c \int_{B_2} |D\overset{\circ}{w}|^2 + Q\left(c-\left(\frac{\theta}{2}\right)^{2+m}\right) \int_{B_2} |D\bar{w}|^2 \\
    \ge \left(c-\left(\frac{\theta}{2}\right)^{2+m}\right) \int_{B_2} |Dw|^2\,.
\end{align*}
For the choice $c= 2^{\delta_2-m-2}$, $\theta = (1+\lambda)$ we have 
\[ c- \left(\frac{\theta}{2}\right)^{2+m} = 2^{-m-2}(2^{\delta_2}- (1+\lambda)^{m+2})\ge c_2\,.\]
Hence the claim follows. 
\end{proof}

\subsection{Splitting and comparing center manifolds} As a consequence of the above unique continuation lemmas we have that \cite[Propositions 3.5, 3.4 $\&$ 3.7]{DLS_Center} hold with the exact same proofs replacing \cite[Lemma 7.1 \& Proposition 7.2]{DLS_Center} with Lemmas \ref{lem.no-decayconsequence} and \ref{lem.Lemma7.1.replacement}. We remark here that \cite[Lemma 7.1 \& Proposition 7.2]{DLS_Center} only need to be applied in cubes of type excess (i.e., cubes of type $\mathcal W_e$ following their notation) and their domain of influence, where assumption \ref{eq.no-decay} holds by Step I in the proof of \cite[Proposition 3.4]{DLS_Center}. 
Moreover, using Proposition \ref{prop:basicpersistency} in place of \cite[Theorem 2.7]{DS1}, we also have \cite[Proposition 3.6]{DLS_Center} with $\bar{s},\eta_2$ fixed positive constant.

\section{Blow-Up argument and proof of Theorem \ref{thm:main}}\label{ss:BU}
We follow a similar procedure as in \cite{DLS_Blowup}.
The main changes are in the capacity argument since our final blow-up might not be continuous as we don't pass higher integrability to the limit. 

\subsection{Contradiction sequence and its properties} 

We proceed by contradiction, so we assume the following

\begin{ipotesi}[Contradiction]
    There exist integers $m>2$, $n$ and a Lipschitz function $f\colon \Omega\subset \R^2 \to \I2(\R^n)$ with stationary graph, such that $\Ha^{m-1+\alpha}(\Sing(f)\cap \Omega)>0$ for some $\alpha>0$. 
\end{ipotesi}

 From Theorem \ref{thm:varstrat}, and reasoning as in \cite[Section 6]{DLS_Blowup} with $T=\bG_f$, we have the following:

\begin{proposition}[Final blow-up, cp. {\cite[Theorem 6.2]{DLS_Blowup}}]\label{prop.finalBlow-up}
    There is a sequence of maps $N^b_k$ which converge strongly in $L^2(B_{3/2})$ to a function $N^b_\infty \colon B_{3/2}\to \I{2}(\R^n)$ such that $\|N^b_\infty\|_{L^2(B_{3/2})}=1$, $\eta\circ N_\infty^b\equiv 0$.
\end{proposition}

 At difference from \cite{DLS_Blowup} we do not have strong $W^{1,2}$ convergence and energy minimality for the limit $N^b_\infty$, however we can prove the following:

\begin{lemma}\label{lem.StationaryLimit}
    Let $N_k, N_\infty$ be as in the previous proposition, then the limit $\mathcal{F}=\lim_{k\to \infty} \mathcal{E}_{N_k}$ is an inner and outer measure solution in $B_1$.
\end{lemma}

\begin{proof}
    To the blow-up sequence of Proposition \ref{prop.finalBlow-up} we have the associated currents $\bar{T}_k=(\iota_{0,\bar{r}_k})_\sharp T_{j(k)}$, center manifolds $ \bar{\mathcal{M}}_k=(\iota_{0,\bar{r}_k})_\sharp\mathcal{M}_{j(k)}$ and normal approximations $\bar{N}_k(p)=\bar{r}_k^{-1} N_{j(k)}(\bar{r}_kp)$. Furthermore we denote with $\textbf{e}_k$ the exponential map associated to $\bar{\mathcal{M}}_k$ and the blow-up sequence ${N}_k^b= \textbf{h}_k^{-1} \bar{N}_k\circ \textbf{e}_k$. For the relevant definitions and notations we refer the reader to \cite{DLS_Blowup}.

We want to show that for every admissible $\phi, \varphi$ we have
\begin{align}\label{eq.StationaryLimit01}
    \mathcal{O}(\mathcal{E}_{{N}_k^b}, \varphi) = o(1)\\
    \mathcal{I}(\mathcal{E}_{{N}_k^b}, \phi) = o(1)\label{eq.StationaryLimit02}
\end{align}
We start with \eqref{eq.StationaryLimit01}.
Applying \cite[theorem 4.2]{DLS_Currents} to the couple $\bar{\mathcal{M}}_k, \bar{F}_k(x)=\sum_{i=1}^Q \a{x+(\bar{N}_k)_i(x)}$ we obtain for $\varphi(p)=\tilde{\varphi}\circ \textbf{e}_k^{-1}, \tilde{\varphi}\in C^\infty_c(B_{\frac32})$ that 
\begin{align}\label{eq.StationaryLimit03}
    \underbrace{\int_{\bar{\mathcal{M}}_k} \left(\varphi |D\bar{N}_k|^2  + \sum_i ((\bar{N}_k)_i\otimes D\varphi : D(\bar{N}_k)_i) \right)}_{=:O_k}= \delta\textbf{T}_{\bar{F}_k}(X) + \text{Err}_1 - \text{Err}_2 - \text{Err}_3\,.
\end{align}
First we note that appealing to \cite[Lemma 6.1 (ii)-(iii)]{DLS_Blowup} by a change of variable with $\textbf{e}_k$ 
the l.h.s satisfies 
\begin{align*}
   O_k=\textbf{h}_k^{2} \int\left( \tilde{\varphi} |DN_k^b|^2 +\sum_i ((N^b_k)_i\otimes D\varphi : D(N^b_k)_i) \right) + \text{Err}_0\,.
\end{align*}
with 
\[ |\text{Err}_0|\le C \norm{\textbf{e}_k-\operatorname{id}}_{C^1} \textbf{h}_k^2 \int_{B_{\frac32}} |DN_k^b|^2 + |DN_k^b||N_k^b| = o(1) \textbf{h}_k^2\,. \]
It remains to show that the r.h.s of \eqref{eq.StationaryLimit03} is of order $o(1)\textbf{h}_k^2$\,. To estimate these errors we can appeal to \cite[formulas (7.1)-(7.3)]{DLS_Blowup}: 
\begin{align*}
    |\delta\textbf{T}_{\bar{F}_k}(X)|&=|\delta\textbf{T}_{\bar{F}_k}(X)-\delta\bar{T}_k(X)|\le \norm{X}_{C^1} \textbf{M}((\textbf{T}_{\bar{F}_k} - \bar{T}_k)\res \textbf{p}_k^{-1}(\mathcal{B}_{3/2})) \le C \textbf{h}_k^{2+2\gamma}\\
    |\text{Err}_1|&\le C \int |\varphi||H_{\bar{\mathcal{M}_k}}| |\eta\circ \bar{N}_k|\le C \norm{H_{\bar{\mathcal{M}}_k}}_\infty \int_{\mathcal{B}_{3/2}} |\eta\circ \bar{N}_k| \le o(1) \textbf{h}_k^2\\
    |\text{Err}_2|&\le C \int |\varphi||A_{\bar{\mathcal{M}}_k}|^2|\bar{N}_k|^2 \le o(1) \textbf{h}_k^2\\
    |\text{Err}_2|&\le C(1+\norm{A_{\bar{\mathcal{M}}_k}}_\infty)\norm{\varphi}_{C^1} \Lip(\bar{N}_k) \int_{\mathcal{B}_{3/2}} |\bar{N}_k|^2+|D\bar{N}_k||\bar{N}_k| + |D\bar{N}_k|^2=o(1) \textbf{h}_k^2\,.
    \end{align*}

In the same way, applying \cite[Theorem 4.3]{DLS_Currents} to the couple $\bar{\mathcal{M}}_k, \bar{F}_k(x)$, with $\phi(p)=(\textbf{e}_k)_\sharp(\tilde{\phi}\circ \textbf{e}_k^{-1})$, we obtain
\begin{align*}
   \underbrace{ \int_{\bar{\mathcal M}_k} \left(\frac{|D\bar{N_k}|^2}{2}\, \Div_{\bar{\mathcal M}_k}{\phi}-\sum_i D({\bar{N}_k})_i\colon (D(\bar{N}_k)_i\cdot D_{\bar{\mathcal M}_k}\phi)\right)}_{=:I_k}=\delta {\bf T}_{\bar{F}_k}(\phi)+\sum_{i=1}^3\text{Err}_i\,.
\end{align*}
We can argue essential as before setting $\text{Err}_0=I-\textbf{h}_k^2 \int \left(\frac{DN_k^b}{2} \Div(\tilde{\phi}) - \sum_i (D(N^b_k)_i\cdot D\tilde{\phi})\right)$
\begin{align*}
    |\text{Err}_0| & \le C\norm{\textbf{e}_k-\operatorname{id}}_{C^2} \int_{B_{3/2}} |DN_k^b|^2=o(1)\textbf{h}_k^2\\
    |\delta\textbf{T}_{\bar{F}_k}(Y)|&\le C \norm{Y}_{C^1}\textbf{h}_k^{2+2\gamma}\\
    |\text{Err}_1| &\le C\norm{\phi}_{C^1} \norm{H_{\bar{\mathcal{M}}_k}}_\infty \int_{\mathcal{B}_{3/2}} |\eta\circ \bar{N}_k| \le o(1) \textbf{h}_k^2\\
    |\text{Err}_2| &\le C\norm{\phi}_{C^1} \norm{A_{\bar{\mathcal{M}}_k}}_\infty^2 \int_{\mathcal{B}_{3/2}} |\bar{N}_k|^2+ |D\bar{N}_k|^2 \le o(1) \textbf{h}_k^2\\
    |\text{Err}_3| &\le C\norm{\phi}_{C^1} (1+\norm{A_{\bar{\mathcal{M}}_k}}_\infty)\Lip(\bar{N}_k) \int_{\mathcal{B}_{3/2}} |\bar{N}_k|^2+|D\bar{N}_k||\bar{N}_k| + |D\bar{N}_k|^2=o(1) \textbf{h}_k^2\,.
\end{align*}
Hence we can appeal to the continuity to the functionals $\mathcal{F} \mapsto \mathcal{O}(\mathcal{F}, \varphi), \mathcal{I}(\mathcal{F},\phi)$ along $\Iq$ generalized gradient Young measures to deduce the lemma. 
\end{proof}

We will also need the following two properties, which follows from the construction of the center manifold and normal approximation. Before doing that we introduce the following notations:
\begin{align*}
    D_Q(\overline{T}_k)&=\{ p \in \bC_{3/2} \colon \Theta(\overline{T}_k,p)=Q \}\\
    \underline{D}_Q(\overline{T}_k)&=\textbf{e}_k^{-1}(\textbf{p}(D_Q(\overline{T}_k)))\,.\\
    \overline{\ell_k}&=\bar{r}_k^{-1}\sup\left\{ \ell(L) \colon L  \in \mathscr{W}^{(j(k))}_e \text{ and } L \cap B_{19\bar{r}_k/16}(0,\pi) \neq \emptyset \right\},
\end{align*}
where $\bar{T}_k, r_k, \mathcal W_e^{j(k)}$ are as in the proof of the previous lemma.

\begin{lemma}[Weak Hardt-Simon estimate] Let $(N_k^b)_k$ be as in Proposition \ref{prop.finalBlow-up}. There are universal constants $0<s<1$ and $C>0$ such that whenever $x \in \underline{D}_Q(\overline{T}_k)\setminus \{x \colon |N_k^b(x)|=0\} $ there is a radius $0<\rho=\rho(\overline{T}_k,x)<\bar{l}_k\,r_k$ such that 
\begin{equation}\label{eq.weakHardtSimonNew}
    \fint_{B_{s\rho}(x)} |N_k^b \ominus (\eta\circ N_k^b)|^2 \le C \rho^2 \fint_{B_\rho(x)} |DN_k^b|^2 \,.
\end{equation}
Moreover we have that 
    \begin{equation}\label{eq.stoppingradius}
    \overline{\ell_k}\leq o(1)\,.
\end{equation}
\end{lemma}

\begin{proof} We will show that \eqref{eq.weakHardtSimonNew} is an immediate consequence of \cite[Proposition 3.6 (Persistence of $Q$-points)]{DLS_Center}. In fact we may fix a $k$ in our sequence and consider the associated current $\overline{T}_k$ with associated normal approximation $N_k^b$. To make the notation simpler, we drop the subindex $k$ for this argument. If $x \in \underline{D}_Q(\overline{T})\setminus \{ x \colon |N^b|(x)=0\}$, then there is $q\in D_Q(\overline{T})$ such that $\textbf{p}(q)=\textbf{e}(x)$ and $\textbf{e}(x)\notin \Gamma$.  This means that $\textbf{p}_{\pi_0}(q)\in J$, for some $J \in \mathscr{W}$. 
Due to \cite[Proposition 3.1 (Separation)]{DLS_Center}, $J \notin \mathscr{W}_h$. Hence we must have $J$ is either an ``excess'' stopping cube or a ``neigbhouring'' stopping cube, and so in the domain of influence of an excess cube, so that we must have that $\textbf{p}_{\pi_0}(q)$ is in the domain of influence of some cube $L \in \mathscr{W}_e$. Hence we have $\dist(\textbf{p}_{\pi_0}(\textbf{p}(q)), L) \le 4 \sqrt{m} \ell(L)$. It follows that the assumptions of \cite[Proposition 3.6 (Persistence of $Q$-points)]{DLS_Center} are satisfied for the choice of $\eta_2=1$: 
\[ \fint_{\mathcal{B}_{\overline{s}\ell(L)}(\textbf{p}(q)))}|N\ominus (\eta \circ N)|^2 \le \ell(L)^{2-m} \int_{\mathcal{B}_{\ell(L)}(\textbf{p}(q)))}|DN|^2\,.\]
Since $\norm{\textbf{e}(x)-x}_{W^{1,\infty}(B_{3/2})}=o(1)$, we conclude for $\rho = r^{-1} \ell(L)\leq \bar{l}_k\,r_k$ that 
\[ \fint_{B_{\overline{s}\rho/2}(x)}|N^b\ominus \eta \circ N^b|^2 \le \rho^{2-m} \int_{B_{2\rho}(x)}|DN^b|^2\,.\]

    Finally, \eqref{eq.stoppingradius} follows from the previous section, in particular the validity of \cite[Proposition 3.5 (3.3)]{DLS_Center} which implies that \cite[Formula (6.8)]{DLS_Blowup} holds, which is precisely \eqref{eq.stoppingradius}.
\end{proof}

\subsection{Capacity argument}
We give a modified proof for the capacity argument presented in \cite{DLS_Blowup}. Our proof doesn't rely on the specific structure of the problem, but only on the properties of precise representatives in Sobolev spaces and the capacity version of the maximal function estimate. 

We consider a blow-up sequence introduced in Proposition \ref{prop.finalBlow-up}: $(N^b_k)_k$ with strong $L^2$ limit $N=N^b_\infty$ on the ball $\Omega:=B_{3/2}$. To make the notation less cumbersome we drop the superindex $b$. Our aim is to proof the following two statement:

\begin{proposition}[Capacity argument]\label{lem.Capacity}
    Suppose that there are $\delta>0, \eta>0$ such that 
\begin{equation}\label{eq.largeCollapsedSet}
    \mathcal{H}_\infty^{m-2+\delta}(D_Q(\overline{T}_k))\ge \eta \qquad \forall k\,,
\end{equation}
then there is a subsequence of the blow-up sequence $(N_k)_k$ and a set $F_\infty \subset \Omega$ with $\mathcal{H}^{m-2+\delta}(F_\infty)=0$ such that 
\begin{equation}\label{eq.largeFinalCollapsedSet}
    |N(x)| =0 \qquad \forall x \in \left(\limsup_{k} \underline{D}_Q(\overline{T}_k)\right) \setminus F_\infty. 
\end{equation}
\end{proposition}

\begin{proof} Let us recall the definition of the not centered maximal function
\[ M(h)(x) = \sup_{x \in B_r} \fint |h|\,.\]
We note that by the above definition $M(h)$ is lower-semicontinuous, hence the upper-level sets are open. We will divide the proof into the following steps.

%\medskip
%\emph{Step 2: Weak Hardt-Simon estimate}
%There are universal constants $0<s<1$ and $C>0$ such that whenever $x \in \underline{D}_Q(\overline{T}_k)\setminus \{x \colon |N_k(x)|=0\} $ there is a radius $0<\rho=\rho(\overline{T}_k,x)$ s.t. 
%\begin{equation}\label{eq.weakHardtSimonNew}
 %   \fint_{B_{s\rho}(x)} |N_k \ominus \eta\circ N_k|^2 \le C \rho^2 \fint_{B_\rho(x)} |DN_k|^2 \,.
%\end{equation}
% For every $L \in \mathcal{W}^e \cup \mathcal{W}^n$ such that $T$ contains a $Q$-point, meaning there is $p$ such that $\Theta(T,p)=Q$ with $\bold{p}_{\pi_0}(\bold{p}(p))\in L$. Then there is $\bar{s}>0$ such that
% \begin{equation}\label{eq.weakHardtSimon}
%     {\ell(L)}^{-2}\fint_{\mathcal{B}_{\bar{s}\ell(L)}(\bold{p}(p))} \G(N, Q\a{\eta \circ N})^2 \le \fint_{\mathcal{B}_{\ell(L)}(\bold{p}(p))} |DN|^2\,.
% \end{equation}

% \medskip
% \emph{Step 2: ``Stopping radius comparison''}
% There is a sub-sequence such that if 
% \begin{equation}\label{eq.stoppingradius}
%     \overline{\ell_k}=\bar{r}_k^{-1}\sup\left\{ \ell(L) \colon L  \in \mathscr{W}^{(j(k))}_e \text{ and } L \cap B_{19\bar{r}_k/16}(0,\pi) \neq \emptyset \right\} \le 2^{-.....}
% \end{equation}

\medskip

\emph{Step 1: Precise representative approximation estimate.}
 There exist sequences $N_{\epsilon_k} \in \Lip(\Omega, \A_Q)$ and $F_{\epsilon_k}\subset B_{3/2}$ open and nested non increasing with the property that 
\begin{equation}\label{eq.capacity1}
    \operatorname{Cap}_2(F_{\epsilon_k})< 2^{-k} \qquad\text{and}\qquad M(\G(N, N_{\epsilon_l}))(x)< 2^{-k} \quad \forall x \notin F_{\epsilon_k}, l\ge k\,.
\end{equation}
In particular, for any sequence $(x_n) \subset \Omega \setminus F_k$ with $x_n \to x$ and $r_n \downarrow 0$, we have 
\begin{equation}\label{eq.almostuniformconvergence1}
    \lim_{n \to \infty} \fint_{B_{r_n}(x_n)} \G(N,N(x)) =0\,.
\end{equation}

\medskip

\emph{Step 2: good subsequence.}
For $q=2-\frac{\delta}{2}$ \footnote{In fact for every $q<2$ one can find such a subsequence} there exists a further subsequence $N_k$ (not relabeld) and $\hat{F}_k\subset \Omega$ open and nested non increasing such that 
\begin{equation}\label{eq.capacity2}
    \operatorname{Cap}_q(\hat{F}_{k})< 2^{-k} \qquad \text{and}\qquad M(\G(N, N_{l}))(x)< 2^{-k} \quad \forall x \notin \hat{F}_{k}, l\ge k\,.
\end{equation}

\medskip

\emph{Step 3: quantitative Hausdorff estimate.}
There is a further sub-sequence and open and nested non increasing subsets $\tilde{F}_k \subset \Omega$ s.t. 
\begin{equation}
    \mathcal{H}^{m-2+\delta}(\tilde{F}_k)< 2^{-k} \text{ and } r^{2-m-\delta/2} \int_{\mathcal{B}_r(x)} |DN_l|^2 \le 1 \quad \forall x \notin \tilde{F}_k, r\le \bar{\ell_k} \text{ and } l \ge k\,.
\end{equation}

\medskip 

\emph{Step 4: Conclusion assuming Steps 1, 2 and 3.}
Let us first show how to conclude the lemma assuming that we have found the subsequence $N_k$, not relabeled, and the sets $F_k,\hat{F}_k,\tilde{F}_k$. First we note that $\operatorname{Cap}_q(F_k \cup \hat{F}_k) \le 2^{1-k+1}$ hence we deduce that for $F_\infty^1= \bigcap_{k} (F_k\cup \hat{F}_k)$ we have $\operatorname{Cap}_q(F_\infty^1)=0$ and since $m-2+\delta > m-q$ we deduce that $\mathcal{H}^{m-2+\delta}(F_\infty^1)=0$. 
Since $\mathcal{H}^{m-2 + \delta}(\tilde{F}_k)< 2^{-k}$ we deduce that for $F_\infty^2= \bigcap_k \tilde{F}_k$ we have $\mathcal{H}^{m-2+\delta}(F_\infty^2)=0$ so we conclude 
\[ \mathcal{H}^{m-2+\delta}(F_\infty)=0 \text{ for } F_\infty = F_\infty^1\cup F_\infty^2\,.\]
Since both sequence $F_k\cup \hat{F}_k$ and $\tilde{F}_k$ are nested it is sufficient to show that for any $k_0>0$ we have the slightly weaker version of \eqref{eq.largeFinalCollapsedSet}:
\begin{equation}\label{eq.largeFinalCollapsedSet2}
    |N(x)| =0 \qquad \forall x \in \left(\limsup_{k} \underline{D}_Q(\overline{T}_k)\right) \setminus (F_{\epsilon_{k_0}}\cup \hat{F}_{k_0} \cup \tilde{F}_{k_0}). 
\end{equation}

For any $x \in \left(\limsup_{k} \underline{D}_Q(\overline{T}_k)\right) \setminus (F_{\epsilon_{k_0}}\cup \hat{F}_{k_0} \cup \tilde{F}_{k_0})$ there is a sequence $x_n \in \underline{D}_Q(\overline{T}_{k_n}) \setminus (F_{\epsilon_{k_0}}\cup \hat{F}_{k_0} \cup \tilde{F}_{k_0})$ such that $x_n \to x$. We claim that we can find radii $0<r_n<\overline{\ell}_n$ converging to $0$ such that 
\begin{equation}\label{eq.goodradii}
  \fint_{B_{r_n}(x_n)} |N_{k_n}\ominus (\eta\circ N_{k_n})|^2 \le \frac1n \,.
\end{equation}
Having found this sequence will provide the claim due to \eqref{eq.almostuniformconvergence1} and 
\begin{align*}
    \fint_{B_{r_n}(x_n)} |N|^2 &=  \fint_{B_{r_n}(x_n)} |N\ominus (\eta \circ N)|^2 \le C  \fint_{B_{r_n}(x_n)} \cG(N,N_{k_n})^2 + C \fint_{B_{r_n}(x_n)} |N_{k_n}\ominus (\eta\circ N_{k_n})|^2 \\
    &\le C 2^{-k_n} + C/n\,,
\end{align*}
where we have used in the first equality that $\eta \circ N\equiv0$ and $x_n \in \Omega \setminus \tilde{F}_{k_0}\subset \Omega \setminus \tilde{F}_{k_n}$  for $n$ sufficient large in the last line.

Let us now argue for \eqref{eq.goodradii}. Here we have to distinguish two cases. If $x_n$ in the ``collapsed set'', that is $N_{k(n)}(x_n)=0$, then %$x_n \in r_{k_n}^{-1}\textbf{e}_{k_n}^{-1}\left(\Phi_{k_n}(\Gamma_{k_n}) 
%\right)$ 
by continuity of each $N_{k_n}$ we can fix a radius $0<r_n< \overline{\ell}_{k_n}$ such that 
\[ \fint_{B_{r_n}(x_n)} |N_{k_n}\ominus (\eta\circ N_{k_n})|^2 \le \frac1n\,. \]
If $x_n$ not in the above we are precisely in the situation of step 2 and hence \eqref{eq.weakHardtSimonNew} applies. Setting $r_n= \bar{s}\rho(x_n)=\bar{s}\rho_n$ we deduce that 
\begin{align*}
   \fint_{B_{r_n}(x_n)} |N_k \ominus (\eta\circ N_k)|^2 \le C \rho_n^2 \fint_{B_{\rho_n}(x_n)} |DN_k|^2 \le C \rho_n^{\delta/2} \le \frac1n\,,
\end{align*}
where we have used in second to last inequality that $x_n\notin \tilde{F}_{k_0}$ and $\rho_n \downarrow 0$.

%{\color{red} What is this for??
%For any $s,\delta>0$ and any $w \in L^1(\Omega)$ one has for any $\rho,\lambda>0$
%\begin{equation}\label{eq.capacity3}
    %\mathcal{H}^{s+\delta}(\{ x \colon \exists \, 0<r< \rho, |y-x|\le r \text{ s.t. } r^{-s}\int_{B_r(y)} |w|> \lambda \}) \le C \rho^\delta \int_\Omega |w|
%\end{equation}}

\emph{Proof of Step 1: } It is essentially a classical argument leading to the fine properties of Sobolev functions, compare \cite[chapter 4.8]{EG}. For the sake of completeness we will present the argument. It could be stated as follows: 

\emph{Step 1 - reformulated: }
For any $g \in W^{1,2}(\Omega, \Iqn)$ there exist sequences $g_{\epsilon_k} \in \Lip(\Omega)$ and $F_{\epsilon_k}\subset \Omega$ open with the property that 
\begin{equation}\label{eq.capacity1}
    \operatorname{Cap}_2(F_{\epsilon_k})< 2^{-k} \qquad\text{and}\qquad M(\G(g, g_{\epsilon_l}))(x)< 2^{-k} \quad \forall x \notin F_{\epsilon_k}, l\ge k\,.
\end{equation}
In particular we have for any sequence $(x_n) \subset \Omega \setminus F_k$ with $x_n \to x$ and $r_n \downarrow 0$ we have 
\begin{equation}\label{eq.almostuniformconvergence}
    \lim_{n \to \infty} \fint_{B_{r_n}(x_n)} \G(g,g(x)) =0\,.
\end{equation}

We fix a sequence of Lipschtiz approximations $g_{\epsilon} \in W^{1,2}(\Omega, \Iqn)$ that converge strongly to $g$ in $W^{1,2}$ as $\eps\to0$. The classical capacity estimate, \cite[Theorem 4.18]{EG}\footnote{The theorem there is stated only on full space. In fact we want to use the following version: 
\[ \operatorname{Cap}_p(\{M(f)>\lambda\}) \le \frac{C}{\lambda^p} \int_{\{ f \ge \frac{\lambda}{2}\}\cap \Omega} |Df|^p\,\text{ for any } f \in W^{1,p}(\Omega)\,. \]
This can be seen as follows. Given any nonnegative $f \in W^{1,p}(\Omega)$ we set $f_1= (f-\frac{\lambda}{2})^+ \in W^{1,p}(\Omega), f_2 = f-f_1 \in W^{1,p}(\Omega)$. Note that both are non-negative and $f=f_1+f_2$. We may choose an extension of $f_1$ %and $f_2$
to the full space, since $\Omega$ is a Lipschitz domain. %If necessary we may take $f_2=\min\{f_2,\frac{\lambda}{2}\}$ to ensure that $f_2\le \frac{\lambda}{2}$. 
It is straightforward to see that $\{M(f)>\lambda\} \subset \{M(f_1)>\frac{\lambda}{2}\}$ hence the full space estimate \cite[Theorem 4.18]{EG} provides
\[\operatorname{Cap}_p(\{M(f)>\lambda\})\le \operatorname{Cap}_p\left(\left\{M(f_1)>\frac{\lambda}{2}\right\}\right) \le 2^p\frac{C}{\lambda^p} \int_{\R^n\cap\{f_1>\frac\lambda2\}} |Df_1|^p \le \frac{C}{\lambda^p} \int_{\{f>\frac{\lambda}{2}\}\cap \Omega} |Df|^p\,.\]}
states 
\[ \operatorname{Cap}_2(E^\epsilon_k) \le C 2^{2k} \int_{\Omega} |D\G(g,g_\epsilon)|^2\,,\qquad  \text{ for } E^\epsilon_k=\{ M(\G(g, g_{\epsilon}))>0\}\,.\]
Hence we may choose $\epsilon_k$ sufficient small such that the left hand side is less than $2^{-k-4}$ i.e. $\operatorname{Cap}_2(E^{\epsilon_k}_k) \le 2^{-k-4}$.
Hence we may take $F_k= \bigcup_{l\ge g} E^{\epsilon_k}_k$. As pointed out \cite[Theorem 4.19]{EG} this choice implies uniform strong convergence on $F_k^c$ for each $k$:
\[ \G(g_{\epsilon_i}(x), g_{\epsilon_j}(x)) \le \lim_{r\to 0} \fint_{B_r(x)} \G(g_{\epsilon_i}, g_{\epsilon_j}) \le 2 (2^{-i} + 2^{-j}) \qquad \forall x \notin F_k, i,j> k\,.  \]
In particular this implies that $g$ is continuous on $\Omega\setminus F_k$. Furthermore it implies \eqref{eq.almostuniformconvergence} since for any $l >k$ we have 
\begin{align*}
    \fint_{B_{r_n}(x_n)} \G(g,g(x)) 
        &\le \fint_{B_{r_n}(x_n)} \left(\G(g, g_{\epsilon_l}) + \G(g_{\epsilon_l}, g_{\epsilon_l}(x)) +  \G(g_{\epsilon_l}(x), g(x)) \right)\\ &\le 2^{-l} + \fint_{B_{r_n}(x_n)} \G(g_{\epsilon_l}, g_{\epsilon_l}(x)) + 2^{-l}\,.
\end{align*}
Taking the limit $n \to \infty$ gives the claim since we showed that \[\lim_{n \to \infty} \fint_{B_{r_n}(x_n)} \G(g,g(x)) \le 2^{-l+1}\,.\]

We obtain the result claimed in the original version for the choice $g=N$.

\smallskip
\emph{Proof of Step 2: }
As in Step 1 there is a abstract capacitary argument behind. This time it could be phrased as follows:

\emph{Step 2 - reformulated}
Suppose $g_k \in W^{1,2}(\Omega, \Iqn)$ conveges weakly to $g$ then for every $q<2$ there exists a subsequence $g_k$ (not relabeld) and $\hat{F}_k\subset \Omega$ open such that 
\begin{equation}\label{eq.capacity2}
    \operatorname{Cap}_q(\hat{F}_{k})< 2^{-k} \qquad \text{and}\qquad M(\G(g, g_{l}))(x)< 2^{-k} \quad \forall x \notin \hat{F}_{k}, l\ge k\,.
\end{equation}

The argument is as well classical and can be found for instance in \cite[Chapter 1, Theorem 7]{E}. Let us define $E_k^l = \{ M(\G(g, g_l))> 2^{-k}\}, \tilde{E}_k^l = \{ \G(g,g_l)> 2^{-k-1}\}$. The classical capacity estimate provides 
\[ \operatorname{Cap}_q(E^l_k) \le C 2^{qk} \int_{\tilde{E}_k^l} |D\G(g,g_l)|^{q} \le C 2^{qk} |\tilde{E}_k^l|^{1-\frac{q}{2}} \left(\norm{Dg}^q_{L^2(\Omega)}+ \norm{Dg_l}^q_{L^2(\Omega)} \right).\]
Since $|\tilde{E}_k^l|^{\frac12}\le 2^{-l} \norm{\G(g,g_l)}_{L^2(\Omega)} \to 0$ for $l \to \infty$ we may select $l_k$ sufficient large such that the right hand side becomes less than $2^{-k-4}$. We set $\hat{F}_k = \bigcup_{l\ge k} E^{l_k}_k$. Thus the new obtained sequence $g_{l_k}$ has the desired properties.

We obtain the claimed estimate by consider the sequences $N_k^b$.

\smallskip

\emph{Proof of Step 3: } First we need a quantitative version of comparison between Hausdorff measure and integrals over balls, compare \cite[theorem 2.10]{EG}: For any $s,t>0$ and any $w \in L^1(\Omega)$ one has for any $\rho,\lambda>0$
\begin{equation}\label{eq.capacity3}
    \mathcal{H}^{s+t}(A_\lambda^\rho) \le \frac{C}{\lambda} \rho^s \int_\Omega |w|\,,\quad \text{ where } A_\lambda^\rho=\left\{ x \colon \exists  0<r< \rho, |y-x|\le r \text{ s.t. } r^{-t}\int_{B_r(y)} |w|> \lambda \right\}
\end{equation}
The argument is classical. Let $\mathcal{B}=\left( B_r(y) \right)$ be the collection of all balls with $r \le \rho$ such that $r^{-t} \int_{B_r(y)} |w| > \lambda$. We clearly have $A_\lambda^\rho \subset \bigcup_{B \in \mathcal{B}} B$. Appealing to the Vitali covering theorem we can find a disjoint sub-collection $\bigl(B_{r_i}(y_i)\bigr)_i$ where their by factor $5$ increased balls still cover everything. Hence we have 
\[ \lambda \,\mathcal{H}_{\infty}^{s+t}(A_\lambda^\rho) \le C\sum_i \lambda \,r_i^{s+t} \le C \rho^{s} \int |w|\,. \]

Now we may apply the above estimate to $w=|DN_k|^2$, $t=m-2+\frac{\delta}{2}, s=\frac{\delta}{2}$ and $\rho=\overline{\ell}_k, \lambda =1$. Hence we deduce 
\[ \mathcal{H}^{m-2+\delta}(\tilde{E}_k) \le C\, \overline{\ell}_k^{\delta/2}\le 2^{-k-1}\,, \]
for $\tilde{E}_k=\left\{ x \colon \exists  0<r< \overline{\ell}_k, |y-x|\le r \text{ s.t. } r^{2-m-\delta/2}\int_{B_r(y)} |DN_k|^2> 1 \right\}$ and where we have used in the end that for any $C>0$, there is a sub-sequence such that $\overline{\ell_k}\leq \frac{1}{C} 2^{-k}$, which follows from \eqref{eq.stoppingradius} with the same justification as in \cite[formula (6.8)]{DLS_Blowup}.
To obtain a decreasing family we set $\tilde{F}_k= \bigcup_{l \ge k} \tilde{E}_l$.

% Due to step 1 we can pass to further sub-sequence of the $N_l$ such that $s_l^\delta C < 2^{-k-4}$ thus we have for them the desired estimate.  
\end{proof}

\textbf{Final contradiction: }
Combining Theorem \ref{thm:UCformeas}, Proposition \ref{prop.finalBlow-up}, Lemma \ref{lem.StationaryLimit} and Proposition \ref{lem.Capacity}, leads to a contradiction: 
Proposition \ref{lem.Capacity} implies that $\mathcal{H}_\infty^{m-1+\delta}( \{ |N_\infty^b|=0\}\cap B_{3/2})\ge \eta$, Lemma \ref{lem.StationaryLimit} implies that the limit $\mathcal{F}=\lim_{k} \mathcal{E}_{N_k^b}$ is stationary and has $N_\infty^b$ as its associated $W^{1,2}$-function. Hence Theorem \ref{thm:UCformeas} applies and we deduce that $N_\infty^b\equiv 0$. This contradicts Proposition \ref{prop.finalBlow-up}, namely $\int_{B_{3/2}} |N^b_\infty| = 1$.

\appendix
\section{On the notion of stationary graphs}\label{app:statio}

In this section for $Q=4$ we give an example of a stationary graph whose graph is not a stationary varifold. It is a small modification of the example presented in \cite{HS}.  In the second part we show that for $Q=2$ the two notions are in fact equivalent under an additional regularity assumption.

\begin{remark}[A stationary map whose graph is not stationary]
    Consider for any $m>0$ the $2$-valued map
\[g(x)=\begin{cases} \a{m\,x} + \a{-m\,x} &\text{ for } x> 0 \\ 2\a{0} & \text{ for } x< 0 \end{cases}\,.\]
One can directly check that $g$ satisfies the outer variation:% since $g$ is the composition of two geodesics in the region $x>0$ and two geodesics $x<0$. Furthermore it satisfies the outer variation at $x=0$ because 
\begin{align*}
    \int_0^\infty \sum_{l=0}^1 (-1)^l (1+m^2)^{-1/2}m \frac{d}{dx}\psi(x,m\,x) \, dx = - \sum_{l=0}^1 (-1)^l (1+m^2)^{-1/2}m \psi(0,0) =0\,.
\end{align*}
One immediately checks that $\bG_g$ is not a stationary varifold, not even for $m= \sqrt{3}$ which corresponds to an opening angle of $\frac{2\pi}{3}$, since the multiplicities do not match. In particular the inner variation of $g$ is not $0$ in any neighborhood of $0$.  

If for any value $a \in \R$, we consider the $4$-valued map
\[ f_a(x) = (g(x)\oplus a)+ (g(-x) \oplus (-a))\,,\]
then we obtain a stationary map, since $f_a$ is stationary for the outer variation as a result of both $g(x)$ and $g(-x)$ being stationary for the outer variation separately; 
whereas the inner variation is trivially satisfied since
\[\sum_{l=1}^4 \sqrt{g(f_l)}\, g(f_l)^{-1} = \sum_{l=1}^4 \frac{1}{\sqrt{|g(f_l)|}} = \frac{2}{\sqrt{1+m^2}} + 2 \quad \forall x\,.\]

However the graph of $f_a$ is stationary only for $a=0$, that is the choice that corresponds to four crossing lines which are clearly a stationary varifold.\footnote{This observation can be used to see that the inner variation is satisfied: Since the inner variation does not localise it doesn't depend on the choice of $a$, and for $a=0$ the inner variation is satisfied.}
% The inner variation does not depend on the choice of $a$ since by our assumption it is not localizing. But for the specific choice of $a=0$ we consider the situation of four crossing lines which is clearly a stationary varifold. Hence the inner variation is satisfied. 
\end{remark}

The conjecture of Lawson and Osserman states that for a Lipschitz function $f\colon \R^m \to \R^n$ if the outer variation for area is zero then so is the inner. When $n=1$, this is known to be true by De Giorgi-Nash-Moser, while for arbitrary $n$ and $m=2$ the conjecture has recently been confirmed by the first named author together with Mooney and Tione (\cite{HMT}). For general $m$ and $n$ the conjecture is still open. 

In fact we would only need a slightly weaker version of the conjecture: assume that $f_i\colon \R^m\to \R^n$, $i=1,2$ are two Lipschitz functions such that $\mathcal O^A(f_i)\equiv 0$ for $i=1,2$, and moreover, if we set $f:=\sum_{i=1}^2\a{f_i}$, then $\mathcal I^A(f)\equiv 0$, then $\mathcal I^A(f_i)\equiv 0$ for $i=1,2$. 

\begin{lemma}[Stationary map implies stationary graph] Let $f\colon\Omega \to \I{2}(\R^n)$ be a map that is stationary for the area functional, then $\bG_f$ is a stationary varifold if the weaker version of the Lawson and Osserman conjecture holds.
\end{lemma}
\begin{proof}
    We have to show that $\delta \bG_f(Y)=0$ for all Lipschitz continous variations $(x,y) \in \Omega \times \R^n \mapsto Y(x,y) \in \R^m$ with $\spt(Y) \Subset \Omega\times \R^n$.

    Let $\Sigma = \{ x \in \Omega \colon f(x)=2 \a{\eta \circ f} \}$ the closed set where both graphs intersect. Since $f$ is assumed to be Lipschitz the complement $U = \Omega \setminus \Sigma$ is open. Due to the continuity for each $x \in U$ there are two Lipschtiz functions $f_i, i=1,2$ such that $f= \a{f_1} + \a{f_2}$ in $B_r(x)$ for some $r>0$. Furthermore $f_i$ are stationary with respect to the outer variation. Under the assumption that the conjecture holds, each $f_i$ solves the full minimal surface system separately. Hence we deduce that $\delta \bG(Y)=0$ for all $Y$ as above with $\spt(Y) \Subset U \times \R^m$. 

    Now let $Y$ be any admissible variation as above, and let $\eta$ be a smooth, non-increasing  cut-off function with $\eta(t)\equiv 1$ for $t\le 0$ and $\eta(t)=0$ for $t\ge 1$. 

    Note that for all $\epsilon >0$
    \[ Y= \left(1-\eta\left(\frac{\dist(x,\Sigma)}{\epsilon}\right)\right) Y + \eta\left(\frac{\dist(x,\Sigma)}{\epsilon}\right) Y =: \hat{Y}_\epsilon + \tilde{Y}_\epsilon\,. \]
    By the above we have $\delta \bG(\hat{Y}_\epsilon) =0$ for every $\epsilon>0$. Let us define 
    \[\bar{Y}_\epsilon := \eta\left(\frac{\dist(x,\Sigma)}{\epsilon}\right) Y(x, \eta \circ f)\,.\]
    By the assumption that $f$ is stationary for the area we have $\delta \bG(\bar{Y}_\epsilon) =0$ for every $\epsilon >0$. Thus we conclude
    \begin{align*}
        \delta\bG_f(Y)&=\delta\bG_f(\tilde{Y}_\epsilon)= \delta\bG_f(\tilde{Y}_{\epsilon}-\bar{Y}_\epsilon)\\
        &= \int \sum_{l=1}^2 \sqrt{|g(f_l)|} g^{ij}(f_l) \partial_j \left( \eta\left(\frac{\dist(x,\Sigma)}{\epsilon}\right)\left(Y(x,f_l(x)) - Y(x,\eta\circ{f})\right) \right)\\
        &= \underbrace{\int \sqrt{|g(f_l)|} g^{ij}(f_l)\, \eta\left(\frac{\dist(x,\Sigma)}{\epsilon}\right)\,\partial_j\left(Y(x,f_l(x)) - Y(x,\eta\circ{f})\right)  }_{I_\epsilon} \\
        &\quad+ \underbrace{\int \sqrt{|g(f_l)|} g^{ij}(f_l)\, \eta'\left(\frac{\dist(x,\Sigma)}{\epsilon}\right) \, \partial_j \dist(x,\Sigma) \frac{\left(Y(x,f_l(x)) - Y(x,\eta\circ{f})\right)}{\epsilon} }_{II_\epsilon}\,.
    \end{align*}
    $I_\epsilon \to 0$ as $\epsilon \downarrow 0$ since $\eta\left(\frac{\dist(x,\Sigma)}{\epsilon}\right)\sum_l |\partial_j\left(Y(x,f_l(x)) - Y(x,\eta\circ{f})\right) |\to 0$ in measure by Rademacher theorem for $Q$-valued function and the  uniform boundedness due to the Lipschitz continuity of $f$.
    
    $II_\epsilon \to 0$ since 
    \[\sum_l\left|\frac{\left(Y(x,f_l(x)) - Y(x,\eta\circ{f})\right)}{\epsilon}\right| \le \Lip(Y)\Lip(f) \frac{\dist(x,\Sigma)}{\epsilon}\lesssim 1\]    
    on the support of $\eta'\left(\frac{\dist(x,\Sigma)}{\epsilon}\right)$ and $\eta'\left(\frac{\dist(x,\Sigma)}{\epsilon}\right) \to 0$ in measure.  
\end{proof}

\section{$L^\infty$-$L^2$ estimate}

Although the following estimate is known even in the more general context of stationary varifolds (see \cite{All}), we provide here a  proof for Lipschitz multivalued functions which are stationary for area, for the reader's convenience.

\begin{lemma}\label{lem.DeGiorgiClass}
Let $f\colon\Omega\subset \R^m \to \Iq(\R^n)$ be a Lipscitz function which is stationary for area, then for any $e,t\in \R^n$ the function 
\begin{equation*}
    u(x)=\max_{l=1,\dotsc,Q} e\cdot (f_l(x)-t)
\end{equation*}
is in the De Giorgi class $DG(\Omega)$ (see \cite{GiaquintaMartinazzi} for the definition).
\end{lemma}
\begin{proof}
We let $\varphi \in C^1_c(\Omega)$ and we test the outer variation with the vector field $\psi(x,u)= \varphi^2(x) (e\cdot(u-t) -k)^+ e$ to obtain 
\begin{align*}
    0 =& \int \varphi^2\sum_{l=1}^Q  \sqrt{|g(f_l)|}\, g^{ij}(f_l)\, e\cdot\partial_if_l \,e\cdot  \partial_jf_l\, \mathbf{1}_{\{e\cdot(f_l(x)-t)\ge k\}}\\& + 2 \int \varphi \partial_i \varphi \sum_{l=1}^Q \sqrt{|g(f_l)|}\, g^{ij}(f_l)\, e\cdot\partial_if_l \,(e\cdot(f_l-t)-k)^+\\\ge& \frac12 \int \varphi^2\sum_{l=1}^Q  \sqrt{|g(f_l)|}\, g^{ij}(f_l)\, e\cdot\partial_if_l \,e\cdot  \partial_jf_l\, \mathbf{1}_{\{e\cdot(f_l(x)-t)\ge k\}}\\& - 2 \int \partial_j\varphi \partial_i \varphi \sum_{l=1}^Q \sqrt{|g(f_l)|}\, g^{ij}(f_l) \left((e\cdot(f_l-t)-k)^+\right)^2 \,.
\end{align*}
We notice that by \eqref{eq:lip1} and \eqref{eq:lip2} we have
\begin{align*}
    \frac{1}{(1+L^2)^\frac12} |D(u-k)^+|^2 &\le \sum_{l=1}^Q  \sqrt{|g(f_l)|} g^{ij}(f_l)\, e\cdot\partial_if_l \,e\cdot  \partial_jf_l \mathbf{1}_{\{e\cdot(f_l(x)-t)\le k\}}\\
    Q (1+L^2)^{\frac{m-1}{2}} |(u-k)^+|^2&\ge\sum_{l=1}^Q \sqrt{|g(f_l)|} g^{ij}(f_l) \left((e\cdot(f_l-t)-k)^+ \right)^2 
\end{align*}
where we have used that $\Lip(f)\leq L$. 
Combining these estimates and choosing $\varphi=1$ on $B_{\rho}(x_0)$, supported in $B_R(x_0)$ with $|D\varphi|\le \frac{C}{R-\rho}$ for a given $x_0 \in \Omega$ and $0<\rho<R<\dist(x_0,\partial \Omega)$, we obtain
\[\int_{B_\rho(x_0)} |D(u-k)^+|^2 \le C\, \frac{4Q(1+L^2)^{\frac m2}}{(R-\rho)^2} \int_{\Omega} |(u-k)^+|^2 \,. \]
This estimate is precisely the definition of $u$ being in the De Giorgi class $DG(\Omega)$.
\end{proof}

\begin{corollary}\label{cor.LinftyL2}
There is a dimensional constant $C$ depending on $Q$ and $L$ such that if $f\colon \Omega\subset \R^m \to \Iq(\R^n)$ is a Lipscitz function with $\Lip(f)\leq L$ and which is stationary for area, and $B_{2R}(x_0) \subset \Omega$, then for any $t \in \R^n$ one has 
\begin{equation}\label{eq.LinftyL2}
    \sup_{B_R(x_0)} |f\ominus t|^2 \le C \fint_{B_{2R}(x_0)} |f\ominus t|^2 \,. 
\end{equation}
\end{corollary}

\section{Stratification for stationary varifolds}

We recall here the following theorem, essentially due to Allard and Almgren.

\begin{theorem}[Stratification]\label{thm:varstrat} Suppose that there exist integers $m\geq 2$, $n\geq 1$ and a $m$-dimensional integral current $T$ which is stationary in an open set $U\subset \R^{m+n}$ and such that 
\[
\Ha^{m-1+\alpha}(\Sing(T)\cap U) >0\qquad \text{for some }\alpha>0\,,
\]
then there exist $m,n,Q\geq2$, an $m$-dimensional integral current $\bar{T}$ which is stationary in $B_{8}$, and a sequence $r_k \downarrow 0$ such that $0\in D_Q(T)$ and 
\begin{gather}\label{eq.stratifications01}
\lim_{k\to \infty} \bE(T_{0,r_k}, \bB_{6\sqrt{m}})=0\,,\\\label{eq.stratifications02}
\lim_{k\to \infty}\Ha_\infty^{m-1+\alpha}(D_Q(T_{0,r_k})\cap \bB_1)>\eta>0\,,\\\label{eq.stratifications03}
\Ha^m\left(\bB_1\cap \spt (T_{0,r_k})\setminus D_Q(T_{0,r_k}) \right)>0\,.
\end{gather}
 Conclusions \eqref{eq.stratifications01}-\eqref{eq.stratifications03} hold true as well under the assumption that $T=\bG_f$ for a Lipschitz function $f\colon \Omega \to \A_2(\R^n)$ and the regular set $\reg(f)$ contains both points of density $2$ and $1$.  
    
\end{theorem}

\begin{proof}
  In the first case the proof is the same as in \cite[Proposition 8.7 (iii)s]{DHMS}, with the obvious modification.

   The second case, though not explicitly covered in \cite[Proposition 8.7 (iii)s]{DHMS}, is essentially the same. 
Indeed, note that by our assumption the topological boundary \[\partial_*D_2:=\{ x \in \spt T \colon \limsup_{r\downarrow 0 } \frac{ \Ha^m(D_2(T)\cap \bB_r(x))}{r^m}>0,\, \limsup_{r\downarrow 0 } \frac{ \Ha^m(\bB_r(x)\setminus D_2(T))}{r^m}>0\}\]
cannot be empty. 
To see this one may argue by using sets of finite perimeter as follows: To avoid introducing them varifolds, we will argue for the graph situation. Let us denote by $E=\pi_0(D_2(T) \cap \bC_1)$ the projection of all points with density $2$. By assumption we have $\min\{|E\cap B_1|, |B_1\setminus E|\} >0$. Furthermore note that $\pi_0(\partial_*D_2(T)\cap \bC_1))=\partial_*E\cap B_1$. If $\partial_*E$ would be empty in particular it would have finite $\mathcal{H}^{n-1}$ measure, but then Federer's characterisation result, \cite[Section 4.5.11]{Federer} applies and we would deduce that $\mathbf{1}_E$ is a BV-function that is either constant $0$ or constant $1$. 

Hence we may assume that $0 \in \partial_*D_2$. From now on the argument presented in \cite[Proposition 8.7 (iii)s]{DHMS} applies. 
\end{proof}

\bibliographystyle{plain}
\bibliography{stationaryforarea.bib}

\end{document}